\documentclass[a4paper,12pt]{article} 
 
\usepackage[utf8]{inputenc}
\usepackage{comment, amsthm, amsmath, amssymb, bbm, epsfig, enumerate, color, hyperref, nicefrac, mathtools, geometry}

\newcommand{\smallsum}{\textstyle\sum}

\newcommand{\D}{\mathcal D}

\newcommand{\V}{V}
\newcommand{\ga}{\gamma_0}
\newcommand{\gb}{\gamma_1}
\newcommand{\gc}{\gamma_2}

\newcommand{\R}{\mathbb{R}}
\newcommand{\N}{\mathbb{N}}
\newcommand{\Z}{\mathbb{Z}}
\newcommand{\E}{\mathbb{E}}
\renewcommand{\P}{\mathbb{P}}

\newcommand{\tr}{\operatorname{trace}}

\newcommand{\Hess}{\operatorname{Hess}}

\newtheorem{lemma}{Lemma}[section]
\newtheorem{remark}[lemma]{Remark}
\newtheorem{theorem}[lemma]{Theorem}
\newtheorem{definition}[lemma]{Definition}
\newtheorem{prop}[lemma]{Proposition}
\newtheorem{corollary}[lemma]{Corollary}

\newcommand{\U}{U}

\providecommand{\N}{{\ensuremath{\mathbbm{N}}}}
\providecommand{\R}{{\ensuremath{\mathbbm{R}}}}

\providecommand{\E}{{\ensuremath{\mathbb{E}}}}
\renewcommand{\P}{{\ensuremath{\mathbb{P}}}}
\providecommand{\1}{{\ensuremath{\mathbbm{1}}}}
\providecommand{\HS}{{\ensuremath{\textup{HS}}}}
\renewcommand{\H}{{\ensuremath{\mathbb{H}}}}

\begin{document}
\title{Exponential moments for numerical approximations 
of stochastic partial differential equations}
\author{ 
Arnulf Jentzen
and Primo\v{z} Pu\v{s}nik}
\maketitle
\begin{abstract}
Stochastic partial differential equations
(SPDEs)
have become a crucial ingredient
in a number of models from economics and the natural sciences.
Many SPDEs that appear in such applications include non-globally
monotone nonlinearities.
Solutions of SPDEs
with non-globally monotone nonlinearities
are in nearly all cases not known explicitly.
Such SPDEs can thus only be solved approximatively and
it is an important research problem
to construct and analyze discrete numerical approximation schemes which 
converge with positive
strong 
convergence rates
to the solutions of such
infinite dimensional SPDEs.
In the case of finite dimensional stochastic ordinary differential
equations (SODEs) with non-globally monotone nonlinearities
it has recently been revealed 
that exponential integrability
properties of the discrete numerical approximation scheme are
a key instrument to establish positive strong  convergence rates for the considered approximation scheme.
Exponential integrability properties
for appropriate 
approximation schemes have been established
in the literature in the case of a large class
of finite dimensional SODEs with non-globally
monotone nonlinearities.
To the best of our knowledge, there exists no result in the scientific literature which proves exponential integrability properties for a time discrete approximation scheme in the case of an infinite dimensional SPDE.
In particular, to the best of our knowledge,
there exists no result in the scientific literature which establishes strong convergence rates for a time discrete approximation scheme in the case of 
a
SPDE with a non-globally monotone nonlinearity.
In this paper we propose a new class of tamed
space-time-noise discrete
exponential Euler approximation
schemes
that admit exponential integrability properties in the case of infinite dimensional
SPDEs.
More specifically, the main result
of this article  
proves that these approximation schemes
enjoy exponential integrability properties
for a large class of SPDEs with possibly non-globally monotone nonlinearities. 
In particular, we establish exponential moment bounds for the
proposed approximation schemes in the case of
stochastic Burgers equations,
stochastic Kuramoto-Sivashinsky equations,
and two-dimensional stochastic Navier-Stokes equations.
\end{abstract}
\tableofcontents
\newpage
\section{Introduction}
Stochastic partial differential equations
(SPDEs)
have become a crucial ingredient
in a number of models from economics and the natural sciences.
For example, SPDEs frequently appear in models for the approximative
pricing of interest-rate based financial derivatives
(cf., e.g., Theorem~2.5 in Harms et al.~\cite{HarmsStefanovitsTeichmannWutrich2015} 
and (1.2) in Filipovi\'c et al.~\cite{FilipovicTappeTeichmann2010}),
for the approximative description of random surfaces
in surface growth models
(cf., e.g., (1) in Bl\"omker \& Romito~\cite{BlomkerRomito2015}
and (3) in Hairer~\cite{HairerKPZ}),
for describing the temporal dynamics associated to Euclidean quantum field
theories (cf., e.g., (1.1) in Mourrat \& Weber~\cite{MourratWeber2015}),
for the approximative description of velocity fields
in fully developed turbulent flows 
(cf., e.g., (7) in Birnir~\cite{Birnir2013a}
and (1.5) in Birnir~\cite{Birnir2013b}),
and for the approximative description
of the temporal evolution of the concentration of an undesired
(chemical or biological)
contaminant in water
(e.g., in a water basin,
the groundwater system, or a river;
cf., e.g., (1.1) in Kouritzin \& Long~\cite{KouritzinLong2002}
and also (1.1) in Kallianpur \& Xiong~\cite{KallianpurXiong1994}).
Many SPDEs that appear in such applications include non-globally
monotone nonlinearities.
Solutions of SPDEs
with non-globally monotone nonlinearities
are in nearly all cases not known explicitly.
Such SPDEs can thus only be solved approximatively and
it is an important research problem
to construct and analyze discrete numerical approximation schemes which 
converge with positive
strong 
convergence rates
to the solutions of such
infinite dimensional SPDEs.
In the case of finite dimensional stochastic ordinary differential
equations (SODEs) with non-globally monotone nonlinearities
it has recently been revealed
in the literature
that exponential integrability
properties of the discrete numerical approximation scheme are
a key ingredient to establish positive strong convergence rates for the considered approximation scheme;
cf., e.g., 
Hutzenthaler et al.~\cite{HutzenthalerJentzenWang2014},
Hutzenthaler \& Jentzen~\cite{HutzenthalerJentzen2014PerturbationArxiv},
and
Cozma \& Reisinger~\cite{CozmaReisinger2015}.
In particular, e.g., Corollary~3.8 in 
Hutzenthaler et al.~\cite{HutzenthalerJentzenWang2014}
and
Proposition~3.3 in Cozma \& Reisinger~\cite{CozmaReisinger2015} 
(cf. also Lemma~3.6 in Bou-Rabee \& Hairer~\cite{BouRabeeHairer2013})
establish
exponential integrability properties
for appropriate
stopped/tamed/truncated
approximation schemes in the case of a large class
of finite dimensional SODEs with non-globally
monotone nonlinearities.
To the best of our knowledge, there exists no result in the scientific literature which proves exponential integrability properties for a time discrete approximation scheme in the case of an infinite dimensional SPDE.
In particular, to the best of our knowledge,
there exists no result in the scientific literature which establishes strong convergence rates for a time discrete approximation scheme in the case of 
a
SPDE with a non-globally monotone nonlinearity
(cf., e.g., D\"orsek~\cite{Dorsek2012} 
and Hutzenthaler \& Jentzen~\cite{HutzenthalerJentzen2014PerturbationArxiv}).
In this paper we propose a new class of tamed
space-time-noise discrete
 exponential Euler approximation
schemes
that admit exponential integrability properties in the case of infinite dimensional
SPDEs.
More specifically, the main result
of this article 
(see Theorem~\ref{theorem:full_discrete_scheme_moments} 
in Section~\ref{section:3} below)
proves that these approximation schemes
enjoy exponential integrability properties
for a large class of SPDEs with possibly non-globally monotone nonlinearities. 
In particular, we establish exponential moment bounds for the
proposed approximation schemes in the case of
stochastic Burgers equations (see Corollary~\ref{corollary:Burgers}
in Subsection~\ref{sec:Burgers} below),
stochastic Kuramoto-Sivashinsky equations (see Corollary~\ref{corollary:Kuramoto}
in Subsection~\ref{sec:Kuramoto} below),
and two-dimensional stochastic Navier-Stokes equations (see Corollary~\ref{corollary:2DNavier}
in Subsection~\ref{sec:2DNavier} below).

In this introductory section we now illustrate the proposed approximation schemes and our
main result
(see Theorem~\ref{theorem:full_discrete_scheme_moments}) in the case of a stochastic Burgers
equation
(cf., e.g., 
Section~1 in 
Da Prato et al.~\cite{DaPratoDebusscheTemam1994}
and
Section~2 in Hairer \& Voss~\cite{HairerVoss2011}).
Let  
$ T \in (0,\infty) $,    
$ \delta \in (0, \nicefrac{1}{18}) $, 
$ H = L^2( (0,1) ; \R) $, 
let
$ Q \in L_1(H) $
be non-negative and symmetric,
  let
 $ ( \Omega, \mathcal{ F }, \P  ) $
 be a probability space,
 let
$ (W_t)_{t\in [0,T]} $
be an
$ \operatorname{Id}_H $-cylindrical 
 $ \P $-Wiener process,
let
$ A \colon D( A ) \subseteq H \to H $ 
be the Laplacian
with Dirichlet boundary conditions on $ H $,
let 
$ ( e_n )_{n \in \N} \subseteq H $,
$ (P_n)_{n \in \N} \subseteq L(H) $,
$ F \colon W_0^{1,2}( (0,1), \R) \to H $, 
$ \xi \in W_0^{1,2}( (0,1), \R ) $
satisfy for all 
$ n \in \N $, 
$ u \in H $,  
$ v \in W_0^{1,2}( (0,1), \R ) $ 
that   
$ e_n ( \cdot ) = \sqrt{2} \sin( n \pi ( \cdot ) ) $,
$ P_n(u) = 
\sum_{k=1}^n \langle e_k, u \rangle_H
e_k $,
  $ F( v ) = - v' \cdot v $,
let $ W^n \colon [0,T] \times \Omega \to P_n(H) $, $ n \in \N $, 
be
stochastic processes
with continuous sample paths
which satisfy for all $ n \in \N $, $ t \in [0,T] $
that
$ \P ( W_t^n = \int_0^t P_n \, dW_s ) = 1 $,
and
let 
$ Y^{ N, M }\colon [0, T] \times \Omega \to P_N(H) $, 
$ N, M \in \N $,
be   
stochastic processes 
which satisfy for all 
$ N, M \in \N $,
$ m \in \{ 0, 1, \ldots, M-1 \} $,   
$ t \in [ \frac{mT}{M}, \frac{ (m+1) T }{M} ] $
that
$ Y^{N, M}_0 = P_N(\xi) $ and
\begin{multline}
\label{scheme:full_discrete_intro}
  Y_t^{N, M}   
=  
e^{(t-\nicefrac{mT}{M} )A} 
\Big( 
Y_{
	mT / M
}^{  N, M} 
+ 
\1_{  \{ 
\| ( - A )^{\nicefrac{1}{2}} Y^{N, M}_{ mT / M } \|_H^2  
	+ 1
	\leq 
	\nicefrac{ M^\delta }{ T^\delta }   \}}
P_N
\Big[ 
F(
Y_{ mT / M }^{N, M}
) \,
(
t - \tfrac{mT}{M}
)  
\\
+
\tfrac{ 
	Q^{\nicefrac{1}{2}}
		( W_t^N - W_{ mT / M }^N )
}{
1 + 
\|  
P_N
Q^{\nicefrac{1}{2}} 
( W_t^N - W_{ mT / M }^N )
\|_H^2
} 
\Big]
\Big) 
\end{multline}
(cf., e.g., 
\cite{HutzenthalerJentzenKloeden2012,
HutzenthalerJentzenMemoires2015,
Sabanis2013ECP,
Sabanis2014Arxiv,
HutzenthalerJentzenWang2014,
MasterRyan,
GyongySabanisSiska2016,
JentzenPusnik2015,
BeckerJentzen2016,
HutzenthalerJentzenSalimova2016}
for related schemes).
In Corollary~\ref{corollary:Burgers}
in Subsection~\ref{sec:Burgers}
below we demonstrate
that the approximation scheme~\eqref{scheme:full_discrete_intro}
enjoys finite exponential moments.
More precisely,
Corollary~\ref{corollary:Burgers}
in Subsection~\ref{sec:Burgers}
proves\footnote{(with 
	$ d = 1 $,
	$ \mathcal{D} = (0,1) $,
	$ \eta = 0 $,
$ \gamma = \nicefrac{1}{2} $,  
$ T = T $,
$ \varepsilon = \varepsilon - \nicefrac{1}{ \sqrt{3} } $,
$ \delta = \delta $,
$ U = H $,
$ H = H $,
$ \H = \{ e_n \colon n \in \N \} $,
$ \mathbb{U} = \{ e_n \colon n \in \N \} $,
$ \lambda_{e_N} = - \pi^2 N^2 $,
$ ( \Omega, \mathcal{F}, \P, ( \mathcal{F}_t )_{t \in [0,T]} )
=
( \Omega, \mathcal{F}, \P, ( \sigma_\Omega( (W_s )_{s \in [0,t]} ) )_{t \in [0,T]} ) $,
$ W = W $,
$ Q = Q $,
$ A = A $,
$ r = ( H_{ \nicefrac{1}{2} } \ni v \mapsto 
 2 \varepsilon \max \{ 1, \sqrt{ \operatorname{trace}_H(Q) } \} 
+
   2 \varepsilon 
  \max \{ 1, \sqrt{ \operatorname{trace}_H(Q) } \}
\|  (-A)^{ \nicefrac{1}{2} }  v  \|_H^2 \in [0,\infty) ) $, 
$ b = ( (0,1) \times \R \ni (x,y) \mapsto 1 \in \R ) $,
$ \vartheta = \operatorname{trace}_H(Q) $, 
$ c = 2 \varepsilon \max \{ 1, \sqrt{ \operatorname{trace}_H(Q) } \} $,
$ R = \operatorname{Id}_H $,
$ F = F $,
$ \xi = ( \Omega \ni \omega \mapsto \xi \in W_0^{1,2}( (0,1), \R) ) $,
$ Y^{ \{ 0, T/M,\ldots, T \}, \{ e_1, \ldots, e_N \}, \{ e_1, \ldots, e_N \}} = Y^{N, M} $
for 
$ N, M \in \N $,
$ \varepsilon \in [1, \infty) $
in the notation of Corollary~\ref{corollary:Burgers})
} 
that
for all $ \varepsilon \in [1,\infty) $ 
it holds that
\begin{equation}
  \begin{split}
  \label{eq:BurgersExpMomentsIntro}
  & 
  \sup_{N, M \in \N} 
  \sup_{t\in [0,T]} 
  \E\!\left[
       \exp\!\left(
         \tfrac{   
         \varepsilon      
           \| Y^{ N, M  }_t \|_H^2
         }
         {
         e^{ 2 \varepsilon  \operatorname{trace}_H(Q)  t }
         }
       \right)
     \right] 
<\infty 
.
\end{split}
  \end{equation}
Corollary~\ref{corollary:Burgers}
follows from an application of Corollary~\ref{Corollary:full_discrete_scheme_convergence} below
(see Subsection~\ref{sec:Burgers} below for details).
Corollary~\ref{Corollary:full_discrete_scheme_convergence},
in turn,
is a direct consequence of
Theorem~\ref{theorem:full_discrete_scheme_moments},
which is the main result of this article.
Theorem~\ref{theorem:full_discrete_scheme_moments}
establishes
exponential integrability properties
for a more general class of SPDEs
(such as
stochastic Burgers equations with non-additive noise,
stochastic Kuramoto-Sivashinsky equations,
and two-dimensional
stochastic Navier-Stokes equations on a torus) 
as well as for a more general type of approximation schemes.
Exponential integrability
properties
such as~\eqref{eq:BurgersExpMomentsIntro}
are a key instrument
to establish
strong convergence rates for 
SPDEs
with
non-globally
monotone
nonlinearities
(cf.~\cite{HutzenthalerJentzen2014PerturbationArxiv}).
In particular we intend to use~\eqref{eq:BurgersExpMomentsIntro} 
and Theorem~\ref{theorem:full_discrete_scheme_moments},
respectively,
in succeeding articles 
to establish strong convergence rates  
for numerical approximations of stochastic Burgers
equations and other SPDEs
with non-globally monotone nonlinearities.

While polynomial moment bounds for numerical approximations of infinite dimensional SPDEs
and
exponential moment bounds
for numerical approximations of finite dimensional SODEs
have been established in the scientific literature,
Theorem~\ref{theorem:full_discrete_scheme_moments} 
is -- to the best of our knowledge -- the first result in the literature
which establishes
exponential moment bounds
for time discrete numerical approximations
in the case of infinite dimensional
SPDEs.
In particular,
Theorem~\ref{theorem:full_discrete_scheme_moments}
and 
its consequences in
Corollaries~\ref{Corollary:full_discrete_scheme_convergence},
\ref{corollary:Burgers},
\ref{corollary:Kuramoto},
and~\ref{corollary:2DNavier},
respectively,
are -- to the best of our knowledge -- the first
results
in the literature
that establish
exponential integrability
properties
for
time discrete numerical approximations of
 stochastic Burgers equations,
stochastic Kuramoto Sivashinsky equations,
and 
two-dimensional stochastic Navier Stokes equations.
\subsection{Notation}
Throughout this article the following notation is used.
For sets $ A $ and $ B $ we denote by $ \mathbb{M}(A, B) $ the set of all mappings from $ A $ to $ B $.
For a topological space $ (X, \tau) $ and a set $ D\subseteq X $
we denote by $ \mathring D \subseteq X $ the interior of $ D $. 
For a natural number $ k \in \N $ and 
normed $ \R $-vector spaces 
$ (U, \left \| \cdot \right \|_U) $ and 
$ (V, \left \| \cdot \right \|_V) $ we denote by
$ L^{(k)}(U,V) $ the set of all continuous $ k $-linear mappings from
$ U^k $ to $ V $, we denote by
$ \left \| \cdot \right \|_{L^{(k)}(U, V)}\colon
L^{(k)}(U, V) \to [0,\infty) $
the mapping which satisfies for all $ A\in L^{(k)}(U, V) $
 that
$ \| A \|_{L^{(k)}(U, V)}
=
\sup_{u_1,u_2,\ldots, u_k\in U\backslash \{0\}}
\big(
 \frac{
 \| A(u_1, u_2,\ldots, u_k)\|_V
 }
 {
 \|u_1\|_U \cdot \| u_2 \|_U \cdots \|u_k\|_U
 }
\big) $, 
we denote by $ L^{(0)}(U,\V) $ the set given by $ L^{(0)}(U,V) = V $,
and we denote by $ \left \| \cdot \right \|_{L^{(0)}(U, \V)}\colon
\V \to [0,\infty) $
the mapping which satisfies
for all $ v \in \V $ that
$ \| v \|_{L^{(0)}(U, \V)} =   \| v \|_\V $.
For measurable spaces $ (\Omega_1, \mathcal{F}_1) $ and
$ (\Omega_2, \mathcal{F}_2) $ we denote by $ \mathcal{M}( \mathcal{F}_1 , \mathcal{F}_2 ) $ 
the set of all $ \mathcal{F}_1/\mathcal{F}_2 $-measurable functions. 
For a normed $ \R $-vector space
$ ( V, \left\| \cdot \right\|_V) $,
a measure space 
$ ( \Omega, \mathcal{F}, \mu) $,
a real number $ p \in (0,\infty) $,
and a measurable function 
$ f \in \mathcal{M}( \mathcal{F}, 
\mathcal{B}(V) ) $
we denote by 
$ \| f \|_{\mathcal{L}^p(\mu; V)} \in [0,\infty] $
and
$ \| f \|_{\mathcal{L}^\infty(\mu; V)}
\in [0,\infty] $
the extended real numbers
given by
$ \| f \|_{\mathcal{L}^p(\mu; V)}
=
\left(
\int_\Omega \| f(\omega) \|_V^p \,\mu(d\omega)
\right)^{\nicefrac{1}{p}} $
and
$ \| f \|_{\mathcal{L}^\infty(\mu; V)}
=
\inf \{ c\in [0,\infty)\colon
\mu( \{ v \in V \colon | f(v) | > c  \} )
= 0
 \}
$.
For a topological space $ (X, \tau) $ we denote by $ \mathcal{B}(X) $ the sigma-algebra of all Borel measurable sets in $ X $.
For a natural number $ d\in \N $ and a Borel
measurable set
$ A \in \mathcal{B}( \R^d ) $
we denote by
$ \mu_{ A }
  \colon
  \mathcal{B}( A )
  \rightarrow [0,\infty] $
the Lebesgue-Borel
measure on
$ A \subseteq \R^d $.
For
$ \R $-Hilbert spaces  
$ (H, \left< \cdot, \cdot \right>_H, \left \| \cdot \right \|_H) $
and
$ (U, \left< \cdot, \cdot \right>_U, \left \| \cdot \right \|_U) $,
a set $ \mathcal{H} \in \mathcal{P}(H) $,
and
functions
$ F
  \colon \mathcal{H}
  \rightarrow H $
and
$ B
  \colon \mathcal{H}
  \rightarrow \HS(U,H) $
we denote by
$ \mathcal{G}_{ F, B }
  \colon
  \mathcal{C}^2( H, \mathbb{R} )
  \rightarrow
  \mathbb{M}( \mathcal{H}, \R) $
the function which satisfies for all $ x \in \mathcal{H} $,
$ \phi \in \mathcal{C}^2( H,
  \mathbb{R} ) $ 
  that
\begin{equation}
\label{eq:generator}
\begin{split}
  ( \mathcal{G}_{ F, B } \phi )
  (x)
  =
  \left<
    F(x),
    (\nabla \phi)(x)
  \right>_H
+
    \tfrac{ 1 }{ 2 }
  \tr_H\!\big(
  B(x)
    B(x)^{ * }
    (\text{Hess } \phi)(x) 
  \big)
  .
\end{split}
\end{equation}
For sets $ x $ and $ A $
we denote by
$ \1_A(x) $ the real number given by
\begin{equation}
\1_A(x) =
\begin{cases}
1 \colon x \in A
\\
0 \colon x \notin A
\end{cases}
.
\end{equation}
For sets $ \Omega $ and $ A $  
we denote by
$ \1_A^\Omega \colon \Omega \to \{ 0, 1 \} $
the function which satisfies for all $ x \in \Omega $ that
$ \1_A^\Omega ( x ) = \1_A ( x ) $.
For a set $ X $ we denote by $ \mathcal{P}(X) $ the power set of $ X $, we denote by 
$ \#_X \in \N_0 \cup \{\infty \} $
the number of elements of $ X $,
and we denote by $ \mathcal{P}_0(X) $ the set given by
$ \mathcal{P}_0(X) = \{ \theta \in \mathcal{P}(X) \colon \#_\theta <\infty\} $.
For a normed $ \R $-vector space $ (V, \left \| \cdot \right \|_V) $
with $ \#_V > 1 $,
real numbers
$ n \in \mathbb{N} $,
$ c \in [1,\infty) $,
a set $ B \subseteq \R $,
and an open and convex set $ A \subseteq V $
we denote by
$ \mathcal{C}^n_{ c }( A, B ) $
the set given by
\begin{equation}
\begin{split} 
  \mathcal{C}^n_{c}( A, B )
 & =
  \left\{ \!
    f \in \mathcal{C}^{ n - 1 }( A, B )
    \colon
    \!\!\!
    \begin{array}{c}
    \forall\,x,y\in A,
    i \in \N_0 \cap [0,n)\colon  
\| f^{(i)}(x)-f^{(i)}(y)\|_{L^{(i)}(H,\R)} 
\\
\leq
c
\| x - y \|_H ( 1 + \sup_{r\in [0,1]}|f(rx + (1-r)y) |)^{1-\nicefrac{1}{c}}
    \end{array}
    \!\!\!
  \right\}
\end{split}
\end{equation}
(cf., e.g.,\,(1.12) in Hutzenthaler \& Jentzen~\cite{HutzenthalerJentzenMemoires2015}).
We denote by 
$ ( \cdot ) \wedge ( \cdot )
  \colon \R^2 \to \R $
the function which satisfies
for all $ x, y \in \R $ that 
$ x \wedge y = \min \{ x,y \} $.
For a real number
$ T \in (0,\infty) $ 
we denote by $ \varpi_T $ the set given by
$ \varpi_T
  =
  \{
    \theta \subseteq [0,T] \colon
    \{0, T\} \subseteq \theta \text{ and } \#_\theta < \infty \}$. 
    For a real number $ T \in ( 0, \infty ) $
we denote by 
$ \left | \cdot \right|_T \colon \varpi_T \to [0,T] $
the mapping which satisfies for all
$ \theta \in \varpi_T $ that
\begin{equation}
  \left| \theta \right|_T
  =
  \max\!\Big\{
    x \in (0,\infty)
    \colon
    \big(
      \exists \, a, b \in \theta
      \colon
      \big[
        x = b - a
      \text{ and }
        \theta \cap ( a, b ) = \emptyset
      \big]
    \big)
  \Big\}
  \in (0,T].
\end{equation}
Let us note for every $ T \in (0,\infty) $,
$ \theta \in \varpi_T $
that $ |\theta|_T \in [0,T] $ is the maximum
step size of the partition $ \theta $.
We denote by
$ \lfloor \cdot \rfloor_{ \theta } \colon [0,\infty) \to [0,\infty) $,
$ \theta \in ( \cup_{T\in (0,\infty)} \varpi_T) $,
and
$ \llcorner \cdot \lrcorner_\theta \colon [0,\infty) \to [0,\infty) $,
$ \theta \in ( \cup_{T\in (0,\infty)} \varpi_T) $, 
the mappings which satisfy for all 
$ \theta \in ( \cup_{T\in (0,\infty)} \varpi_T) $,
$ t\in (0,\infty) $
that
$ \lfloor t \rfloor_{ \theta }
=
\max\!\left( [0,t] \cap \theta \right) $, 
$ \llcorner t \lrcorner_\theta
=  \max\!\left( [0,t) \cap \theta \right) $,
and $ \lfloor 0 \rfloor_\theta = \llcorner 0 \lrcorner_\theta  = 0 $.
For a measure space
$ ( \Omega, \mathcal{F}, \mu) $,
a measurable space $ ( S, \mathcal{S} ) $,
a set $ R $,
and a function
$ f \colon \Omega \to R $
we denote by
$ [f]_{\mu, \mathcal{S}} $
the set given by
$ [f]_{\mu, \mathcal{S}} =
\{
g \in \mathcal{M}(\mathcal{F},
\mathcal{S})
\colon
(
\exists\, A \in \mathcal{F} \colon
\mu(A) = 0 \,\,\text{and}\,\,
\{\omega \in \Omega \colon
f(\omega) \neq g(\omega)
\}
\subseteq A
)
\} $.
\section{Exponential moments for time discrete approximation schemes} 
\label{section:2}
\subsection{Factorization lemma for conditional expectations}
\label{subsection:Factorization}
In this subsection we recall in Definitions~\ref{def:cap_stability}--\ref{def:generator},
Lemma~\ref{lemma:delta_closed_system},
Theorem~\ref{thm:Dynkin},
and Lemmas~\ref{lemma:pointwise_approximation}--\ref{lemma:conditional_expectation}
some well known concepts and facts
from measure and probability theory.
In particular,
we recall in Lemma~\ref{lemma:conditional_expectation} below
a well-known factorization property for conditional expectations.
We use this factorization property in the proofs of our later results.
Definitions~\ref{def:cap_stability}--\ref{def:generator},
Lemma~\ref{lemma:delta_closed_system},
Theorem~\ref{thm:Dynkin},
and Lemma~\ref{lemma:pointwise_approximation}
can, e.g., in a very similar form be found in
Section~1 in Klenke~\cite{Klenke2008}
(see
Definition~1.1,
Definition~1.10,
Theorem~1.16,
Theorem~1.18,
Theorem~1.19,
and
Theorem~1.96
in Klenke~\cite{Klenke2008}).
Lemmas~\ref{lemma:conditional_expectation0}--\ref{lemma:conditional_expectation} 
can, e.g., in a very similar form be found in
Chapter~1 in Da Prato \& Zabczyk~\cite{dz92}
(see Proposition 1.12 in Da Prato \& Zabczyk~\cite{dz92}).
\begin{definition}[$ \cap $-Stability]
	\label{def:cap_stability}
	Let
	$ \mathcal{E} $
	be a set.
	Then we say that $ \mathcal{E} $
	is $ \cap $-stable
	if and only if 
	for all $ a, b \in \mathcal{E} $
	it holds that $ a \cap b \in \mathcal{E} $. 
\end{definition}
\begin{definition}[Dynkin system]
	Let $ \Omega $ 
	and
	$ \mathcal{A} $ be sets. 
	Then we say that
	$ \mathcal{A} $
	is a
	Dynkin system on $ \Omega $
	if and only if 
	\begin{enumerate}[(i)]
		\item it holds that 
		$ \Omega \in \mathcal{A} 
		\subseteq
		\mathcal{P}(\Omega) $,
		\item it holds for all $ A \in \mathcal{A} $
		 that
		$ \Omega \backslash A \in \mathcal{A} $, and
		\item it holds for all pairwise disjoint
		sets $ (A_n)_{n\in \N} \subseteq \mathcal{A} $
		that
		$ \cup_{n\in \N} A_n \in \mathcal{A} $.
	\end{enumerate}
\end{definition}
\begin{definition}
	\label{def:generator}
	Let $ \Omega $ and
	$ \mathcal{A}  $
	be sets with $ \mathcal{A} \subseteq \mathcal{P}(\Omega)  $.
	Then we
	denote by
	$ \delta_\Omega( \mathcal{A} ) $
	the set
	given by
	\begin{equation}
	\begin{split}
	\delta_\Omega(\mathcal{A})
	=
	\cap_{ \mathcal{B}
		\in
		\big\{\substack{ 
		\mathcal{C}	
			\text{ is a Dynkin system}
			\\
			\text{on }
			\Omega
			\text{ with }
			\mathcal{C} \supseteq \mathcal{A} 
		}
		\big\}
	}
	\mathcal{B}
	.
	\end{split}
	\end{equation}
\end{definition}
\begin{lemma}
	\label{lemma:delta_closed_system}
	Let $ \Omega $ be a set and let
	$ \mathcal{A} $
	be a Dynkin system on $ \Omega $. 
	Then it holds that
	$ \mathcal{A} $
	is $ \cap $-stable  
	if and only if 
	$ \mathcal{A} $
	is a sigma-algebra on $ \Omega $.
\end{lemma}
\begin{proof}[Proof of Lemma~\ref{lemma:delta_closed_system}]
	Throughout this proof 
	assume w.l.o.g.\,that 
	$ \mathcal{A} \subseteq \mathcal{P}(\Omega) $
	is
	a
	$ \cap $-stable
	Dynkin system on $ \Omega $
	(otherwise the statement
	of Lemma~\ref{lemma:delta_closed_system}
	is clear). 
	Note that the assumption that
	$ \mathcal{A} $ is a Dynkin system on $ \Omega $
	ensures for all
	$ A \in \mathcal{A} $
	that
	\begin{equation} 
	\label{eq:cap_closed}
	 (\Omega \backslash A) \in \mathcal{A}
	 .
	\end{equation}
	This and the fact that 
	$ \forall \, A, B \in \mathcal{A} \colon
	A \cap B \in \mathcal{A} $
	imply
	that for all
	$ A, B \in \mathcal{A} $
	it holds that
	$ A \backslash B = A \cap ( \Omega \backslash B ) \in \mathcal{A} $.
	Hence, we obtain 
	that for all
	$ (A_n)_{n\in \N} \subseteq \mathcal{A} $
	it holds that
	\begin{equation}
	\begin{split}
	\cup_{n\in \N}
	A_n
	=
	A_1
	\cup
	\big[
	\cup_{n \in \N}
	( 
	(
	\cdots
	(
	(
	A_{n+1} 
	\backslash
	A_n
	)
	\backslash
	A_{n-1}
	)
	\cdots
	)
	\backslash A_1
	) 
	\big]
	\in 
	\mathcal{A}
	.
	\end{split}
	\end{equation}
	Combining
	this,
	the fact that
	$ \Omega \in \mathcal{A} $,
	and~\eqref{eq:cap_closed}
	proves that
	$ \mathcal{A} $
	is a sigma-algebra on $ \Omega $.
	The proof of Lemma~\ref{lemma:delta_closed_system}
	is thus completed.
\end{proof}
\begin{theorem}
	\label{thm:Dynkin}
	Let $ \Omega $ be a set
	and
	let $ \mathcal{A} \in
	\mathcal{P} ( \mathcal{P}(\Omega) ) $
	be
	$ \cap $-stable.
	Then 
	$ \sigma_\Omega(\mathcal{A})
	= \delta_\Omega(\mathcal{A}) $.
\end{theorem}
\begin{proof}[Proof of Theorem~\ref{thm:Dynkin}]
	Throughout
	this proof 
	let
	$ \mathcal{D}_A \subseteq \mathcal{P}(\Omega) $,
	$ A \in \delta_\Omega(\mathcal{A}) $,
	be the sets which satisfy
	for all
	$ A \in \delta_\Omega(\mathcal{A}) $
	that
	$ \mathcal{D}_A
	= \{ B \in \delta_\Omega(\mathcal{A})
	\colon A \cap B \in \delta_\Omega(\mathcal{A})
	\} $.
	Note that for all
	$ A \in \delta_\Omega(\mathcal{A}) $
	it holds that
	$ A \cap \Omega  = A \in \delta_\Omega (\mathcal{A}) $.
	This proves that
	for all $ A \in \delta_\Omega(\mathcal{A}) $
	it holds that
	\begin{equation} 
	\label{eq:condd1}
	\Omega \in \mathcal{D}_A
	.
	\end{equation}
	In the next step we observe 
	that for all
	$ A \in \delta_\Omega (\mathcal{A}) $,
	$ B \in \mathcal{D}_A $ 
	it holds that
	\begin{equation}
	\label{eq:condd2}
	A \cap (\Omega \backslash B)
	= 
	A \backslash ( A \cap B )
	=
	A \cap [ \Omega \backslash ( A \cap B ) ]
	= 
	\Omega \backslash
	[
	( \Omega \backslash A )
	\cup
	( A \cap B ) 
	]
	\in 
	\delta_\Omega(\mathcal{A}) 
	.
	\end{equation} 
	Moreover, note that
	for all
	$ A \in \delta_\Omega(\mathcal{A}) $
	and all pairwise disjoint sets
	$ (B_n)_{n\in \N} 
	\subseteq \mathcal{D}_A $
	it holds that
	\begin{equation}
	\label{eq:condd3}
	A \cap
	\left(
	\cup_{n\in \N} 
	B_n
	\right) 
	=
	\cup_{n\in \N}
	( A \cap B_n )
	\in
	\delta_\Omega(\mathcal{A}).
	\end{equation}
	Combining~\eqref{eq:condd1},
	\eqref{eq:condd2},
	and~\eqref{eq:condd3} proves that for every
	$ A \in \delta_\Omega(\mathcal{A}) $
	it holds that
	$ \mathcal{D}_A $
	is a Dynkin system on $ \Omega $.
	Next
	note 
	that 
	the assumption that $ \mathcal{A} $ is
	$ \cap $-stable 
	implies that
	for all 
	$ A \in \mathcal{A} $
	it holds
	that
	$ \mathcal{A}  \subseteq
	\mathcal{D}_A $.
	This
	and the fact that
	for every $ A \in \delta_\Omega( \mathcal{A} ) $
	it holds that
	$ \mathcal{D}_A $
	is a Dynkin system
	on $ \Omega $
	proves 
	that for all
	$ A \in \mathcal{A} $
	it holds that
	$ \delta_\Omega(\mathcal{A}) \subseteq 
	\delta_\Omega( \mathcal{D}_A)
	=
	\mathcal{D}_A $. 
	This implies that for all
	$ A \in \mathcal{A} $,
	$ B \in \delta_\Omega( \mathcal{A} ) $
	it holds that
	$ A \cap B \in \delta_\Omega( \mathcal{A} )$.
	This ensures that for all
	$ A \in \mathcal{A} $,
	$ B \in \delta_\Omega(\mathcal{A}) $
	it holds that
	$ A \in \mathcal{D}_B $.
	Hence, we obtain that
	for all
	$ B \in \delta_\Omega(\mathcal{A}) $
	it holds that
	$ \mathcal{A} \subseteq \mathcal{D}_B $.
	In particular,
	we obtain that for all
	$ A \in \delta_\Omega( \mathcal{A} ) $
	it holds that
	$ \mathcal{A} \subseteq \mathcal{D}_A $.
	Combining
	this 
	with the fact that
	for every $ A \in \delta_\Omega( \mathcal{A} ) $
	it holds that
	$ \mathcal{D}_A $
	is a Dynkin system on $ \Omega $
	assures
	that
	for all $ A \in \delta_\Omega(\mathcal{A}) $
	it holds that
	$
	\delta_\Omega(\mathcal{A})
	\subseteq 
	\delta_\Omega( \mathcal{D}_A )
	=
	\mathcal{D}_A 
	\subseteq
	\delta_\Omega( \mathcal{A} ) 
	$.
	Therefore,
	we obtain that for all
	$ A, B \in \delta_\Omega(\mathcal{A}) $
	it holds that
	$ A \cap B \in \delta_\Omega( \mathcal{A} ) $.
	Combining this with
	Lemma~\ref{lemma:delta_closed_system}
	completes the proof of Theorem~\ref{thm:Dynkin}.
\end{proof}
\begin{lemma}
	\label{lemma:pointwise_approximation}
	Let $ (\Omega, \mathcal{F}) $
	be a measurable space
	and let
	$ f \in \mathcal{M}(\mathcal{F},
	\mathcal{B}([0,\infty])) $.
	Then there exists a sequence
	$ f_n \in
	\mathcal{M}( \mathcal{F}, \mathcal{B}( [0,\infty) ) ) $, 
	$ n \in \N $,
	which satisfies for all $ n \in \N $,
	$ \omega \in \Omega $ 
	that
	$ \#_{f_n(\Omega)} <  \infty $,
	$ f_n(\omega) \leq f_{n+1}(\omega) $,
	and $ \lim_{m \to \infty} f_m(\omega) = f(\omega) $.
\end{lemma}
\begin{proof}[Proof of Lemma~\ref{lemma:pointwise_approximation}]
	Throughout this proof let
	$ A_n \in \mathcal{P}(\Omega) $,  
	$ n \in \N $,
	and
	$ B_{n,j}\in \mathcal{P}(\Omega) $, 
	$ j \in \N \cap [1, n2^n] $,
	$ n \in \N $,
	be the sets 
	which satisfy
	for all
	$ n \in \N $,
	$ j \in \N \cap [1, n2^n] $
	that
	$ A_n = \{ \omega \in \Omega \colon 
	f(\omega) \in [n,\infty] \} $
	and
	$
	B_{n,j}
	= \{ \omega \in \Omega
	\colon f(\omega) \in [ \nicefrac{(j-1)}{2^n}, \nicefrac{j}{2^n}) \} $
	and let
	$ f_n \colon \Omega \to [0,\infty) $,
	$ n \in \N $,
	be the functions 
	which satisfy for all
	$ n \in \N $,
	$ \omega \in \Omega $ that
	\begin{equation}
	\label{eq:define_simple}
	f_n(\omega)
	=
	n \, \mathbbm{1}_{A_n}(\omega)
	+
	\sum_{j=1}^{n2^n}
	\tfrac{ j-1}{2^n}
	\, 
	\mathbbm{1}_{B_{n,j}}(\omega)
	.
	\end{equation}
	Note that for every
	$ \omega \in \Omega $
	with $ f ( \omega ) \in [0,\infty) $
	it holds that
	there exist 
	$ n \in \N $,
	$ j \in \N \cap [1, n2^n] $
	such that
	$ f(\omega)
	\in [
	\nicefrac{(j-1)}{2^n},
	\nicefrac{j}{2^n}
	)
	$.
	This and~\eqref{eq:define_simple} imply that for every
	$ \omega \in \Omega $
	with $ f ( \omega ) \in [0,\infty) $
	it holds that
	there exists $ n \in \N $
	such that
	$ 0 \leq f(\omega) - f_n(\omega) < 2^{-n} $.
	Hence, we obtain that
	for all $ \omega \in \Omega $
	with 
	$ f(\omega) \in [0,\infty) $
	it holds that
	\begin{equation} 
	\label{eq:limits}
	\limsup_{n\to \infty} 
	| f_n(\omega) - f( \omega) | = 0 
	.
	\end{equation}
	In addition, note that for all
	$ m \in \N $,
	$ \omega \in \Omega $ with
	$ f( \omega) = \infty $
	it holds that 
	$ f_m(\omega) = m $.
	This proves that
	for all $ \omega \in \Omega $ 
	with $ f(\omega) = \infty $
	it holds that
	$ \liminf_{m \to \infty} f_m(\omega) = \infty $.
	Combining
	this and~\eqref{eq:limits}
	completes the proof of
    Lemma~\ref{lemma:pointwise_approximation}.
\end{proof}
\begin{lemma}
	\label{lemma:conditional_expectation0}
	Let $ (\Omega, \mathcal{F}, \P) $ be a probability space, let $ (D, \mathcal{D}) $
	and
	$ (E, \mathcal{E}) $
	be measurable spaces,
	let $ \mathcal{X},\mathcal{Y}\in \mathcal{P}(\mathcal{F}) $ be 
	$ \P $-independent sigma-algebras,
	let 
	$ X \in \mathcal{M}(\mathcal{X}, \mathcal{D}) $,
	$ Y \in \mathcal{M}(\mathcal{Y}, \mathcal{E}) $,
	$ A \in \mathcal{D} \otimes \mathcal{E} $,
	$ \Psi \in \mathbb{M}( D, [0,\infty] ) $,
	and 
	assume for all $ x \in D $ that 
	$ \Psi(x)= \E[ \1^{D \times E}_A (x,Y) ] $.
	Then
	it holds that
	$ \Psi \in \mathcal{M} ( \mathcal{D}, \mathcal{B}( [0,\infty] ) ) $
	and
	\begin{equation}
	\E [ \1^{D \times E}_A (X, Y) | \mathcal{X}] 
	= 
	[\Psi(X)]_{\P|_{\mathcal{X}}, \mathcal{B}([0,\infty])}
	.
	\end{equation}
\end{lemma}
\begin{proof}[Proof of Lemma~\ref{lemma:conditional_expectation0}]
	Throughout this proof  
	let
	$ \gamma_C \colon D \to [0,\infty] $,
	$ C \in \mathcal{D} \otimes \mathcal{E} $,
	be the functions which satisfy for all
	$ C \in \mathcal{D} \otimes \mathcal{E} $,
	$ x \in D $ 
	that
	$ \gamma_C(x) =
	\E[ \1^{D \times E}_C (x,Y)] $ 
	and let
	$ \mathcal{C} \subseteq
	\mathcal{D} \otimes \mathcal{E} $
	be the set
	given by
	$ \mathcal{C}
	=
	\{ C \in \mathcal{D} \otimes 
	\mathcal{E} \colon 
	\E[ \1^{D \times E}_C(X,Y) | \mathcal{X} ]
	=
	[  
	 \gamma_C(X)
	 ]_{\P|_{\mathcal{X}}, \mathcal{B}([0,\infty])}
	\} $.
	Note that Tonelli's theorem and the fact that
	$ D \times E \ni (x, y) \mapsto \1^{D \times E}_A (x, y) \in [0,\infty] $
	is $ ( \mathcal{D} \otimes \mathcal{E} ) / \mathcal{B}( [0,\infty] ) $-measurable
	show that
	\begin{equation} 
	\label{eq:is_measurable}
	\Psi \in \mathcal{M}( \mathcal{D}, \mathcal{B}( [0,\infty] ) ) 
	.
	\end{equation}
	Moreover, observe that 
	for all
	$ x \in D $,
	$ C  \in
	\mathcal{D} \otimes \mathcal{E} $ 
	it holds 
	that
	\begin{equation} 
	\begin{split}
	\gamma_{( D \times E ) \backslash C}(x)
	&
	=
	\E[
	\1^{D \times E}_{D \times E}(x,Y) 
	-
	\1^{D \times E}_{C}(x,Y)
	]
	\\
	&
	=
	\E[
	\1^{D \times E}_{D \times E}(x,Y)
	] 
	-
	\E[
	\1^{D \times E}_{C}(x,Y)
	]
	=
	\gamma_{D \times E}(x) - \gamma_{C}(x) 
	.
	\end{split}
	\end{equation}
	This ensures for all
	$ C \in
	\mathcal{C} $ 
	that
	\begin{equation}
	\begin{split}
	\label{eq:substractionclosed} 
	\E[ \1^{D \times E}_{( D \times E ) \backslash C}(X,Y) | \mathcal{X} ]
	&
	=
	\E[ \1^{D \times E}_{D \times E}(X,Y) - \1^{D \times E}_{C }(X,Y) |
	\mathcal{X} ]
	\\
	&
	=
	\E[ \1^{D \times E}_{D \times E}(X,Y) |
	\mathcal{X} ]
	- 
	\E[ \1^{D \times E}_{C }(X,Y) |
	\mathcal{X} ] 
	\\
	&
	=
[	\gamma_{D \times E}(X) - \gamma_{C }(X)
]_{\P|_{\mathcal{X}}, \mathcal{B}([0,\infty])}
	=
	[ \gamma_{ ( D \times E ) \backslash C }(X)
	]_{\P|_{\mathcal{X}}, \mathcal{B}([0,\infty])}
	.
	\end{split}
	\end{equation}
	Next observe that
	the monotone convergence theorem
	proves
	that
	for all
	$ x \in D $ and all
	pairwise
	disjoint sets
	$ ( C_n)_{n\in \N} \subseteq \mathcal{D}
	\otimes \mathcal{E} $
	it holds
	that
	\begin{equation}
	\begin{split}
	\gamma_{\cup_{n = 1}^\infty C_n }(x)
	&
	=
	\E \big[ 
	\1^{D \times E}_{\cup_{n = 1}^\infty C_n}(x,Y ) 
	\big ]
	=
	\E\big [ \smallsum_{n = 1}^\infty 
	\1^{D \times E}_{C_n}(x,Y) \big]
	\\
	&
	=
	\smallsum_{n = 1}^\infty
	\E[ \1^{D \times E}_{C_n} (x,Y) ]
	=
	\smallsum_{n = 1}^\infty
	\gamma_{C_n}(x).
	\end{split}
	\end{equation}
	The monotone convergence theorem for conditional
	expectations 
	hence
	shows that 
	for all
	pairwise disjoint sets
	$ (C_n)_{n\in \N} \subseteq \mathcal{C} $
	it holds
	that
	\begin{equation}
	\begin{split}
	\label{eq:countablyclosed}
	\E \big[ 
	\1^{D \times E}_{\cup_{n = 1}^\infty C_n} (X,Y) | \mathcal{X} 
	\big ]
	& 
	=
	\E\big[
	\smallsum_{n = 1}^\infty 
	\1^{D \times E}_{C_n}(X,Y)
	| \mathcal{X}
	\big ]
	=
	\smallsum_{n = 1}^\infty
	\E[
	\1^{D \times E}_{C_n}(X,Y)
	| \mathcal{X}
	]
	\\
	&
	=
	\smallsum_{n = 1}^\infty
	[
	\gamma_{C_n}(X)
	]_{\P|_{\mathcal{X}}, \mathcal{B}([0,\infty])} 
	=
	[
	\gamma_{\cup_{n = 1}^\infty C_n}(X)
	]_{\P|_{\mathcal{X}}, \mathcal{B}([0,\infty])} .
	\end{split}
	\end{equation}
	Combining~\eqref{eq:substractionclosed}, \eqref{eq:countablyclosed},
	and
	the fact that
	$ ( D \times E ) \in \mathcal{C} $
	implies that
	$ \mathcal{C} $
	is a Dynkin system on 
	$ D \times E $.
	Moreover, note that for all
	$ \mathbf{D} \in \mathcal{D} $,
	$ \mathbf{E} \in \mathcal{E} $
	it holds that
	\begin{equation}
	\begin{split}
	\label{eq:indicatot}
	\E[ \1^{D \times E}_{\mathbf{D} \times \mathbf{E}}(X,Y) |
	\mathcal{X} ]
	&
	=
	\E[ \1_{ \mathbf{D} }^D( X ) \, \1_{ \mathbf{E} }^E(Y)
	|\mathcal{X} ]
	=
	 \1^D_{ \mathbf{D} }(X) \, \E[ \1^E_{ \mathbf{E} }(Y) |\mathcal{X}]
	\\
	&
	=
	[ \1_{ \mathbf{D} }^D(X) \,  \E[ \1_{ \mathbf{E} }^E(Y) ] ]_{\P|_{\mathcal{X}}, \mathcal{B}([0,\infty])} 
	.
	\end{split}
	\end{equation}
	This ensures that
	$ \{ \mathbf{D} \times \mathbf{E} 
	\in
	\mathcal{P} ( D \times E )
	\colon \mathbf{D} \in \mathcal{D}, \mathbf{E} \in \mathcal{E} \} 
	\subseteq \mathcal{C} $.
	Combining this,
	the fact that
	the set 
	$ \{ \mathbf{D} \times \mathbf{E} 
	\in \mathcal{P} ( D \times E )
	\colon \mathbf{D} \in \mathcal{D}, \mathbf{E} \in \mathcal{E} \} $
	is $ \cap $-stable, 
	and Theorem~\ref{thm:Dynkin}
	(with
	$ \Omega = D \times E $,
	$ \mathcal{A} = \{ \mathbf{D} \times \mathbf{E}
	\in \mathcal{P} ( D \times E ) \colon \mathbf{D} \in \mathcal{D}, \mathbf{E} \in \mathcal{E} \} $
	in the notation of Theorem~\ref{thm:Dynkin}) 
	proves that
	\begin{equation} 
	\begin{split} 
	&
	\mathcal{D} \otimes \mathcal{E}
	=
	\sigma_{D \times E}
	( \{ \mathbf{D} \times \mathbf{E} 
	\in
	\mathcal{P} ( D \times E )
	\colon \mathbf{D} \in \mathcal{D}, \mathbf{E} \in \mathcal{E} \} )
	\\
	&
	=
	\delta_{D \times E}
	( \{ \mathbf{D} \times \mathbf{E} 
	\in
	\mathcal{P} ( D \times E )
	\colon \mathbf{D} \in \mathcal{D}, \mathbf{D} \in \mathcal{E} \} )
	\subseteq
	\delta_{D \times E}
	( \mathcal{C} ) 
	\subseteq  
	\delta_{ D \times E } ( \mathcal{D} \otimes \mathcal{E} )
	=
	\mathcal{D} \otimes \mathcal{E}
	.
	\end{split}
	\end{equation}
	The fact
	that  
	the set
	$ \mathcal{C} $
	is a Dynkin system
	on $ D \times E $ 
	hence
	assures that 
	$ \mathcal{D} \otimes \mathcal{E} 
	= \delta_{ D \times E }( \mathcal{C} ) 
	= \mathcal{C} $.
	Therefore, we obtain 
	that
	for all
	$ C \in \mathcal{D} 
	\otimes \mathcal{E} $  
	it holds that
	\begin{equation} 
	\E[ \1^{ D \times E }_C(X,Y) | \mathcal{X}]
	=
	[ \gamma_C(X)]_{\P|_{\mathcal{X}}, \mathcal{B}([0,\infty])} 
	.
	\end{equation}
	This and~\eqref{eq:is_measurable} 
	complete the proof of Lemma~\ref{lemma:conditional_expectation0}.
\end{proof}
\begin{lemma}
\label{lemma:simple_functions}
	Let $ (\Omega, \mathcal{F}, \P) $ be a probability space, 
	let $ (D, \mathcal{D}) $
	and
	$ (E, \mathcal{E}) $
	be measurable spaces,
	let $ \mathcal{X},\mathcal{Y}\in \mathcal{P}(\mathcal{F}) $ be 
	$ \P $-independent sigma-algebras,
	let 
	$ X \in \mathcal{M}(\mathcal{X}, \mathcal{D}) $,
	$ Y \in \mathcal{M}(\mathcal{Y}, \mathcal{E}) $, 
	$ \Phi \in \mathcal{M}( \mathcal{D} \otimes \mathcal{E}, \mathcal{B} ( [0, \infty] ) ) $,
	$ \Psi \in \mathbb{M}( D, [0,\infty] ) $,
	and 
	assume for all $ x \in D $ that
	$ \#_{ \Phi(D \times E) } < \infty $, 
	$ \Psi(x)= \E[ \Phi(x,Y) ] $.
	Then
	it holds that
	$ \Psi \in \mathcal{M} ( \mathcal{D}, \mathcal{B}( [0,\infty] ) ) $
	and
	\begin{equation}
	\E [ \Phi (X, Y) | \mathcal{X}] 
	= 
	[\Psi(X)]_{\P|_{\mathcal{X}}, \mathcal{B}([0,\infty])}
	.
	\end{equation}
\end{lemma}
\begin{proof}[Proof of Lemma~\ref{lemma:simple_functions}]
Throughout this proof  
	let
	$ \gamma_C \colon D \to [0,\infty] $,
	$ C \in \mathcal{D} \otimes \mathcal{E} $,
	be the functions which satisfy for all
	$ C \in \mathcal{D} \otimes \mathcal{E} $,
	$ x \in D $ 
	that
	$ \gamma_C(x) =
	\E[ \1^{ D \times E }_C (x,Y)] $.
	Note that
	for all
	$ x \in D $,
	$ y \in E $
	it holds that
	\begin{equation}
	\begin{split}
	\label{eq:finite_sum}
	\Phi(x, y) 
	=
	\sum_{z \in \Phi( D \times E )}
	z \,\1^{ D \times E }_{ \Phi^{-1}( \{ z \} ) }(x, y) 
	.
	\end{split}
	\end{equation}
	The assumption that
	$ \Phi \in \mathcal{M}( \mathcal{D} \otimes \mathcal{E}, [0, \infty] ) $
	implies that
	for all $ z \in \Phi( D \times E ) $ it holds that
	$ \Phi^{-1}( \{ z \} ) \in \mathcal{D} \otimes \mathcal{E} $.
	Combining this and Lemma~\ref{lemma:conditional_expectation0}
	(with
	$ ( \Omega, \mathcal{F}, \P ) = ( \Omega, \mathcal{F}, \P ) $,
	$ ( D, \mathcal{D} ) = ( D, \mathcal{D} ) $,
	$ ( E, \mathcal{E} ) = ( E, \mathcal{E} ) $,
	$ \mathcal{X} = \mathcal{X} $,
	$ \mathcal{Y} = \mathcal{Y} $,
	$ X = X $,
	$ Y = Y $,
	$ A = \Phi^{-1}( \{ z \} ) $
	for $ z \in \Phi( D \times E ) $
	in the notation of Lemma~\ref{lemma:conditional_expectation0})
	proves that
	for all $ z \in \Phi( D \times E ) $
	it holds that
	$ ( D \ni x \mapsto 
	\E[ \1^{ D \times E }_{ \Phi^{-1}( \{ z \} ) }(x,Y) ] 
	\in [0,\infty] )  
	\in 
	\mathcal{M}( \mathcal{D}, \mathcal{B}([0,\infty]) ) $
	and
	\begin{equation} 
	\E[ \1^{ D \times E }_{ \Phi^{-1}( \{ z \} ) }(X,Y) | \mathcal{X}]
	=
	[ \gamma_{ \Phi^{-1}( \{ z \} ) }(X)]_{\P|_{\mathcal{X}}, \mathcal{B}([0,\infty ])} 
	.
	\end{equation}
	This and~\eqref{eq:finite_sum}
	show that 
	$ \Psi \in \mathcal{M}( \mathcal{D}, \mathcal{B}([0,\infty]) ) $
	and
	\begin{equation}
	\begin{split}
	&
	\E[\Phi(X,Y) |\mathcal{X}]  
	=
	\E\big[ \smallsum_{z \in \Phi( D \times E )}
	z \, \1^{ D \times E }_{ \Phi^{-1}( \{ z \} ) }(X, Y) \big| \mathcal{X} \big]
	=
	\smallsum_{z \in \Phi( D \times E )} 
	z \,
	\E\big[ 
	\1^{ D \times E }_{ \Phi^{-1}( \{ z \} ) }(X, Y) \big| \mathcal{X} \big]
	\\
	& 
	=
	\smallsum_{z \in \Phi( D \times E )} 
	z \,
	[ \gamma_{ \Phi^{-1}( \{ z \} ) }(X)]_{\P|_{\mathcal{X}}, \mathcal{B}([0,\infty])} 
	= 
	\big[
	\smallsum_{z \in \Phi( D \times E )} 
	z \,
	\gamma_{ \Phi^{-1}( \{ z \} ) }(X)
	\big ]_{\P|_{\mathcal{X}}, \mathcal{B}([0,\infty])} 
	\\
	&
	=
	[ \Psi(X) ]_{\P|_{\mathcal{X}}, \mathcal{B}([0,\infty])}.
	\end{split}
	\end{equation}
This completes the
proof of Lemma~\ref{lemma:simple_functions}.
\end{proof}
\begin{lemma}
	\label{lemma:conditional_expectation}
	Let $ (\Omega, \mathcal{F}, \P) $ be a probability space, let $ (D, \mathcal{D}) $
	and
	$ (E, \mathcal{E}) $
	be measurable spaces,
	let $ \mathcal{X},\mathcal{Y}\in \mathcal{P}(\mathcal{F}) $ be 
	$ \P $-independent sigma-algebras,
	let 
	$ X \in \mathcal{M}(\mathcal{X}, \mathcal{D}) $,
	$ Y \in \mathcal{M}(\mathcal{Y}, \mathcal{E}) $,
	$ \Phi \in \mathcal{M}(\mathcal{D} \otimes \mathcal{E}, \mathcal{B}([0,\infty])) $,
	$ \Psi \in \mathbb{M}( D, [0,\infty] ) $,
	and 
	assume for all $ x \in D $ that 
	$ \Psi(x)= \E[ \Phi(x,Y) ] $.
	Then
	it holds that
	$ \Psi \in \mathcal{M} ( \mathcal{D}, \mathcal{B}( [0,\infty] ) ) $
	and
	\begin{equation}
	\E [ \Phi(X, Y) | X]
	=
	[ \Psi(X) ]_{ \P|_{ \sigma_\Omega(X) }, \mathcal{B}( [0,\infty]) }
	\subseteq
	\E[ \Phi(X, Y) | \mathcal{X}] 
	= 
	[\Psi(X)]_{\P|_{\mathcal{X}}, \mathcal{B}([0,\infty])}
	.
	\end{equation}
\end{lemma}
\begin{proof}[Proof of Lemma~\ref{lemma:conditional_expectation}]
	Throughout this proof 
	let 
	$ \phi_n 
	\in \mathcal{M}(
	\mathcal{D} \otimes \mathcal{E},
	\mathcal{B}( [0,\infty] )
	) $,
	$ n \in \N $,
	be functions
	which 
	satisfy
	for all
	$ m \in \N $,
	$ z \in D \times E $ 
	that
	$ \#_{\phi_m(D\times E)} < \infty $,
	$ \phi_m( z ) \leq \phi_{m+1}( z ) $,
	and
	$ \lim_{n\to \infty } \phi_n( z ) = \Phi( z ) $
	and
	let
	$ \psi_n \colon D \to [0,\infty] $, $ n \in \N $,
	be the functions which satisfy
	for all
	$ n \in \N $,
	$ x \in D $
	that
	$ \psi_n(x) = \E[\phi_n(x,Y)] $.
	Note that the monotone convergence theorem ensures for all $ x \in D $ that
	\begin{equation}
	\label{eq:imp_limit}
	\lim\nolimits_{n\to \infty}
	\psi_n(x)
	=
	\lim\nolimits_{n\to \infty}
	\E[\phi_n(x,Y)]
	=
	\E[
	\lim\nolimits_{n\to \infty}
	\phi_n(x,Y)
	]
	=
	\E[ \Phi(x,Y)] = \Psi(x).
	\end{equation}
	Combining  
	the monotone convergence theorem
	for conditional expectations
	and Lemma~\ref{lemma:simple_functions}
	(with
	$ ( \Omega, \mathcal{F}, \P ) = ( \Omega, \mathcal{F}, \P ) $,
	$ ( D, \mathcal{D} ) = ( D, \mathcal{D} ) $,
	$ ( E, \mathcal{E} ) = ( E, \mathcal{E} ) $,
	$ \mathcal{X} = \mathcal{X} $,
	$ \mathcal{Y} = \mathcal{Y} $,
	$ X = X $,
	$ Y = Y $,
	$ \Phi = \phi_n $,
	$ \Psi = \psi_n $
	for $ n \in \N $
	in the notation of 
	Lemma~\ref{lemma:simple_functions})
	hence shows 
	\begin{enumerate}[(i)]
	\item  
	 that
	$ \forall \, n \in \N \colon \psi_n \in \mathcal{M}( \mathcal{D}, \mathcal{B}( [0,\infty] ) ) $
	and
	\item
	 that
	\begin{equation}
	\begin{split}
	&
	\E[\Phi(X,Y) |\mathcal{X}] 
	=
	\E[ \lim\nolimits_{n\to \infty} \phi_n(X,Y) |\mathcal{X}]
	=
	\lim\nolimits_{n\to \infty} \E[\phi_n(X,Y)|\mathcal{X}]
	\\
	& 
	=
	\lim\nolimits_{n\to \infty}  
	[ \psi_n(X) ]_{\P|_{\mathcal{X}}, \mathcal{B}([0,\infty])}
	= 
	[ \lim\nolimits_{n\to \infty}  \psi_n(X) ]_{\P|_{\mathcal{X}}, \mathcal{B}([0,\infty])}
	=
	[ \Psi(X) ]_{\P|_{\mathcal{X}}, \mathcal{B}([0,\infty])}.
	\end{split}
	\end{equation}
	\end{enumerate}
	Combining
    this and~\eqref{eq:imp_limit} proves that  
    $ \Psi \in \mathcal{M}( \mathcal{D}, \mathcal{B}( [0, \infty] ) ) $
    and
	\begin{equation}
	\begin{split}
	\E[\Phi(X,Y) |X]
	&
	=
	\E\big[\E[\Phi(X,Y) |\mathcal{X}]|X\big]
	=
	\E [   \Psi(X)   |X ]
	\\
	&
	=
	[ \Psi( X ) ]_{
	\P |_{ \sigma_\Omega(X) },
	\mathcal{B}( [0, \infty] )
	}
	\subseteq
	[ \Psi(X) ]_{\P|_{\mathcal{X}}, \mathcal{B}([0,\infty])}
	.
	\end{split}
	\end{equation}
	The proof of Lemma~\ref{lemma:conditional_expectation} is thus completed.
\end{proof}
\subsection{From one-step estimates to exponential moments}
In this subsection we establish in
Corollary~\ref{Cor:exp.mom.abstract}
below
exponential integral
properties
for approximation
schemes
(see~\eqref{eq:exp.mom.abstract}
in Corollary~\ref{Cor:exp.mom.abstract})
under a general one-step
condition
on the considered
approximation scheme
(see~\eqref{eq:exp.mom.abstract.assumption}
in Corollary~\ref{Cor:exp.mom.abstract} below).
We will verify this one-step
condition for a specific class
of approximation schemes in Subsection~\ref{subsection:main_lemma} below.
Corollary~\ref{Cor:exp.mom.abstract}
is an extension
of Corollary~2.3
in Hutzenthaler et al.~\cite{HutzenthalerJentzenWang2014}.
\begin{corollary} 
\label{Cor:exp.mom.abstract}
 Let
  $ \left( H, \left< \cdot, \cdot \right> _H, \left \| \cdot \right\|_H \right) $
  and 
  $ \left( \U, \left< \cdot, \cdot \right> _\U, \left\| \cdot \right \|_\U \right) $
  be separable $ \R $-Hilbert spaces, 
  let
  $ T \in (0,\infty) $,
  $ \theta \in \varpi_T $,
  $ \rho, c \in [0,\infty) $,
  $ \V \in \mathcal{M} \big(
  \mathcal{B}(H), \mathcal{B}([0,\infty)) \big) $,
  $ \bar \V \in \mathcal{M} \big(
  \mathcal{B}(H), \mathcal{B}(\R)\big) $,
   $ \Phi \in \mathcal{M} \big(
   \mathcal{B}(H  \times
  [0, T] \times \U),
 \mathcal{B}( H)\big) $,
  $ E \in \mathcal{B}( H ) $,
  $ S \in \mathbb{M}((0,T], L(H)) $,
  let
  $ ( \Omega, \mathcal{ F }, \P, ( \mathcal{ F }_t )_{t \in [ 0, T]} ) $
  be a filtered probability space,
  let
  $ W \colon [0,T]\times \Omega \to \U $
  be an
  $ \operatorname{Id}_\U $-cylindrical 
  $ ( \mathcal{F}_t )_{t \in [0, T]} $-Wiener process 
  with continuous sample paths,
  let $ Y  \in \mathcal{M} \big( 
  \mathcal{B}([0,T]) \otimes \mathcal{F}, \mathcal{B}(H) \big)$ be 
  an $ (\mathcal{F}_t)_{t\in[0,T]} $-adapted
 stochastic process,
 assume for all $ t \in (0,T] $, $ x \in H $ that
  $ \V(S_t x) \leq \V(x) $,
  $ \bar \V(S_t x) \leq \bar \V(x) $,
  and
  \begin{equation}
  \label{eq:process_dynamic}
    Y_t =
     S_{t-\llcorner t \lrcorner_\theta}
    \left[
    \1_{ H \backslash E }( Y_{ \llcorner t \lrcorner_{ \theta } } )
    \cdot   
     Y_{ \llcorner t \lrcorner_{ \theta } }
    +
    \1_E(
      Y_{ \llcorner t \lrcorner_{ \theta } }
    )
    \cdot
    \Phi\!\left(
      Y_{ \llcorner t \lrcorner_{ \theta } } ,
      t - \llcorner t \lrcorner_{ \theta } ,
      W_{ t } - W_{ \llcorner t \lrcorner_{ \theta } }
    \right)
    \right]
    ,
  \end{equation}
 and assume for all 
 $ x \in E $,
 $ t \in  ( 0, |\theta|_T ] $
 that
 $ \smallint_0^T
 \1_{E}(
 Y_{ \lfloor s \rfloor_{ \theta } }
 )
 \,
 | \bar{\V}(Y_s) |
 \, ds 
 +
 \int_0^{|\theta|_T}
 | \bar V(\Phi(x,s,W_s)) | \, ds         
 < \infty $
 and
  \begin{equation}  
  \begin{split} 
  \label{eq:exp.mom.abstract.assumption}
    \E\!\left[
      \exp\!\left(
        \tfrac{
            \V( \Phi( x, t, W_t ) )
        }{
          e^{ \rho t }
        }
        +
        \smallint_0^t
        \tfrac{
          \bar{\V}( \Phi( x, s, W_s ) )
        }{
          e^{ \rho s }
        }
        \, ds
      \right)
    \right]
    \leq e^{c t + \V(x)}
    .
  \end{split}     
  \end{equation}
  Then it holds for all $ t \in [0,T] $ that
  \begin{equation}
  \label{eq:exp.mom.abstract}
    \E\!\left[
      \exp\!\left(
        \tfrac{
          \V( Y_t )
        }{
          e^{ \rho t }
        }
        +
        \smallint_0^t
        \tfrac{
          \mathbbm{1}_E(
            Y_{ \lfloor s \rfloor_{\theta} }
          )
          \,
          \bar{\V}(Y_s)
        }{
          e^{ \rho s }
        }
        \, ds
      \right)
    \right]
    \leq
      e^{ct} \, \E\!\left[e^{{\V}(Y_0)}\right]
    .
\end{equation}
\end{corollary}
\begin{proof}[Proof of Corollary~\ref{Cor:exp.mom.abstract}]
   We prove Corollary~\ref{Cor:exp.mom.abstract}
   through an application of Lemma~2.2 in Hutzenthaler et al.~\cite{HutzenthalerJentzenWang2014}.
  Assumption~\eqref{eq:exp.mom.abstract.assumption}
  implies that
  for all
  $ t \in (0, |\theta |_T] $, $ x \in H $
  it holds that
  \begin{equation}  \begin{split}
  \label{eq:exp.mom.abstract.assumption2}
    \E\!\left[
      \exp\!\left(
        \tfrac{
          \1_E(x) \, \V( \Phi( x, t, W_t ) )
        }{
          e^{ \rho t }
        }
        +
        \smallint_0^t
        \tfrac{
          \1_E(x)
          \,
          \bar{\V}( \Phi( x, s, W_s ) )
        }{
          e^{ \rho s }
        }
        \, ds
      \right)
    \right]
    \leq
    e^{
      c t
      +
      \1_E( x )
      \,
      \V(x)
    }
    .
  \end{split}
  \end{equation}
  Next note that
  \eqref{eq:process_dynamic}
  ensures that
  for all $ t \in (0, T] $
  it holds
  that
  \begin{equation}
  \begin{split}
   &
     \E\!\left[
       \exp\!\left(
         \tfrac{
           \1_E(Y_{ \llcorner t \lrcorner_{\theta} })
           \,
           \V(
             Y_t
           )
         }{
           e^{ \rho t }
         }
    +
    \smallint_{ \llcorner t \lrcorner_{\theta} }^{ t }
    \tfrac{
      \1_E(Y_{ \llcorner t \lrcorner_{\theta} })
      \,
      \bar{\V}(
        Y_s
      )
    }{
      e^{ \rho s }
    }
    \, ds
    \right)
    \Big| \, (Y_s)_{ s \in [0, \llcorner t \lrcorner_{\theta} ] }
    \right]
  \\ 
  & 
  =
     \E\bigg[
       \exp\!\bigg(
         \tfrac{
           \1_E(Y_{ \llcorner t \lrcorner_{\theta} })
           \,
           \V(S_{t- \llcorner t \lrcorner_\theta}
             \Phi(Y_{ \llcorner t \lrcorner_{\theta} }, t- \llcorner t \lrcorner_{\theta},
             W_{ t } - W_{\llcorner t \lrcorner_{\theta} })
           )
         }{
           e^{ \rho t }
         }
   \\
   &
   \quad
    +
    \smallint_{ \llcorner t \lrcorner_{\theta} }^{ t }
    \tfrac{
      \1_E(Y_{ \llcorner t \lrcorner_{\theta} })
      \,
      \bar{\V}(
      S_{s - \llcorner t \lrcorner_\theta}
        \Phi( Y_{ \llcorner t \lrcorner_{\theta} }, s - \llcorner t \lrcorner_\theta, W_s - W_{\llcorner t \lrcorner_\theta})
      )
    }{
      e^{ \rho s }
    }
    \, ds
    \bigg)
    \,
    \Big| \, (Y_s)_{ s \in [0, \llcorner t \lrcorner_{\theta} ] }
    \bigg]
  \\ 
  & 
  \leq
     \E\bigg[
       \bigg\{
       \exp\! \bigg(
         \tfrac{
           \1_E(Y_{ \llcorner t \lrcorner_{\theta} })
           \,
           \V(
             \Phi(Y_{ \llcorner t \lrcorner_{\theta} }, t- \llcorner t \lrcorner_{\theta},
             W_{ t} -W_{\llcorner t \lrcorner_{\theta} } )
           )
         }{
           e^{ \rho ( t - \llcorner t \lrcorner_{\theta} ) }
         }
   \\
   &
   \quad
    +\smallint_{ 0 }^{ t - \llcorner t \lrcorner_{\theta}}
    \tfrac{
      \1_E(Y_{ \llcorner t \lrcorner_{\theta} })
      \,
      \bar{\V}(
        \Phi( Y_{ \llcorner t \lrcorner_{\theta} }, s, 
        W_{\llcorner t \lrcorner_\theta + s} - W_{\llcorner t \lrcorner_\theta} )
      )
    }{
      e^{ \rho s }
    }
    \, ds
    \bigg)
    \bigg\}^{
      \exp( - \rho \llcorner t \lrcorner_{\theta} )
    }
    \,
    \Big| \, (Y_s)_{ s \in [0, \llcorner t \lrcorner_{\theta} ] }
    \bigg]
	.
  \end{split}
  \end{equation}
  Jensen's
  inequality,
  \eqref{eq:exp.mom.abstract.assumption2},
  and, e.g., 
  Lemma~\ref{lemma:conditional_expectation}
  hence 
  imply for all $ t \in (0, T] $ that
  \begin{equation} 
  \begin{split}
  &
     \E\!\left[
       \exp\!\left(
         \tfrac{
           \1_E(Y_{ \llcorner t \lrcorner_{\theta} })
           \,
           \V(
             Y_t
           )
         }{
           e^{ \rho t }
         }
    +
    \smallint_{ \llcorner t \lrcorner_{\theta} }^{ t }
    \tfrac{
      \1_E(Y_{ \llcorner t \lrcorner_{\theta} })
      \,
      \bar{\V}(
        Y_s
      )
    }{
      e^{ \rho s }
    }
    \, ds
    \right)
    \Big| \; (Y_s)_{ s \in [0, \llcorner t \lrcorner_{\theta} ] }
    \right]
    \\  
  & 
  \leq
       \bigg|
       \,
     \E\bigg[
       \exp \! \bigg(
         \tfrac{
           \1_E(Y_{ \llcorner t \lrcorner_{\theta} })
           \,
           \V(
             \Phi(Y_{ \llcorner t \lrcorner_{\theta} }, t- \llcorner t \lrcorner_{\theta},
             W_{t} - W_{\llcorner t \lrcorner_{\theta} })
           )
         }{
           e^{ \rho ( t - \llcorner t \lrcorner_{\theta} ) }
         }
   \\ 
   & 
   \quad
    +
    \smallint_{ 0 }^{ t - \llcorner t \lrcorner_{\theta}}
    \tfrac{
      \1_E(Y_{ \llcorner t \lrcorner_{\theta} })
      \,
      \bar{\V}( \Phi(Y_{ \llcorner t \lrcorner_{\theta} }, s, 
      W_{\llcorner t \lrcorner_\theta + s} - W_{\llcorner t \lrcorner_\theta} ) )
    }{
      e^{ \rho s }
    }
    \, ds
    \bigg)
    \,
    \Big| \; (Y_s)_{ s \in [0, \llcorner t \lrcorner_{\theta} ] }
    \bigg]
    \bigg|^{
      \exp( - \rho \llcorner t \lrcorner_{\theta} )
    }
  \\  
  & 
  \leq   
       \left|
       \left[
       e^{
         c \left( t - \llcorner t \lrcorner_{\theta} \right)
         +
           \1_E(Y_{ \llcorner t \lrcorner_{\theta} })
           \,
           \V(
             Y_{ \llcorner t \lrcorner_{\theta} }
           )
       }
       \right]_{\P, \mathcal{B}([0,\infty])}
    \right|^{
      e^{ - \rho \llcorner t \lrcorner_{\theta} }
    }
  \\
  & 
  \leq 
  \left[
       \exp\!\left( 
         c  \!\left( t - \llcorner t \lrcorner_{\theta} \right)
         +
         \tfrac{
           \1_E(Y_{ \llcorner t \lrcorner_{\theta} })
           \,
           \V(
             Y_{ \llcorner t \lrcorner_{\theta} }
           )
         }{
           e^{ \rho \llcorner t \lrcorner_{\theta} }
         }
    \right) 
    \right]_{\P, \mathcal{B}([0,\infty])} 
    .
  \end{split}
  \end{equation}
  This and~\eqref{eq:process_dynamic} 
  show  
  for all $ t \in (0,T] $ that
\begin{align} 
\nonumber
\label{eq:finale_estimate}
    &
    \E\!\left[
      \exp\!\left(
        - c t
        +
        \tfrac{
          \V( Y_t )
        }{
          e^{ \rho t }
        }
        +
        \smallint_0^t
        \tfrac{
          \1_E( Y_{ \lfloor r \rfloor_{\theta} } )
          \,
          \bar{\V}(Y_r)
        }{
          e^{ \rho r }
        }
        \,
        dr
      \right)
      \Big|
      \,
      ( Y_r )_{ r \in [0, \llcorner t \lrcorner_{\theta}] }
    \right]
  \\ 
  \nonumber
  &
  =
    \E\!\left[
      \exp\!\left(
        \tfrac{
          \1_E( Y_{ \llcorner t \lrcorner_{\theta} } )
          \,
          \V( Y_t )
        }{
          e^{ \rho t }
        }
        +
        \tfrac{
          \1_{ H \backslash E }( Y_{ \llcorner t \lrcorner_{\theta} } )
          \,
          \V( Y_t )
        }{
          e^{ \rho t }
        }
        +
        \smallint_{ \llcorner t \lrcorner_{\theta} }^t
        \tfrac{
          \1_E( Y_{ \llcorner t \lrcorner_{\theta} } )
          \,
          \bar{\V}(Y_r)
        }{
          e^{ \rho r }
        }
        \,
        dr
      \right)
      \Big|
      \,
      ( Y_r )_{ r \in [0, \llcorner t \lrcorner_{\theta}] }
    \right]
  \\ 
  \nonumber
  & 
  \quad 
  \cdot
   \exp\!\left(
     - c t
     +
     \smallint_0^{ \llcorner t \lrcorner_{\theta} }
     \tfrac{
       \1_E( Y_{ \lfloor r \rfloor_{\theta} } )
       \,
       \bar{\V}( Y_r )
     }{
       e^{ \rho r }
     }
     \, dr
   \right)
  \\ & =
    \E\!\left[
      \exp\!\left(
        \tfrac{
          \1_E( Y_{ \llcorner t \lrcorner_{\theta} } )
          \,
          \V( Y_t )
        }{
          e^{ \rho t }
        }
        +
        \tfrac{
          \1_{ H \backslash E }( Y_{ \llcorner t \lrcorner_{\theta} } )
          \,
          \V(S_{t-\llcorner t \lrcorner_\theta} Y_{ \llcorner t \lrcorner_{\theta} } )
        }{
          e^{ \rho t }
        }
        +
        \smallint_{ \llcorner t \lrcorner_{\theta} }^t
        \tfrac{
          \1_E( Y_{ \llcorner t \lrcorner_{\theta} } )
          \,
          \bar{\V}(Y_r)
        }{
          e^{ \rho r }
        }
        \,
        dr
      \right)
      \Big|
      \,
      ( Y_r )_{ r \in [0, \llcorner t \lrcorner_{\theta}] }
    \right]
  \\ & 
  \nonumber
  \quad 
  \cdot
   \exp\!\left(
     - c t
     +
     \smallint_0^{ \llcorner t \lrcorner_{\theta} }
     \tfrac{
       \1_E( Y_{ \lfloor r \rfloor_{\theta} } )
       \,
       \bar{\V}( Y_r )
     }{
       e^{ \rho r }
     }
     \, dr
   \right)
 \\& 
 \nonumber
 \leq
 \left[
      \exp\!\left(
        c
        \!
        \left( t - \llcorner t \lrcorner_{\theta} \right)
        +
        \tfrac{
          \1_E( Y_{ \llcorner t \lrcorner_{\theta} } )
          \,
          \V( Y_{ \llcorner t \lrcorner_{\theta} } )
        }{
          e^{ \rho \llcorner t \lrcorner_{\theta} }
        }
        +
        \tfrac{
          \1_{ H \backslash E }( Y_{ \llcorner t \lrcorner_{\theta} } )
          \,
          \V( Y_{ \llcorner t \lrcorner_{\theta} } )
        }{
          e^{ \rho t }
        }
        - c t
     +
     \smallint_0^{ \llcorner t \lrcorner_{\theta} }
     \tfrac{
       \1_E( Y_{ \lfloor r \rfloor_{\theta} } )
       \,
       \bar{\V}( Y_r )
     }{
       e^{ \rho r }
     }
     \, dr
   \right)
   \right]_{\P, \mathcal{B}([0,\infty])}
 \\ 
 \nonumber
 & 
 \leq
 \left[
      \exp\!\left(
        - c \llcorner t \lrcorner_{\theta}
        +
        \tfrac{
          \V( Y_{ \llcorner t \lrcorner_{\theta} } )
        }{
          e^{ \rho \llcorner t \lrcorner_{\theta} }
        }
     +
     \smallint_0^{ \llcorner t \lrcorner_{\theta} }
     \tfrac{
       \1_E( Y_{ \lfloor r \rfloor_{\theta} } )
       \,
       \bar{\V}( Y_r )
     }{
       e^{ \rho r }
     }
     \, dr
   \right)
   \right]_{\P, \mathcal{B}([0,\infty])}
   .
\end{align}
  Lemma~2.2 in Hutzenthaler et al.~\cite{HutzenthalerJentzenWang2014}
  and~\eqref{eq:finale_estimate}
  establish~\eqref{eq:exp.mom.abstract}.
  The proof of Corollary~\ref{Cor:exp.mom.abstract} is thus completed.
\end{proof}
\begin{remark}
Let  
$ (\U, \left< \cdot, \cdot \right>_\U, \left \| \cdot \right \|_U) $
be a separable
$ \R $-Hilbert space, let $ T \in (0,\infty) $,
let $ (\Omega, \mathcal{F}, \P) $ be a probability space,
and let $ W \colon [0,T] \times \Omega \to \U $ be an 
$ \operatorname{Id}_\U $-cylindrical $ \P $-Wiener process.
Then $ \dim(\U) < \infty $. 
\end{remark}
\subsection{A one-step estimate for exponential moments}
\label{subsection:main_lemma}
In this subsection we establish in~\eqref{eq:exp.mom.abstract.one-step}
in Lemma~\ref{l:exp.mom.abstract.one-step}
below
an appropriate
exponential
one-step estimate
for a general class
of one-step
approximation schemes.
This exponential one-step estimate
and Corollary~\ref{Cor:exp.mom.abstract}
above 
(cf. \eqref{eq:exp.mom.abstract.one-step} 
in Lemma~\ref{l:exp.mom.abstract.one-step} below
with~\eqref{eq:exp.mom.abstract.assumption} 
in Corollary~\ref{Cor:exp.mom.abstract} above)
will allow us to establish
exponential integrability properties
for some tamed approximation schemes
in Subsection~\ref{subsection:use_big_lemma} below.  
Lemma~\ref{l:exp.mom.abstract.one-step} below
extends
Lemma~2.7 in Hutzenthaler et al.~\cite{HutzenthalerJentzenWang2014}
from finite dimensional stochastic ordinary differential equations
to infinite dimensional stochastic partial differential
equations.
Our proof of Lemma~\ref{l:exp.mom.abstract.one-step}
exploits several elementary/well known
auxiliary lemmas
(see Lemmas~\ref{l:exp.Gauss}--\ref{lemma:limit_Lp} below). 
Lemma~\ref{l:exp.Gauss} below
is a straightforward extension of Lemma~2.5
in Hutzenthaler et al.~\cite{HutzenthalerJentzenWang2014}.
Lemma~\ref{lemma:Rademacher_theorem}
below
follows, e.g., from
Theorem~5.8.12 in Bogachev~\cite{Bogachev}.
%
%
\begin{lemma}
\label{l:exp.Gauss}
 Let
 $ \left( H, \left< \cdot, \cdot \right> _H, \left \| \cdot \right\|_H \right) $
 and 
 $ \left( \U, \left< \cdot, \cdot \right> _\U, \left\| \cdot \right \|_\U \right) $
 be separable $ \R $-Hilbert spaces,
 let 
 $ T \in (0,\infty) $, 
 $ B \in \HS(\U, H) $,
 let
 $ ( \Omega, \mathcal{ F }, \P ) $
 be a probability space,
 and
 let
$ (W_t)_{t\in [0,T]} $
be an
$ \operatorname{Id}_\U $-cylindrical 
 $ \P $-Wiener process.
  Then it holds
  for all $ t \in [0, T] $ that
\begin{equation}
\E \big[ \! \exp\!\big(  \| \smallint\nolimits_0^t B \, dW_s  \|_H \big) \big]
  \leq 2
  \exp\!\big(
    \tfrac{ t }{ 2 } \| B \|^2_{ \HS( \U, H ) }
  \big) .
  \end{equation}
\end{lemma}
\begin{lemma}
\label{lemma:compozitum_loc_Lip}
Let $ ( E, d_E ) $, 
$ ( F, d_F ) $, and 
$ ( G, d_G ) $ be metric spaces and let 
$ f \colon E \to F $ and 
$ g \colon F \to G $ be locally Lipschitz continuous functions. Then it holds that
$ g \circ f \colon E \to G $ 
is a locally Lipschitz continuous function.
\end{lemma}
\begin{proof}[Proof of Lemma~\ref{lemma:compozitum_loc_Lip}]
The assumption that $ f \colon E \to F $ 
is locally Lipschitz continuous implies
that for every $ x \in E $
there exist
real numbers
$ \delta_x, L_x \in (0,\infty) $
such that for all
$ x_1, x_2 \in E $
with
$ \max\{ d_E(x, x_1),
		d_E(x, x_2) \} < \delta_x $
it holds that
\begin{equation}
\label{eq:f_Lip}
 d_F(f( x_1), f(x_2) ) \leq
L_x d_E( x_1, x_2).
\end{equation}
Moreover, the assumption that 
$ g \colon F \to G $
is locally Lipschitz continuous
implies that for every $ y \in F $ 
there exist real numbers
$ \hat \delta_y, \hat L_y \in (0,\infty) $
such that for all
$ y_1, y_2 \in F $
with
$ \max\{ d_F(y, y_1 ),
		d_F(y, y_2 ) \} < \hat \delta_y $
		it holds that
\begin{equation}
\label{eq:g_Lip}
 d_G( g(y_1), g(y_2) ) 
		\leq \hat L_y d_F( y_1,  y_2 ) .
		\end{equation}
Next note that~\eqref{eq:f_Lip}
proves for all
$ x, x_1, x_2 \in E $
with
$ 
\max\{ d_E(x, x_1), d_E(x, x_2) \} 
<
\min\{
\delta_x,
\nicefrac{ \hat \delta_{f(x)}}{ L_x }
\}
$
that
$ 
\max\{ d_F( f(x), f(x_1) ),
d_F( f(x), f(x_2) ) \} \leq 
L_x 
\max\{
d_E( x, x_1 ),
d_E( x, x_2 )
\}
<
\hat \delta_{f(x)} $.
Combining this, \eqref{eq:f_Lip}, and~\eqref{eq:g_Lip}
ensures that  
for all $ x, x_1, x_2 \in E $
with
$ 
\max\{ d_E(x, x_1), d_E(x, x_2) \} 
<
\min\{
\delta_x,
\nicefrac{ \hat \delta_{f(x)}}{ L_x }
\}
$ 
it holds
that
\begin{equation}
d_G( g ( f ( x_1 ) ), g ( f ( x_2 ) ) )
\leq 
\hat L_{f(x)} 
d_F( f ( x_1 ), f( x_2 ) )
\leq
  \hat L_{f(x)} L_x d_E( x_1, x_2 ) 
.
\end{equation}
The proof of Lemma~\ref{lemma:compozitum_loc_Lip}
is thus completed.
\end{proof}
\begin{lemma}
\label{lemma:Loc_Lip}
Let $ ( V, \left\| \cdot \right\|_V ) $ be a
normed $ \R $-vector space with $ \#_V > 1 $
and let 
$ c \in [1,\infty) $,
$ n \in \N_0 $, 
$ U \in \mathcal{C}_c^{n+1}(V, \R) $,
$ i \in \N_0 \cap [0,n] $.
Then it holds that 
$ U^{(i)} $ is locally Lipschitz continuous.
\end{lemma}
\begin{proof}[Proof of Lemma~\ref{lemma:Loc_Lip}]
The fact that $ U $ is continuous proves that for 
every $ (x, \varepsilon) \in V\times (0,\infty) $
there exists 
a real number
$ \delta_{x, \varepsilon} \in (0,\infty) $ such 
that for all
$ v \in V $
with $ \| x - v \|_V < \delta_{x, \varepsilon} $
it holds that
$ | U (x) - U (v) | < \varepsilon $.
This and the triangle inequality prove that for all 
$ x, x_1, x_2 \in V $ 
with 
$ \max\{ \| x-x_1\|_V, \| x- x_2\|_V \} <
\delta_{x, 1} $
it holds that
\begin{equation}
\begin{split}
&
\| U^{(i)}( x_1 ) - U^{(i)}( x_2 ) \|_{L^{(i)}(V, \R)}
\leq
c
\| x_1 - x_2 \|_V
( 1 + \sup\nolimits_{r\in [0,1] } | U ( r x_1 + (1-r) x_2) | )
\\
&
\leq
c
\| x_1 - x_2 \|_V
( 1 + | U(x) | + \sup\nolimits_{r\in [0,1]}
| U( rx_1 + (1-r)x_2) - U(x) |)
\\
&
\leq
c \| x_1 - x_2 \|_V
( 2 + | U(x) | )
.
\end{split}
\end{equation}
The proof of Lemma~\ref{lemma:Loc_Lip}
is thus completed.
\end{proof}
\begin{lemma}
	\label{lemma:Rademacher_theorem}
	Let $ a \in \R $, $ b \in (a,\infty) $
	and let
	$ f \in \mathcal{C}([a,b], \R) $
	be a locally Lipschitz continuous function.
	Then 
	\begin{enumerate}[(i)]
		\item \label{item:differentiability} 
		it holds that $ \{ s \in [a,b] \colon f 
		\text{ is differentiable at } s\} \in \mathcal{B}(\R) $,
		\item \label{item:diff_ae} 
		it holds that $ \mu_\R( [a,b] \backslash 
		\{ s \in [a,b] \colon f 
		\text{ is differentiable at } s\} ) = 0 $,
		and
		\item \label{item:absolutely_continuous}
		it holds that $ f $ is absolutely
		continuous.
	\end{enumerate}
\end{lemma}
\begin{lemma}
\label{lemma:equivalent0}
Let $ (H, \left< \cdot, \cdot \right>_H, \left \| \cdot \right \|_H) $
be an $ \R $-Hilbert space
with $ \#_H > 1 $
 and let $ c  \in [1, \infty) $,
$ n \in \N_0 $,
$ x, y\in H $,
$ V \in \mathcal{C}^{n+1}_{c}(H, [0, \infty)) $.
Then
\begin{enumerate}[(i)]
	\item \label{eq:bounded_derivative} it holds
for all 
$ t \in \{ s \in [0, 1] \colon 
 	\R \ni u \mapsto V(x+uy) \in \R\,\text{is differentiable at } s
 	\} $ 
that
$ | \tfrac{\partial}{\partial t} V(x+ty)|\leq c\|y\|_H(1+V(x+ty))^{1-\nicefrac{1}{c}} $
and
\item \label{eq:bounded_derivatives} it holds for all
$ i \in \N \cap [0,n] $,
$ z_1,\ldots, z_i \in H $,
$ t \in \{ s \in [0, 1] \colon 
\R \ni u \mapsto V^{(i)}(x+uy)(z_1,\ldots, z_i) \in \R \text{ is differentiable at } s
\} $  
that
\begin{equation}
\left|
\tfrac{\partial}{\partial t}
\!\left(V^{(i)} (x+ty) (z_1,\ldots, z_i)\right)
\right|
\leq
c \| z_1\|_H \cdots \| z_i\|_H  \| y \|_H
( 1 + V(x +ty) )^{1-\nicefrac{1}{c}}.
\end{equation}
\end{enumerate}
\end{lemma}
\begin{proof}[Proof of Lemma~\ref{lemma:equivalent0}]
First of all,
note that
the assumption that $ V \in \mathcal{C}_c^{n+1}(H, [0,\infty)) $
ensures that for all 
$ t \in [0, 1] $, 
$ h \in \R $
it holds
that
\begin{equation}
\label{eq:basic_estimate}
\begin{split}
|V(x+ty) - V(x+(t+h)y)|
&
\leq
c
|h| \|y\|_H [ 1 + \sup\nolimits_{r\in [0,1]}V(x + (t + (1 - r) h) y)]^{1-\nicefrac{1}{c}}
\\
&
=
c
|h| \|y\|_H [ 1 + \sup\nolimits_{r\in [0,1]}V(x + (t + r h) y)]^{1-\nicefrac{1}{c}}.
\end{split}
\end{equation}
Next observe that Lemma~\ref{lemma:Loc_Lip} ensures for all 
$ t \in [0, 1] $ 
that 
\begin{equation}
\label{eq:lim_estimate}
  \limsup\nolimits_{ ( \R \backslash \{0\} )\ni h \to 0 } | \sup\nolimits_{ r \in [0,1] } V( x + ( t + r h ) y ) - V( x + t y ) | = 0 .
  \end{equation}
Combining this with~\eqref{eq:basic_estimate} proves~\eqref{eq:bounded_derivative}.
In the next step observe that for all 
$ i \in \N \cap [0, n] $,
$ z_1,\ldots, z_i\in H\backslash \{0\} $,
$ t \in [0, 1] $, $ h \in \R $  
it holds that
\begin{equation}
\begin{split}
&
\tfrac{
	 |
	 V^{(i)}(x+ty) (z_1,\ldots, z_i)
	-
	 V^{(i)}(x+(t+h)y) (z_1,\ldots, z_i) 
	|
}{
\|z_1\|_H\cdots \|z_i\|_H
}
\leq
	\| V^{(i)}(x+ty)
	-
	V^{(i)}(x+(t+h)y) \|_{L^{(i)}(H,\R)}
\\
&
\leq
c
|h| \| y\|_H
[ 1 + \sup\nolimits_{r\in [0,1]}V(x+ (t+ (1-r) h) y) ]^{1-\nicefrac{1}{c}}
\\
&
=
c
|h| \| y\|_H
[ 1 + \sup\nolimits_{r\in [0,1]}V(x+ (t+ r h) y) ]^{1-\nicefrac{1}{c}}
.
\end{split}
\end{equation}
This 
and~\eqref{eq:lim_estimate} 
establish~\eqref{eq:bounded_derivatives}.
The proof of Lemma~\ref{lemma:equivalent0} is thus completed.
\end{proof}
\begin{lemma}
\label{lemma:equivalent}
Let $ (H, \left< \cdot, \cdot \right>_H, \left \| \cdot \right \|_H) $
be an $ \R $-Hilbert space
with $ \#_H > 1 $
 and let $ c \in [1,\infty) $,
$ x, y\in H $, $V \in \mathcal{C}^1_{c }(H, [0, \infty)) $.
Then  
$ 1 + V(x+y)\leq  2^{c-1}(1+V(x)+\|y\|_H^c) $.
\end{lemma}
\begin{proof}[Proof of Lemma~\ref{lemma:equivalent}]
Throughout this proof 
let $ f \colon \R \to \R $
be the function which satisfies 
for all $ t \in \R $ that
$ f(t)=  V (x+ty) $.
Next observe that 
item~\eqref{eq:bounded_derivative}
in Lemma~\ref{lemma:equivalent0}
implies  
for all
$ t \in \{ s \in [0,1] \colon 
 	\R \ni u \mapsto f(u) \in \R \text{ is differentiable at }s
 	\} $ that
\begin{equation}
\begin{split}
|
\tfrac{\partial}{\partial t}
( 1+ f(t) )|
\leq
c \| y \|_H
( 1 + f(t) )^{1-\nicefrac{1}{c}}
. 
\end{split}
\end{equation}
Lemma~2.11 in Hutzenthaler \& Jentzen~\cite{HutzenthalerJentzenMemoires2015},
Lemma~\ref{lemma:compozitum_loc_Lip},
Lemma~\ref{lemma:Loc_Lip}, and Lemma~\ref{lemma:Rademacher_theorem} 
hence prove for all $ t \in [0,1] $ that
\begin{equation}
\begin{split}
\label{eq:function_estimate}
1 + f (t)
\leq
2^{c-1}
\left[
 1 + f(0) 
+
  t^c \|  y\|_H^c
\right]
.
\end{split}
\end{equation}
This implies that 
$
1 + V(x+y)
\leq
2^{c-1}
\left[
1 + V(x)
+ 
\| y\|_H^c
\right]
.
$
The proof of Lemma~\ref{lemma:equivalent} is thus completed.
\end{proof}
\begin{lemma}
\label{lemma:another_growth_estimate}
Let $ (H, \left< \cdot, \cdot \right>_H, \left \| \cdot \right \|_H) $
be an $ \R $-Hilbert space 
with $ \#_H > 1 $
and let $ c \in [1,\infty) $, $ n\in \N_0 $, 
$ x, y\in H $, 
$ V \in \mathcal{C}^{n+1}_{ c}(H, [0, \infty)) $.
Then 
\begin{equation}
\max\nolimits_{i \in \{0, 1, \ldots, n\}}
\|\V^{(i)}(x) - \V^{(i)}(y)  
\|_{L^{(i)}(H, \R)}
\leq
c\,
2^{c-1} 
\| x-y\|_H
\!
\left(
1+ \V(x) 
+
\| x-y\|_H^{c-1} 
\right) .
\end{equation}
\end{lemma}
\begin{proof}[Proof of Lemma~\ref{lemma:another_growth_estimate}]
Note that
Lemma~\ref{lemma:compozitum_loc_Lip},
Lemma~\ref{lemma:Loc_Lip}, Lemma~\ref{lemma:Rademacher_theorem},
item~\eqref{eq:bounded_derivative} 
in Lemma~\ref{lemma:equivalent0},
 Lemma~\ref{lemma:equivalent},
 and the fact that $ \forall\, r \in [0,1] $, 
 $ a\in [1,\infty) $, 
 $ b \in [0,\infty) \colon (a+b)^r \leq a + b^r $
 prove that
\begin{equation}
\label{derivativeUu1}
\begin{split}
& 
 | \V(y) - \V(x) |  
\leq  
\int_{ 
 	\{ s \in [0,1] \colon
 	\R \ni u \mapsto V(x+u(y-x)) \in 
 	\R  \text{ is differentiable at } s
 	\}
 	 }
\big|
\tfrac{\partial}{
\partial r} 
\V 
( x + r(y-x) )   
\big|
\,dr
\\
&
\leq
c \| x-y\|_H 
\!
\int_0^1 \! [ 1 + \V(x + r(y-x)) ]^{1-\nicefrac{1}{c}} 
\, dr
\leq
c\,
2^{c-1}  
\| x - y\|_H
\left(
  1+ \V(x)  
+
\|y-x\|_H^{c-1}
\right) 
\!
.
\end{split}
\end{equation}
Moreover, 
Lemma~\ref{lemma:compozitum_loc_Lip},
Lemma~\ref{lemma:Loc_Lip},
Lemma~\ref{lemma:Rademacher_theorem},
item~\eqref{eq:bounded_derivatives} 
in
Lemma~\ref{lemma:equivalent0},
 Lemma~\ref{lemma:equivalent},
 and the fact that $ \forall\, r \in [0,1] $, 
 $ a\in [1,\infty) $, 
 $ b \in [0,\infty) \colon (a+b)^r \leq a + b^r $
 ensure that
 for all 
  $ i \in \N \cap [0,n] $,
  $ z_1, \ldots, z_i\in H\backslash\{0\} $ 
  it holds that
\begin{equation}
\begin{split}
&
\tfrac{ 
 |   (\V^{(i)}(y) - \V^{(i)}(x) )(z_1,\ldots, z_i) | 
}
{ 
\| z_1\|_H\cdots \|z_i\|_H
}
\leq
\tfrac{1}{\|z_1\|_H\cdots \| z_i\|_H}
\\
&
\cdot
\!
 \int_{ 
 	\{ s \in [0,1] \colon
 	\R \ni u \mapsto V^{(i)}(x+u(y-x)) (z_1,\ldots, z_i) \in 
 	\R \text{ is differentiable at } s
 	\}
 	 }
 	 \!\!
\big|
\tfrac{\partial}{
\partial r}
\big(
\V^{(i)} 
( x + r(y-x) ) (z_1,\ldots, z_i)
\big)
\big|
\,dr
\\
&
\leq
c \| x-y\|_H 
\!
\int_0^1 \! [ 1 + \V(x + r(y-x)) ]^{1-\nicefrac{1}{c}} 
\, dr
\leq
c\,
2^{c-1}  
\| x - y\|_H
\left(
  1+ \V(x)  
+
\|y-x\|_H^{c-1}
\right) 
\!
.
\end{split}
\end{equation}
Combining this with~\eqref{derivativeUu1}
completes the proof of Lemma~\ref{lemma:another_growth_estimate}.
\end{proof}
\begin{lemma}
	\label{lemma:equivalent000}
	Let $ (H, \left< \cdot, \cdot \right>_H, \left \| \cdot \right \|_H) $
	be an $ \R $-Hilbert space 
	with $ \#_H > 1 $
	and let 
	$ c\in [1, \infty) $,
	$ n\in \N_0 $,
	$ x\in H $,  
	$ V \in \mathcal{C}^{n+1}_{c}(H, [0, \infty)) $.
	Then 
	$ \max_{ i\in \{0,1, \ldots, n \} } \| V^{(i)} (x) \|_{L^{(i)}(H,\R)} \leq
	c\, ( 1 + V(x) ) $.
\end{lemma}
\begin{proof}[Proof of Lemma~\ref{lemma:equivalent000}]
	Lemma~\ref{lemma:another_growth_estimate} 
	proves for all $ y\in H $, $ t\in [0,1] $ 
	that
\begin{equation}
|\V (tx + (1-t)y) - \V (x) |
\leq
c\,
2^{c-1}
\| x-y\|_H
\!
\left(
1+ \V(x) 
+
\| x-y\|_H^{c-1} 
\right)
.
\end{equation}	
	This implies for all $ \varepsilon \in (0,\infty) $,
	$ y \in H $ 
with $ \|x-y\|_H < \varepsilon $ 
that
 \begin{equation}
  |\sup\nolimits_{r\in [0,1]}\V (rx + (1-r)y) - \V (x) |
 \leq
 c\,
 2^{c-1}
 \varepsilon 
 \left(
 1+ \V(x) 
 +
 \varepsilon^{c-1} 
 \right)
 .
 \end{equation}	
	Hence, we obtain that
	\begin{equation}
	\label{eq:continuous}
	\limsup\nolimits_{y\to x}
	\left|
	\sup\nolimits_{r\in [0,1]}
	V(rx + (1-r)y)
	-V(x)
	\right|
	= 0.
	\end{equation}
 Moreover, the assumption that    
	$ V \in \mathcal{C}_c^{n+1}(H, [0,\infty)) $
	assures for all 
	$ i \in \N_0 \cap [0,n] $, $ y \in H $
	that
	\begin{equation}
	\| V^{(i)}(x)- V^{(i)}(y)\|_{L^{(i)}(H, \R)}\leq
	c \| x -y \|_H
	(1 + \sup\nolimits_{r\in [0,1]} V(rx + (1-r)y))^{1-\nicefrac{1}{c}}.
	\end{equation}
Combining this with~\eqref{eq:continuous}
completes
	the proof of Lemma~\ref{lemma:equivalent000}.
\end{proof}
\begin{lemma}
\label{lemma:limit_Lp}
Let $ (V, \left \| \cdot \right \|_V) $
be a normed $ \R $-vector space,
let $ (\Omega, \mathcal{F}, \mu) $ be a finite measure space,
and let $ X \in \mathcal{M}(\mathcal{F}, \mathcal{B}(V)) $.
Then  
$ \liminf_{p\to \infty} \| X\|_{\mathcal{L}^p(\mu; V)} 
=
\limsup_{p\to \infty} \| X\|_{\mathcal{L}^p(\mu; V)}
= \| X \|_{\mathcal{L}^\infty(\mu; V)} $.
\end{lemma}
\begin{proof}[Proof of Lemma~\ref{lemma:limit_Lp}]
Throughout this proof assume
w.l.o.g.\,that $ \|X\|_{\mathcal{L}^\infty(\mu; V)} > 0 $ 
and  
let $ A_\delta \subseteq \Omega $,
$ \delta \in (0,\infty) $,
be the sets with the property that for all 
$ \delta \in (0,\infty) $ 
it holds that
$ A_\delta = \{ \omega \in \Omega \colon \|X(\omega)\|_V \geq   \delta\} $. 
Next observe that for all $ p \in (0,\infty) $, 
$ \delta \in (0, \|X\|_{\mathcal{L}^\infty(\mu; V)} ) $ it holds that
\begin{equation}
\label{eq:limi_estimate}
\| X \|_{\mathcal{L}^p(\mu; V)}
\geq 
\| X \1_{A_\delta} \|_{\mathcal{L}^p(\mu; V)}
\geq
\|  \delta \, \1_{A_\delta} \|_{\mathcal{L}^p(\mu; \R)}
=
\delta \, [\mu(A_\delta)]^{\nicefrac{1}{p}}.
\end{equation}
Hence, we obtain for all 
$ \delta \in (0, \|X\|_{\mathcal{L}^\infty(\mu; V)} ) $ that
$ \liminf_{p\to \infty} \| X \|_{\mathcal{L}^p(\mu; V)} \geq \delta $.
This shows that
$ \liminf_{p\to \infty} \| X\|_{\mathcal{L}^p(\mu; V)}
\geq \| X \|_{\mathcal{L}^\infty(\mu; V)} $.
Moreover, note that for all $ p\in (0,\infty), $ $ q \in (0,p) $
it holds that
\begin{equation}
\begin{split}
&
\| X \|_{\mathcal{L}^p(\mu;V)} 
=
\left(
\int_\Omega  
\| X(\omega) \|_V^q
\| X(\omega) \|_V^{p-q}\,\mu(d\omega)
\!
\right)^{\!\!\frac{1}{p}}
\leq
\| X \|_{\mathcal{L}^\infty(\mu; V)}^{(p-q)/p}
\|  
X \|_{\mathcal{L}^q(\mu; V)}^{q/p}. 
\end{split}
\end{equation} 
This implies that
$ \limsup_{p\to \infty} \| X \|_{\mathcal{L}^p(\mu;V)} 
\leq
\| X \|_{\mathcal{L}^\infty(\mu; V)} $.
The
proof of Lemma~\ref{lemma:limit_Lp}
is thus completed.
\end{proof}
\begin{lemma}
\label{l:exp.mom.abstract.one-step}
Let 
$ \left( H, \left< \cdot, \cdot \right> _H, \left \| \cdot \right\|_H \right) $
and 
$ \left( \U, \left< \cdot, \cdot \right> _\U, \left\| \cdot \right \|_\U \right) $
be separable $ \R $-Hilbert spaces with $ \#_H > 1 < \#_U $,
let
$ \varsigma, h \in (0,\infty) $, 
$ c, \ga, \gb \in [1,\infty) $,
$ \rho, \delta \in [0,\infty) $, 
$ \gc \in [0,\nicefrac{1}{2}] $,
$ x \in H $,
$ F \in \mathcal{M} \big( \mathcal{B}( H), \mathcal{B}(H) \big) $,
$ B \in \mathcal{M} \big( \mathcal{B}( H), \mathcal{B}(\HS(\U,H)) \big) $,
$ \bar{\V}
\in \mathcal{C}(H, \R) $,
$ \V \in \mathcal C^3_c(   H, [0,\infty) ) $,
$ \Phi \in
\mathcal{C}^{  1, 2 }( [0,h] \times \U,  H ) $,
let
$ ( \Omega, \mathcal{ F }, \P) $
be a probability space,
let
$ W \colon [0,h] \times \Omega \to \U $
be an
$ \operatorname{Id}_\U $-cylindrical
$ \P $-Wiener process 
with continuous sample paths,
   assume for all
$ r \in [1, \infty) $, 
  $ s \in (0, h] $,  
  $ y, z \in  H $ 
  that
  $ | \bar{ \V }(y) - \bar{ \V }( z ) |
  \leq
    c
    \left( 1 +
      | \V (y) |^{ \ga } +
      | \V (z) |^{ \ga }
    \right)
    \| y - z \|_H $,
  $ | \bar{ \V }(y) |
  \leq
    c
    \left( 1 +
    \left| \V ( y ) \right|^{ \gb } \right) $, 
$ V(x)  \leq ch^{-\varsigma} $,
     $ \max\{ \| F(x)\|_H, \| B(x) \|_{\HS(\U,H)} \}\leq c h^{-\delta} $,
     $ \Phi( 0, 0 ) = x $, 
 and 
  \begin{equation}
  \label{eq:generator.exponential}
    ( \mathcal{G}_{ F, B } \V )(x)
    +
    \tfrac{1}{2} \left\|B(x)^{*}(\nabla \V)(x)\right\|_{\U}^2
    +
    \bar{\V}(x) 
  \leq
    \rho \V(x),
    \end{equation}
    \begin{equation}
    \max \! \big\{  
     \|
       (
        \tfrac{\partial}{\partial s} \Phi
       )(  s , W_s ) -
      F(x)
     \|_{
      \mathcal{L}^4( \P; H )
    },
     \|
       (
        \tfrac{ \partial }{ \partial y } \Phi
       )(  s, W_s ) - B(x)
     \|_{
      \mathcal{L}^8(
        \P; \HS( \U, H )
      )
    } 
    \big\}
    \leq c s^{ \gc } ,
  \label{eq:assumption.on.phi.sigma}
  \end{equation}
  \begin{equation}
    \big\|
      \smallsum\nolimits_{u\in \mathbb{\U}}
    \big(
       (
      \tfrac{\partial^2}{\partial y^2}
      \Phi
      )
      (s, W_s)
      \big)(u,u)
    \big\|_{\mathcal{L}^{4}(\P;H)}
     \leq c s^{ \gc },
     \label{eq:assumption.on.phi.sigma2}
     \end{equation}
  \begin{equation}
    \left\|
      \Phi(s,W_s) - x
    \right\|_{\mathcal{L}^r(\P; H)}
     \leq
     c 
     \min\!\big\{
      1,
      \big\|
         \| F(x)  \|_H  s    
        +
            \| 
          B(x) W_s
            \|_H
			 \big\|_{ \mathcal{L}^r( \P; \R ) }    
     \big\}
     \label{eq:assumption.on.phi.increment}
.
  \end{equation}
  Then it holds for all  
  $ t \in (0, h ] $ 
  that
  \begin{equation}
  \begin{split}
  \label{eq:exp.mom.abstract.one-step}
&
    \E\!\left[
      \exp\!\left(
        \tfrac{
        \V( \Phi(t,W_t) )}{e^{  \rho t }}
        +
        \smallint_0^t
        \tfrac{
        \bar{\V}( \Phi(s, W_s ) )}{e^{ \rho s }}
        \, ds
      \right)
    \right] 
\\
&
  \leq
    e^{ \V(x) }      
\left[
1
+
    \int_0^t
 \exp\!\left(
      \frac{
        [2c]^{
         (2c+6) \gb 
        }
        s
      }{ 
          \min\{s^{  
          2\delta + \max\{2, \gb\}   \varsigma 
        }, 1\} 
      }
  \right)
%
%
  \frac{  
    [2c]^{
    (2c + 13) \ga 
    }
    \left[
      \max\{s, \rho, 1\}
    \right]^4  
  }{
    \left[
      \min\{ s, 1\}
          \right]^{
    \varsigma + \varsigma \ga +   4 \delta  - \gc
    }
  }
  \, ds 
  \right] 
    .
  \end{split}
  \end{equation}
\end{lemma}
%
%
%
%
\begin{proof}[Proof of Lemma~\ref{l:exp.mom.abstract.one-step}]
Throughout this proof
let $ \mathbb{\U} \in \mathcal{P}(\U) $ be an orthonormal basis of $ \U $,
let $ Y \colon [0, h]\times \Omega \to H $ be the
function which satisfies
for all $ s \in [0,h] $ 
that
   $ Y_s = \Phi(s, W_s) $,
   and  
let $ \tau_{n}\colon \Omega \to [0,h] $, $ n\in \N $,
  be the functions which satisfy for all $ n \in \N $ that 
$ \tau_n = \inf ( \{s\in[0,h]\colon 
 \| W_s \|_\U > n \}\cup\{h\}  ) $.  
Next observe that
$ \P( Y_0 = x ) =
\P ( \Phi( 0, W_0 ) = x )
=
\P ( \Phi( 0, 0 ) = x ) = 1 $.
   It\^o's formula hence implies 
  for all $ t \in [0,h] $  
  that
  \begin{align}
  \label{eq:after.Ito}
  \nonumber
    &
    \left[
    \exp\!\left(e^{-\rho t}\V(Y_t)+\int_0^t e^{-\rho r}\bar{\V}(Y_r)\,dr\right)
    - e^{\V(x)}
    \right]_{\P, \mathcal{B}(\R)}
\\
\nonumber
&
=
\int_0^t
\left[
    \exp\!\left( e^{-\rho s}\V(Y_s)+\int_0^s e^{-\rho r}\bar{\V}(Y_r)\,dr\right)
    e^{-\rho s}
      \V'(Y_s)\big(\tfrac{\partial}{\partial y}\Phi\big)
         (s,W_s)\right]dW_s
    \\
    &
    + 
    \Bigg[
    \int_0^t  
    \exp\!\left(\! e^{-\rho s}\V(Y_s)+ \int_0^s  e^{-\rho r}\bar{\V}(Y_r)\,dr \right) 
    e^{-\rho s}
    \bigg[ 
      \bar{\V}(Y_s)
      -\rho \V (Y_s)
      \\
      \nonumber
      &
      + \V'(Y_s)\big( \tfrac{\partial}{\partial s}\Phi \big)
         (s,W_s)
    %
      +\tfrac{1}{2}
      \tr_U\!
      \left(    
      \big[\big(\tfrac{\partial}{\partial y}\Phi\big)(s,W_s)\big]^*
       (\Hess\V)(Y_s) 
      \big(
      \tfrac{\partial}{\partial y}\Phi\big)(s,W_s)
      \right)
    \\
    & 
    \nonumber
    +
    \tfrac{1}{2e^{\rho s}}
\big\|
    \big[
     \big(\tfrac{\partial}{\partial y}\Phi \big)(s,W_s)\big]^*
    (\nabla \V)(Y_s)
       \big\|_\U^2
      +  
      \tfrac{ 1 }{ 2 } 
      \smallsum\limits_{ u\in \mathbb{\U} }  
      \V'( Y_s ) 
      \big(
      \big( \tfrac{ \partial^2 }{ \partial y^2 } \Phi \big)( s, W_s )
      \big)(u,u)
    \bigg]
    \,ds
    \Bigg]_{\P, \mathcal{B}(\R)}
    .
  \end{align}
Therefore, we obtain
  for all
  $ t \in [0, h ] $, 
  $ n \in \N $ 
  that
  \begin{equation}  
   \begin{split}
    &
    \E\!\left[
      \exp\!\left(
        e^{ - \rho ( t \wedge \tau_n ) }
        \V( Y_{ t \wedge \tau_n } )
        +
        \int_0^{ t \wedge \tau_n }
        e^{ - \rho r }
        \bar{\V}( Y_r )
        \, dr
      \right)
    \right]
    -
    e^{ \V(x) }
  \\ 
  & 
  =
   \E\Bigg[
    \int_0^{t\wedge\tau_n} 
    \exp\!\left(e^{-\rho s} \V(Y_s)
                   +\int_0^{s}e^{-\rho r}\bar{\V}(Y_r)\,dr
             \right)
    e^{-\rho s}
     \bigg(
      \bar{\V}(Y_s)
      -\rho \V(Y_s)
    \\
    & 
      + \V'(Y_s)\big(\tfrac{\partial}{\partial s}\Phi\big)
         (s,W_s)
      +\tfrac{1}{2}
      \tr_U
      \!
      \left(    
      \big[
      \big(
      \tfrac{\partial}{\partial y}\Phi\big)(s,W_s)\big]^*
      (\Hess\V) (Y_s) 
      \big(
      \tfrac{\partial}{\partial y}\Phi\big)(s,W_s)
      \right)
    \\
    &  
    +
    \tfrac{1}{2e^{\rho s}}
\big\|
    \big[
    \big(\tfrac{\partial}{\partial y}\Phi\big)(s,W_s)\big]^*
    (\nabla \V)(Y_s)
       \big\|_\U^2
      +  
      \tfrac{ 1 }{ 2 } 
      \sum_{ u\in \mathbb{\U} }
        \V'( Y_s ) 
        \big(
      \big( \tfrac{ \partial^2 }{ \partial y^2 } \Phi \big)(  s, W_s )
      \big)
       (u,u)
    \bigg)  
    \, 
    ds
    \Bigg]
    .
  \end{split}    
   \end{equation}
  This implies for all $ t \in [0, h ] $, $ n \in \N $ that
  \begin{equation}  
  \begin{split}
    &
    \E\!\left[\exp\!\left(e^{-\rho (t\wedge\tau_n)}\V(Y_{t\wedge\tau_n})
                   +\int_0^{t\wedge\tau_n}e^{-\rho r}\bar{\V}(Y_r)\,dr
    \right)\right]
    -e^{ \V (x)}
 \\ 
 & 
 =
   \E\Bigg[
    \int_0^{t\wedge\tau_n} 
    \exp\!\left(e^{-\rho s}\V(Y_s)
                   +\smallint_0^{s}e^{-\rho r}\bar{\V}(Y_r)\,dr
    \right)
    e^{-\rho s}
    \bigg(
      (\mathcal{G}_{F,B}\V)(x)
      \\
      &
      \quad
      +\tfrac{1}{2 e^{ \rho s} }
      \left\| B(x)^* (\nabla \V)(x)\right\|_\U^2 
      +\bar{ \V }(x)
      -\rho \V (x)
      + \bar{ \V }(Y_s)-\bar{ \V }(x)
       \\
       & 
       \quad
      -\rho( \V (Y_s)- \V (x))
      +\tfrac{1}{2}
       \tr_U\!
      \left( 
      \big[
      \big(
      \tfrac{\partial}{\partial y}\Phi\big)(s,W_s)\big]^*
      (\Hess\V) (Y_s) 
      \big(
      \tfrac{\partial}{\partial y}\Phi\big)(s,W_s)
      \right)
      \\
      &
      \quad
       + \V'(Y_s)
       \big(
       \tfrac{ \partial }{ \partial s } \Phi
       \big)(s,W_s)
       -
       \V'(x) F(x)
       -
       \tfrac{1}{2}
	   \tr_U\!   
       \big(
       B(x)^*
 	   (\Hess \V) (x) 
 	   B(x)
       \big)
    \\
    & 
    \quad
      +
      \tfrac{1}{2 e^{ \rho s}}
      \big\|
  		\big[
         \big(\tfrac{\partial}{\partial y}\Phi\big)(s,W_s)
         \big]^*
         (\nabla \V )(Y_s)
       \big\|_{\U}^2
      -
      \tfrac{1}{2 e^{\rho s} } \left\|
      B(x)^*
         (\nabla \V)(x)
       \right\|_\U^2
      \\
      &
     \quad
      +
      \tfrac{1}{2}
      \sum\nolimits_{u\in \mathbb{\U} }     
      \V'( Y_s )
      \big( 
      \big( \tfrac{\partial^2}{\partial y^2} \Phi \big)( s, W_s )
      \big)
      (u, u)
    \bigg) 
    \,
    ds
    \Bigg].
\end{split}
\end{equation}
Assumption~\eqref{eq:generator.exponential} hence proves
for all $ t \in [0,h] $,  
$ n \in \N $ that
\begin{equation}
\begin{split}
&
    \E\!\left[\exp\!\left(e^{-\rho (t\wedge\tau_n)}\V(Y_{t\wedge\tau_n})
                   +\int_0^{t\wedge\tau_n}e^{-\rho r}\bar{\V}(Y_r)\,dr
    \right)\right]
    -e^{ \V (x)}  
\\ 
& 
\leq
   \E\Bigg[
    \int_0^{t\wedge\tau_n} 
    \exp\!\left(e^{-\rho s} \V (Y_s)
                   +\smallint_0^{s}e^{-\rho r}\bar{\V}(Y_r)\,dr
        \right)
    \\
    & 
    \quad
    \cdot\bigg(
      \left|\bar{\V}(Y_s)- \bar{\V}(x)\right|
      +\rho | \V(Y_s)- \V(x)|
      +\left| \V'(Y_s)\big(\tfrac{\partial}{\partial s}\Phi\big) (s,W_s)
            - \V'(x)F(x)\right|
    \\
    & 
    \quad
      +\tfrac{1}{2}
      \bigg|
      \tr_U\!
      \left(     
      \big[
      \big(
      \tfrac{\partial}{\partial y}\Phi\big)(s,W_s)\big]^*
       (\Hess \V )(Y_s) 
      \big(
      \tfrac{\partial}{\partial y}\Phi\big)(s,W_s)
      \right)
     \\
     &
     \quad
       -
	   \tr_U\!      
       \big(
       B(x)^*
 	    (\Hess \V)(x) 
 	   B(x)
       \big) 
       \bigg|
       +  
       \tfrac{ 1 }{ 2 } 
       \Big|
       \sum\nolimits_{ u\in \mathbb{\U} } 
         \V'( Y_s ) 
         \big(
       \big( \tfrac{ \partial^2 }{ \partial y^2 } \Phi \big)(  s, W_s )
       \big)(u,u)
       \Big|
    \\
    & 
    \quad
    +
    \tfrac{1}{2e^{ \rho s}} 
    \left|
\big\|
    \big[
    \big(\tfrac{\partial}{\partial y}\Phi\big)(s,W_s)\big]^*
    (\nabla \V)(Y_s)
       \big\|_\U^2
       -
       \left\|
       B(x)^*
       (\nabla \V)(x)\right\|_\U^2
       \right|
    \bigg) 
    \,
    ds
    \Bigg]
    .
  \end{split}     
  \end{equation}
  Fatou's lemma therefore shows for 
  all $ t \in [0,h] $ that
  \begin{equation}
  \begin{split}
  \label{eq:after.Fubini}
    &
    \E\!\left[\exp\!\left(e^{-\rho t}\V(Y_t)
                   +\int_0^{t}e^{-\rho r}\bar{\V}(Y_r)\,dr
    \right)\right]
      -e^{ \V(x)}
    \\
    &
    \leq
    \liminf_{n\to\infty}\E\!\left[\exp\!\left(e^{-\rho (t\wedge\tau_n)}
     \V(Y_{t\wedge\tau_n})
                   +\int_0^{t\wedge\tau_n}e^{-\rho r}\bar{\V}(Y_r)\,dr
    \right)\right]
    -e^{ \V(x)}
   \\
   &
   \leq
   \E\Bigg[
    \int_0^{t} 
    \exp\!\left( \V(Y_s)
                   +\smallint_0^{s} 
                   e^{-\rho r} \bar{ \V}(Y_r)\,dr
    \right)
    \bigg(
      \left|\bar{\V}(Y_s)- \bar{\V}(x)\right|
      +\rho| \V(Y_s) - \V(x)|
      \\
      &
      \quad
      +\left| \V'(Y_s)\big(\tfrac{\partial}{\partial s}\Phi\big) (s,W_s)
            - \V'(x)F(x)\right|
    +
    \tfrac{1}{2}
      \bigg|\!
	   \tr_U\!     
       \big(
       B(x)^*
 	   (\Hess  \V)(x) 
 	   B(x)
       \big)        
    \\
    &
    \quad
      -
      \tr_U\!
      \left(     
      \big[
      \big(
      \tfrac{\partial}{\partial y}\Phi\big)(s,W_s)\big]^*
       (\Hess  \V)(Y_s) 
      \big(
      \tfrac{\partial}{\partial y}\Phi\big)(s,W_s)
      \right)
      \!
      \bigg|
    \\
    &
    \quad
      +\tfrac{1}{2e^{ \rho s}} 
    \Big|
\big\|
    \big[
    \big(\tfrac{\partial}{\partial y}\Phi\big)(s,W_s)\big]^*
    (\nabla \V)(Y_s)
       \big\|_\U^2
       -
       \left\|
       B(x)^*
       (\nabla \V)(x)\right\|_\U^2
       \Big|
       \\
       &
       \quad
      +  
      \tfrac{ 1 }{ 2 } 
      \Big|
      \sum\nolimits_{ u\in \mathbb{\U} } 
        \V'( Y_s ) 
        \big(
      \big( \tfrac{ \partial^2 }{ \partial y^2 } \Phi \big)(  s, W_s )
      \big)(u,u)
      \Big|
      \,
    \bigg)    
    \,ds
    \Bigg].
  \end{split}
  \end{equation}
  Tonelli's theorem and H\"older's inequality hence
  imply for all $ t \in [0,h] $ that
  \begin{equation}
  \begin{split}
  \label{eq:after.Fubini1}
  &
  \E\!\left[\exp\!\left(e^{-\rho t}\V(Y_t)
                   +\int_0^{t}e^{-\rho r}\bar{\V}(Y_r)\,dr
    \right)\right]
      -e^{ \V(x)}
  \\ 
  &
   \leq
    \int_0^t
    \left\|
    \exp\!\left(\V(Y_s)
                   +\int_0^{s} 
                   e^{-\rho r} \bar{\V}(Y_r)\,dr
    \right)
    \right\|_{\mathcal{L}^2(\P;\R)}
    \Big(
      \rho\| \V(Y_s)-\V(x)\|_{\mathcal{L}^2(\P;\R)}
      \\
      &
      \quad
      +\left\|\V'(Y_s)\big(
      \tfrac{\partial}{\partial s}\Phi\big) (s,W_s)
            - \V'(x)F(x)\right\|_{\mathcal{L}^2(\P;\R)}
      +
      \tfrac{1}{2} 
      \Big\|
      \!
	   \tr_U\!     
       \big(
       B(x)^*
 	   (\Hess  \V)(x) 
 	   B(x)
       \big)
       \\
       &
       \quad
       -
        \tr_U\!
      \Big(    
      \big[
      \big(
      \tfrac{\partial}{\partial y}\Phi\big)(s,W_s)\big]^*
       (\Hess \V)(Y_s)  
      \big(
      \tfrac{\partial}{\partial y}\Phi\big)(s,W_s) 
      \Big)
       \Big\|_{\mathcal{L}^2(\P;\R)}
   \\
   &
   \quad
      +
      \tfrac{1}{2e^{ \rho s} } 
      \left\|
        \big\|
        \big[
         \big(\tfrac{\partial}{\partial y}\Phi\big)(s,W_s)\big]^*
         (\nabla \V)(Y_s)
       \big\|_{\U}^2
      -\left\|B(x)^* (\nabla \V)(x)\right\|_\U^2
      \right\|_{\mathcal{L}^2(\P;\R)}
  \\
  & 
  \quad
      +
      \tfrac{1}{2}
        \left\|
        \sum\nolimits_{ u\in \mathbb{\U} } 
       \V'( Y_s ) 
       \big(
      \big( \tfrac{ \partial^2 }{ \partial y^2 } 
      \Phi \big)(  s, W_s )
      \big)
       (u,u)
        \right\|_{
          \mathcal{L}^2( \P; \R )
        }
      +
      \left\|\bar{\V}(Y_s)
      - 
      \bar{\V}(x)\right\|_{\mathcal{L}^2(\P;\R)}
    \Big) \,ds
    .
    \end{split}
\end{equation}
  Next we estimate the $ \mathcal{L}^2(\P; \R) $-semi-norms
  on the right-hand side of \eqref{eq:after.Fubini1} separately.
Lemma~\ref{lemma:another_growth_estimate}
implies for all 
$ y, z \in H $, 
$ i \in \{0,1,2\} $
that
\begin{equation}
\label{eq:derivativeU1}
\|
\V^{(i)}(y) 
- 
\V^{(i)}(z)  
\|_{L^{(i)}(   H, \R)}
\leq
c \, 2^{c-1} 
\| y - z \|_H
\left(
1+  \V(y) 
+
\| y - z \|_H^{c-1}
\right)
.
\end{equation}
  The assumption that
  $ \forall\, y \in H \colon
    | \bar{\V}(y) |
  \leq
    c \left( 1 +  | \V(y)  |^{ \gb } \right) $
  and H\"older's inequality hence
  prove for all 
  $ s \in (0, h] $
  that
\begin{equation}
\begin{split}
 \label{eq:estimate.fist.term} 
    &
    \E\!\left[ \exp\!\left(\!2\V(Y_s)
                   +
                   2
                   \int_0^s
                    e^{-\rho r}
                   \bar{\V}(Y_r)\,dr
          -2 \V(x)\right)\right]
          \\
          &
    \leq 
    \E\!\left[ \exp\!\left(\!
        2 \, |  \V( Y_s ) - \V(x)  |
            + 2 \smallint_0^s \left| \bar{\V}( Y_r ) \right| dr
         \! \right)\!\right]
    \\
    & 
    \leq
    \E\!\left[
      \exp\!\left(\!
         c\,2^c
        \big[ 1 +
           \V(x) 
          +
           \|
            Y_s - x
           \|_H^{  c - 1 }
        \big]
         \|
          Y_s - x
        \|_H
        +
        2 c \smallint_0^s \left( 1 + |\V( Y_r )|^{\gb} \right) dr
     \right)
   \right]
   \\
    &
    \leq
    \left\|
      \exp\!\left(
         c\,2^c
        \big[ 1 +
           \V(x)  
          +
           \|
            Y_s - x
           \|_H^{  c - 1  }
        \big]
         \|
          Y_s - x
         \|_H \right) \right\|_{\mathcal{L}^1(\P; \R)}
        \\
        &
        \quad
        \cdot
        \Big\| \! \exp\!\Big(   2 c  \smallint_0^s \big(
        1 + |\V( Y_r )|^{\gb} \big)\, dr \Big)\Big\|_{\mathcal{L}^{\infty}(\P; \R)}
    \\ 
    &
    \leq
    \E \! \left[
      \exp\!\left(
           c\,2^c
        \big[ 1 +
           \V(x) 
          +
           \|
            Y_s - x
           \|_H^{  c - 1  } 
        \big]
        \left\|
          Y_s - x
        \right\|_H  \right)
    \right]
        \exp\!\Big( 
          2 c  \smallint_0^s  \big[ 1 + \| \V( Y_r )
          \|_{\mathcal{L}^{\infty}(\P; \R)}^{\gb} \big]
        \, dr  \Big) .
\end{split}
\end{equation}
Next we estimate the two factors
  on the right-hand side of \eqref{eq:estimate.fist.term} separately.
  Observe that~\eqref{eq:assumption.on.phi.increment}
  ensures that for all
  $ r \in [1,\infty) $,
  $ s \in (0,h] $
  it holds that
  $ \| \Phi( s, W_s ) - x \|_{\mathcal{L}^r(\P; H)} \leq c $.
  Combining this with Lemma~\ref{lemma:limit_Lp}
  establishes that for all $ s \in (0,h] $
  it holds that
  $ \| \Phi( s, W_s ) - x \|_{{\mathcal{L}}^\infty( \P; H ) } \leq c $.
  H\"older's inequality, 
Tonelli's theorem,
   and~\eqref{eq:assumption.on.phi.increment}
   therefore
  show that
  for all
  $ s \in (0,h] $ 
  it holds
  that
\begin{equation}
\begin{split}
  &
  \E\!\left[
      \exp\!\left(
        c\,2^c
        \big[ 1 +
            \V(x) 
          +
           \|
            Y_s - x
           \|_H^{  c - 1  } 
        \big] 
         \left \|
          Y_s - x
         \right \|_H \right)
  \right]
  %
  \\ 
  & 
  =
  \E\!\left[
    \sum_{n=0}^{\infty} \tfrac{( 2^c c )^{ n}}{n !}
        \big[ 1 +
           \V(x) 
          +
           \|
            \Phi(s, W_s ) - x
           \|_H^{  c - 1  } 
        \big]^n
        \left\|
          \Phi(s,W_s) - x
        \right\|_H^n
  \right]
  %
  \\ 
  & 
  =
  \sum_{ n = 0 }^{ \infty }
  \left\| \tfrac{( 2^c c )^{ n}}{n !}
        \big[ 1 +
           \V(x) 
          +
          \left\|
            \Phi(s, W_s ) - x
          \right\|_H^{ c - 1  }\!
        \big]^n
        \left\|
          \Phi(s, W_s) - x
        \right\|_H^n
        \right\|_{\mathcal{L}^1(\P; \R)}
  \\ 
  & 
  \leq
  \sum_{ n = 0 }^{ \infty }
  \tfrac{( 2^c c)^{  n}}{n !}
        \big[ 1 +
           \V(x) 
          +
          \left\|
            \Phi(s, W_s ) - x
          \right\|_{\mathcal{L}^\infty(\P; H)}^{  c - 1  }\!
        \big]^n
        \left\|
          \Phi(s, W_s) - x
        \right\|_{\mathcal{L}^n(\P; H)}^n 
  \\ 
  & 
 \leq
  \sum_{ n = 0 }^{ \infty }
    \tfrac{( 2^c c )^{  n}}{n !}
        \big[ 1 +
           \V(x) 
          +
          c^{c-1} 
        \big]^n
        \left\|
          \Phi(s, W_s) - x 
        \right\|_{\mathcal{L}^n(\P; H)}^n
  %
  \\ 
  & 
 \leq
  \sum_{ n = 0 }^{ \infty }
  \tfrac{ ( 2^c c )^{   n } }{ n! }
  \left[
    c \big( 1 +  \V(x)  
    \big)  
    +
    c^c 
  \right]^n
  \E\big[
     (
     \| F(x)  \|_H s 
    + 
     \|B(x) W_s  \|_H
     )^n
  \big]
  \\ 
  & 
  =
  \E\big[
  \!
    \exp\!\big(
       c\,2^c
       \big[
        c  ( 1 +   \V(x)  
         )  
    +
    c^c 
       \big]
     (
    \| F(x)   \|_H s 
    + 
    \|B(x) W_s\|_H
     )
    \big)
  \big]
  .
\end{split} 
\end{equation}
The assumption that 
$ \forall\, s\in (0,h] \colon \! \max\{\|F(x)\|_H, \|B(x)\|_{\HS(\U,H)} \} \leq c h^{-\delta} \leq c s^{-\delta} $,
  the fact that 
  $ \forall\, a, b \in [0,\infty)\colon  (a+b)^c\leq 2^{c-1} (a^c + b^c) $,
  and Lemma~\ref{l:exp.Gauss} hence show
  that for all 
  $ s \in (0,h] $ 
  it holds
  that
\begin{equation}
\begin{split}
  &
  \E\!\left[ \exp\!\left(
         c\,2^c
        \big[
          1 +
          \V(x) 
          +
           \|
            Y_s - x
           \|_H^{  c - 1  }
        \big]
         \|
          Y_s - x
         \|_H \right)
  \right]
  \\ 
  & 
  \leq
  \E\big[\! \exp\!\big(
       c^2\,2^c
      \left[ 
      1 
      + 
       V(x)   
    +
    c^{c-1} 
      \right]
       \left(
     \| F(x)  \|_H  s 
    + 
     \|B(x) W_s  \|_H
    \right)
    \big)
  \big]
  \\ 
  & 
  \leq
  \E\big[ 
  \!
    \exp\!\big(
       c^2\,2^c
      \left[ 
      1
      + 
      c [\min \{  s, 1\}]^{-\varsigma} 
      +
      c^{c-1}  
      \right]
  %
      \left(
    \| F(x)  \|_H   s 
    + 
     \|B(x) W_s \|_H
    \right)
    \big)
  \big]
  \\ 
  & 
  \leq
  \E\!\left[
    \exp\!\left(
    3
    \cdot
      2^c c^{ c+2} 
      [\min\{  s, 1\}]^{  -   \varsigma }
       \left(
     \| F(x)  \|_H  s 
    + 
     \| B(x) W_s  \|_H
    \right)
    \right)
  \right]
 %
  \\ 
  & 
 =
  \exp\!\left(
   3
    \cdot
      2^c c^{ c+2} 
      [\min\{  s,1 \}]^{ -  \varsigma }
     \| F(x) \|_H  s
  \right)
  \E\!\left[
    \exp\!\left(
     3
    \cdot
      2^c c^{ c+2} 
      [\min\{s, 1\}]^{-  \varsigma }
       \| B(x)W_s  \|_H
    \right)
  \right]
  \\ & \leq
  \exp\!\left(
   3
    \cdot
      2^c c^{ c+2} 
      [\min\{s, 1\}]^{-   \varsigma }
     \|F(x) \|_H   s
  \right)
  2
  \exp\!\left(
    \tfrac{
       9
    \cdot
      2^{2c} c^{2c+4}  
       \| B(x)  \|^2_{\HS(\U, H)}
      s
    }{
      2
      \,
      [\min\{ s, 1\}]^{2  \varsigma }
    }
  \right)
  \\ 
  & 
  \leq
  2
  \exp\!\left(
     9
    \cdot
      2^{2c-1} c^{2c+4} 
      s
      [\min\{s, 1\}]^{-2 \varsigma }
    \left[
       \| F(x)  \|_H
      +
       \| B(x)   \|^2_{ \HS( \U, H ) }
    \right]
  \right)
  \\ 
  & 
  \leq
  2 \exp\!\left(
      9
    \cdot
      2^{2c-1} c^{2c+4} 
      s
      [\min\{ s, 1\}]^{-2  \varsigma }
      \left[
        c s^{-\delta} + c^2 \! s^{-2\delta}
      \right]
    \right)
    \\
    &
    \leq
    2 \exp\!\left(
       9
    \cdot
      2^{2c} c^{2c + 6} 
      s
      [\min\{ s, 1\}]^{  - 2 \delta -  2 \varsigma }
    \right)
   .
  \label{eq:estimate.fist.term1}
\end{split}
\end{equation}
 In the next step we 
 combine~\eqref{eq:assumption.on.phi.increment}, \eqref{eq:derivativeU1}, and Lemma~\ref{lemma:limit_Lp}
  to obtain 
  for all
  $ s \in (0,h] $ 
  that
\begin{equation}
\begin{split}
\label{eq:before.estimate.fist.term2}
&
  \|
    \V( Y_s )
  \|_{ \mathcal{L}^{ \infty }( \P; \R )
  }
\leq
  \V(x) + \| \V( Y_s ) - \V(x) \|_{ \mathcal{L}^{ \infty }( \P; \R ) }
\\ &
  \leq
      \V(x) + c\,2^{c-1}  
      \left\|
        \left( 1 +
           \V(x)  
         +
         \| Y_s - x \|_H^{  c - 1  }
        \right)
         \| Y_s - x
         \|_H
      \right\|_{
        \mathcal{L}^{ \infty }( \P; \R ) }
\\ 
& 
\leq
      \V(x) + c\,2^{c-1} 
      \left[
          c +
         c^2 [\min \{ s, 1\}]^{-\varsigma} 
        +  c^c  
      \right]
  \leq
    cs^{-\varsigma}
    + 
    c\,2^{c-1} 
    \left[
        2 c^2  
         [\min\{ s, 1\}]^{ - \varsigma}
        +
        c^c  
      \right]
      \\ 
      &
   \leq
   2^{c+1}c^{ c + 2}
         [\min \{s, 1\}]^{  - \varsigma} 
         .    
\end{split}
\end{equation}
Therefore, we obtain for all
  $ s \in (0,h] $ 
  that
\begin{equation}  
\begin{split}
\label{eq:estimate.fist.term2}
2 c \smallint_0^s \big[ 1 + \|\V( Y_r )\|_{\mathcal{L}^{\infty}(\P; \R)}^{\gb}\big] \, dr
 & 
 \leq
 2cs +
 2cs \big(
  2^{c+1}c^{ c+2}
         [\min \{ s, 1\}]^{  - \varsigma} \big)^{\gb}
         \\
         &
         \leq
 2^{2+(c+1)\gb}c^{1+( c+2)\gb}
         [\min \{ s,1\}]^{-\gb  \varsigma }          s
         .
\end{split}
\end{equation}
Combining this with~\eqref{eq:estimate.fist.term}
and \eqref{eq:estimate.fist.term1} ensures that for all
  $ s \in (0,h] $ 
  it holds
that
  \begin{equation}
  \begin{split}
    &
    \E\!\left[\exp\!\left(2 \V (Y_s)
                   +2\smallint_0^{s}e^{-\rho r}\bar{\V}(Y_r)\,dr
          -2 \V (x)\right)\right]
\\ 
&
  \leq
  2 \exp\!\left(
       9
    \cdot
      2^{2c} c^{2c+6} 
      s
      [\min\{ s, 1\}]^{ - 2 \delta - 2  \varsigma }
    \right)
%
    \exp\!\left(
     2^{2+(c+1)\gb} c^{1+( c+2)\gb}
         [\min \{ s,1 \}]^{-\gb \varsigma }          s
    \right)
  \\ 
  & 
  \leq
    2
    \exp\!\left(
      \left[
          9
    \cdot
      2^{2c} c^{2c+6} 
        +
       2^{2+(c+1)\gb} c^{1+( c + 2)\gb}
      \right]
     \!
      s
     [\min\{ s, 1\}]^{ -2\delta -  \max\{2,\gb\}  \varsigma }
    \right)
  \\ 
  &
   \leq
    2
    \exp\!\left(
      2^{ 2c\gb + 4 }  
      c^{ 2 c \gb  + \gb + 5 } 
       s  
     [\min\{ s, 1\}]^{ -2\delta -  \max\{2,\gb\}   \varsigma }
    \right)
    .
  \end{split}
  \end{equation}
  Hence,
  we obtain for all
  $ s \in (0,h] $ 
  that
  \begin{equation}  \begin{split}
  \label{eq:L2.expU}
    \left\|
      \exp\!\left(
        \V ( Y_s )
        +
        \smallint_0^s
        \tfrac{
          \bar{\V}( Y_r )
        }{
          e^{\rho r}
        }
        \, dr
      \right)
    \right\|_{\mathcal{L}^2(\P;\R)}\!
  \leq
    \sqrt{ 2 }
    \exp\!\left(
      \frac{
        2^{ 2 c \gb + 3 }  
        c^{ 2 c \gb  + \gb + 5 } 
       s  }
       {
     [\min\{ s,1\}]^{  2\delta +  \max\{2,\gb\}  \varsigma }
     }
  \right)
  e^{ \V (x) } 
  .
  \end{split}
  \end{equation}
  Moreover,
  \eqref{eq:assumption.on.phi.increment},
  the assumption that 
  $ \forall\,s\in (0,h] \colon \!\max\{\|F(x)\|_H $, 
  $ \|B(x)\|_{\HS(\U,H)} \}  \leq c h^{-\delta} \leq c s^{-\delta} $,
  the triangle inequality,
  and the 
  Burkholder-Davis-Gundy type inequality in Lemma~7.7 in Da Prato \& Zabczyk~\cite{dz92}
  assure that for all
  $ r \in [2,\infty) $,
  $ s \in (0,h] $ 
  it holds that
\begin{equation} 
\begin{split}
  \label{eq:increment.Y}
    &
     \| Y_s - x \|_{\mathcal{L}^r(\P;H)}
    =
     \|\Phi(s, W_s)-x \|_{\mathcal{L}^r(\P;H)}
    \leq 
    c
    \big \|
      \| F(x)  \|_H  s 
    + 
     \|B(x) W_s  \|_H
    \big \|_{
      \mathcal{L}^r( \P; \R )
    }
  \\ 
  & 
  \leq
    c   
    \big( 
       \| F(x)  \|_H  s
      + 
      \sqrt{s r (r - 1 )/2}
      \,
       \|
        B(x)
       \|_{
        \HS( \U, H )
      } 
    \big)
  \leq
    c  
    \big( 
      cs^{1 - \delta}
      + 
      c
      \sqrt{ r (r - 1 )/2}
      \,
      s^{\nicefrac{1}{2} - \delta} 
    \big)
    \\
    &
    \leq
	c^2 r
	s^{\nicefrac{1}{2}-\delta }
	\max\{\sqrt{s},1\}
	\leq
	c^2 r [\min \{s, 1\}]^{\nicefrac{1}{2}-\delta  }
	\max\{s,1\}
    .
\end{split}
 \end{equation}
  Combining this with \eqref{eq:assumption.on.phi.increment},
  \eqref{eq:derivativeU1},  
  and H\"older's inequality implies that for all
  $ r \in [2,\infty) $,
  $ i \in \{ 0, 1, 2 \} $,
  $ s \in (0, h ] $ 
  it holds
  that
  \begin{equation}  
  \begin{split}  
  \label{eq:increment.Ui}
    &
      \big\|
        \V^{ (i) }( Y_s ) - \V^{ (i) }( x )
      \big\|_{
        \mathcal{L}^r( \P; L^{ (i) }( H, \R) )
      }
      \\
      &
    \leq
    \left\|
       c \, 2^{c-1}  
    \left( 1 + 
      \V( x ) 
    +
    \left\| Y_s - x
    \right\|_H^{c-1}
    \right)
    \left\| Y_s - x \right\|_H
     \right\|_{\mathcal{L}^r(\P;\R)}
    \\
    &
    \leq
    c \, 2^{c-1}
    \left( \left\| Y_s - x
    \right\|_{\mathcal{L}^r(\P; H)} + 
      \V( x ) 
     \left\| Y_s - x
    \right\|_{\mathcal{L}^r(\P; H)}
    +
    \left\| Y_s - x
     \right\|_{\mathcal{L}^{rc}(\P; H)}^c
    \right)
  \\
  &
    \leq
    c \, 2^{c-1} 
    \left(
      1 +   
        c 
        s^{-
          \varsigma  
        }
      +
    \left\|
      Y_s - x
     \right\|^{ c - 1 }_{\mathcal{L}^{rc}(\P; H)}
    \right)
    \left\| Y_s - x
     \right\|_{\mathcal{L}^{rc}(\P; H)}
 \\ 
 & 
 \leq
    c^4 \, 2^{c-1}
     \left[ 
         2 c 
       [\min\{ s, 1\}]^{-\varsigma}
       +
       c^{ c-1 }
     \right] 
		  r   [\min \{s, 1\}]^{\nicefrac{1}{2}-\delta  }
		\max\{s,1\} 
 \\ 
  & 
   \leq
   r
   \,
   c^{
   	c + 4
   }
   \,
   2^{ c+1}
   \left[ \min\{s, 1\}
   \right]^{ \nicefrac{1}{2}  - \delta -\varsigma
   } 
   \max\{s, 1\} 
   .
  \end{split}
  \end{equation}
%
%
%
  H\"{o}lder's inequality,
  the assumption that
  $ \forall\,s\in (0,h] \colon \!\max\{\|F(x)\|_H, \|B(x)\|_{\HS(\U,H)} \}  \leq c h^{-\delta} \leq c s^{-\delta} $, 
  \eqref{eq:assumption.on.phi.sigma},
  and Lemma~\ref{lemma:equivalent000}
  hence
  show for all
  $ s \in (0, h ] $ 
  that
  \begin{equation}
  \begin{split}
  \label{eq:second.summand}
      &
      \left\|
        \V'( Y_s )
        \big(
          \tfrac{ \partial }{ \partial s } \Phi
        \big)( s, W_s ) -
        \V'(x)  F(x)
      \right\|_{
        \mathcal{L}^2( \P; \R )
      }
      \leq
      \left\|
        \V'( Y_s ) - \V'(x)
      \right\|_{
        \mathcal{L}^2( \P; L( H,\R) )
      }
      \left\| F(x) \right\|_H
    \\
     & 
     \quad
     +
      \left\|
        \left\|
          \V'( Y_s )
        \right\|_{L( H,\R)}
        \left\|
          \left(
            \tfrac{ \partial }{ \partial s } \Phi
          \right)\!( s, W_s ) -
          F(x)
        \right\|_H
      \right\|_{
        \mathcal{L}^2( \P; \R )
      }
    \\ 
    & 
    \leq
      \left\|
        \V'(x)
      \right\|_{L(  H,\R)}
      \left\|
        \left(
          \tfrac{ \partial }{ \partial s } \Phi
        \right)\!( s, W_s) -
        F(x)
      \right\|_{
        \mathcal{L}^2( \P; H )
      }
    \\ 
    & 
    \quad
        +
        \left\|
          \V'( Y_s ) - \V'( x )
        \right\|_{
          \mathcal{L}^4( \P; L(  H, \R ) )
        }
        \left[
          \left\| F(x) \right\|_H
          +
          \left\|
            \left(
              \tfrac{ \partial }{ \partial s } \Phi
            \right)\!( s, W_s )
            -
            F(x)
          \right\|_{
            \mathcal{L}^4( \P; H )
          }
        \right]
  \\ &
      \leq
        c^2
        \left[1 + \V(x)\right]  
           s^{\gc}
    +
     2^{ c + 3}
    \,
     c^{
      c+4
     }
     \left[ \min\{s, 1\}
     \right]^{
          \nicefrac{1}{2}  - \delta  -\varsigma
     } 
    \left[
      c s^{- \delta}
      + c s^{ \gc }
    \right]     
   \max\{s,1\} 
    \\
    &
    \leq
     c^2
        \left[1 + cs^{- \varsigma } \right] 
            s^{\gc}
    +
      2^{ c + 4}
    \,
     c^{
       c + 5
     }
     \left[ \min\{s, 1\}
     \right]^{
          \nicefrac{1}{2}   - 2\delta  -\varsigma
     }     
   [\max\{s,1\}]^{1+\gc}
\\
&
\leq
    2^{ c + 5}
    \,
     c^{
         c + 5
     }
     \left[ \min\{s, 1\}
     \right]^{
          \gc - 2\delta  -\varsigma
     }     
   [\max\{s,1\}]^{1+\gc}
    .
  \end{split}
  \end{equation}
  In addition,
  observe that
  \eqref{eq:increment.Ui},  
  H\"{o}lder's inequality,
  \eqref{eq:assumption.on.phi.sigma},
  Lemma~\ref{lemma:equivalent000},
  and
  the assumption that
  $ \forall\,s \in (0,h] \colon \!\max\{\|F(x)\|_H, \|B(x)\|_{\HS(\U,H)} \}  \leq c h^{-\delta} \leq c s^{-\delta} $
  ensure that for all
  $ s \in (0, h ] $ 
  it holds
  that
  \begin{equation}
  \label{eq:UprimePhiy}
  \begin{split}
    &
    \big\|
      \V'( Y_s ) 
      \big(
        \tfrac{ \partial }{ \partial y } \Phi
      \big)(  s, W_s )
      -
      \V'(x)
       B(x)
    \big\|_{
      \mathcal{L}^4( \P; L(\U, \R) )
    }
  \\ 
  &
  \leq
    \left\|
      \left\|
        \V'( Y_s )
      \right\|_{
        L(   H, \R )
      }
      \big\|
        \big(
          \tfrac{ \partial }{ \partial y } \Phi
        \big)(  s, W_s )
        -
        B( x )
      \big\|_{
        \HS( \U, H )
      }
    \right\|_{
      \mathcal{L}^4( \P; \R )
    }
  \\ &  
    +
    \left\|
      \V'( Y_s ) - \V'(x)
    \right\|_{
      \mathcal{L}^4( \P; L(  H,\R) )
    }
    \left\|
      B(x)
    \right\|_{
      \HS( \U, H )
    }
  \\ &
  \leq
    \left\|
      \V'(x)
    \right\|_{
      L(  H, \R )
    }
    \big\|
      \big(
        \tfrac{ \partial }{ \partial y } \Phi
      \big)( s, W_s )
      -
      B( x )
    \big\|_{
      \mathcal{L}^4( \P; \HS(\U, H ) )
    }
  \\ & 
    +
    \left\|
      \V'( Y_s ) - \V'(x)
    \right\|_{
      \mathcal{L}^8( \P; L(  H,\R) )
    }
    \left[
      \left\|
        B(x)
      \right\|_{
        \HS( \U, H )
      } 
      + 
        \big\|
          \big(
            \tfrac{ \partial }{ \partial y } \Phi
          \big)( s, W_s ) -
          B(x)
        \big\|_{
          \mathcal{L}^8( \P; \HS( \U, H ) )
        }
      \right]
  \\ & \leq
        c
        \left[
          1 +
          \V(x)
        \right] 
        c s^{\gc}
   +
    2^{ c + 4}
     c^{
       c + 4
     }
     \left[ \min\{s, 1\}
     \right]^{
          \nicefrac{1}{2} -   \delta  -\varsigma
     }
           [
            c
            s^{- \delta}
            +
            c s^{ \gc }
             ]
          \max\{s, 1\}
  \\
   & 
   \leq
        c^2
        s^{ \gc }
        \left[
          1 + c s^{ - \varsigma }
        \right]
 +
      2^{ c + 5}
     c^{
        c + 5
     }
     \left[ \min\{s, 1\}
     \right]^{
           \nicefrac{1}{2}  - 2\delta  -\varsigma
     }
      [\max\{s,1\} ]^{1+ \gc} 
     \\
   & 
   \leq
    2^{ c + 6}
     c^{
         c+ 5
     }
     \left[ \min\{s, 1\}
     \right]^{
          \gc -    2\delta  -\varsigma
     }
      [\max\{s,1\} ]^{ 1+ \gc} 
       .
  \end{split}
  \end{equation}
Moreover, note
that
  for all
  $ A_1, A_2 \in \HS(\U, H) $, 
  $ B_1, B_2 \in L(H) $
  it holds
  that
  \begin{align}
  \label{eq:matrix_identity}
    & 
    | \tr_U\!\left( A_1^* B_1 A_1 - A_2^* B_2 A_2\right) |
    =
    \nonumber
    \bigg|
    \sum_{u \in \mathbb{\U}} 
    \left<
    \left(A_1^{*}B_1A_1-A_2^{*}B_2A_2\right)u,u \right>_\U
    \bigg|
    \\
    \nonumber
    &
    =
    \bigg|
    \sum_{u \in \mathbb{\U}} 
    \left<
    B_1A_1u, A_1 u\right>_H
    -
    \sum_{u\in \mathbb{\U}}
    \left< B_2A_2 u, A_2 u \right>_H
    \bigg|
    =
    |
     \left< A_1 ,B_1A_1 \right>_{\HS(\U,H)}- \left< A_2  ,B_2A_2 \right>_{\HS(\U,H)}
     |
  \\ &
  \nonumber
    =
    |
    \left < A_1-A_2,B_1A_1 \right>_{\HS(\U,H)}
    +
    \left < A_2,B_1(A_1-A_2) \right>_{\HS(\U,H)}
    +
    \left < A_2,(B_1-B_2)A_2 \right>_{\HS(\U,H)}
    |
  \\ & 
  \leq
  \left\|
  A_1 - A_2
  \right\|_{
  	\HS( \U, H )
  }
  \left\|
  B_1 A_1
  \right\|_{
  	\HS( \U, H )
  }
  +
  \left\|
  A_2
  \right\|_{\HS(\U,H)}
  \left\|
  B_1
  ( A_1 - A_2 )
  \right\|_{\HS(\U,H)}
  \\ &
  \quad
    +
    \left\|
    A_2
    \right\|_{
    	\HS( \U, H )
    }
    \left\|
    ( B_1 - B_2 ) A_2
    \right\|_{
    	\HS( \U, H )
    }
  \nonumber
  \\ &
  \nonumber
  \leq
    \left\|
      A_1 - A_2
    \right\|_{
      \HS( \U, H )
    }
      \left\|
        B_1
      \right\|_{
        L( H )
      }
    \left[
      \left\|
        A_1
      \right\|_{
        \HS( \U, H )
      }
      +
      \left\|
        A_2
      \right\|_{\HS(\U,H)}
    \right]
    +
    \left\|
    B_1 - B_2
    \right\|_{
    	L( H )
    }
    \left\|
    A_2
    \right\|_{
    	\HS( \U, H )
    }^2
  \\ &
  \nonumber
  \leq 
    \left[
      \left\|
        A_1 - A_2
      \right\|_{
        \HS( \U, H )
      }^2
      +
      2
      \left\|
        A_1 - A_2
      \right\|_{
        \HS( \U, H )
      }
      \left\|
        A_2
      \right\|_{
        \HS(\U,H)
      }
    \right]
    \left[
      \left\|
        B_1 - B_2
      \right\|_{
        L( H )
      }
      +
      \left\|
        B_2
      \right\|_{
        L( H )
      }
    \right]
    \\
    \nonumber
    &
    \quad
    +
     \left\|
     B_1 - B_2
     \right\|_{
     	L( H )
     }
     \left\|
     A_2
     \right\|_{
     	\HS( \U, H )
     }^2 
  \end{align}
  %
  (cf., e.g., (62) in Hutzenthaler et al.~\cite{HutzenthalerJentzenWang2014}).
  Next we apply~\eqref{eq:matrix_identity}  
  (with
  $ A_1 =
    ( \tfrac{ \partial }{ \partial y } \Phi
    )(s, W_s) $,
  $ A_2 = B(x) $,
  $ B_1 = ( \Hess \V )( Y_s ) $,
  and
  $ B_2 = (\Hess \V )( x ) $
  for $ s \in [0,h] $  
  in the notation of~\eqref{eq:matrix_identity}),
  we take expectations, we apply H\"older's inequality, 
  we apply Lemma~\ref{lemma:equivalent000},
  we use
  the assumption that
  $ \forall\,s\in (0,h] \colon \! \max\{\|F(x)\|_H, \|B(x)\|_{\HS(\U,H)} \}  \leq c h^{-\delta} \leq c s^{-\delta} $,
  and we apply 
  \eqref{eq:assumption.on.phi.sigma}
  and \eqref{eq:increment.Ui} to obtain
  that
  for all
  $ s \in (0,h] $ 
  it holds that
  \begin{align} 
  \label{eq:fourth.summand}
    \nonumber
    & \left\|
      \tr_U
      \!\left(\!
      \big[ 
       \big( 
         \tfrac{ \partial }{ \partial y } \Phi
        \big)( s, W_s )
       \big]^* 
       (\operatorname{Hess}\V)
       ( Y_s ) 
         \big( \tfrac{\partial}{\partial y}\Phi\big)(s,W_s)
       -
      B(x)^*
      ( \operatorname{Hess} \V )(x)
      B(x)
      \!
     \right)
     \right\|_{\mathcal{L}^2(\P;\R)}
  \\
  \nonumber
  & \leq 
    \Big\|   
      \big\|
        \big(
          \tfrac{ \partial }{ \partial y } \Phi
        \big)( s, W_s)
        -
        B(x)
      \big\|_{
        \HS( \U, H )
      }^2
  + 
      2 \big\|
        \big(
          \tfrac{ \partial }{ \partial y }
        \Phi 
        \big)(s,W_s)
        -
        B(x)
      \big\|_{\HS(\U,H)}
      \|
        B(x)
       \|_{
        \HS(\U,H)
      }
    \Big\|_{
      \mathcal{L}^4( \P; \R )
    }
    \\
    \nonumber
    &
    \quad
    \cdot
    \big[
    \|
        ( \Hess \V )( x )
       \|_{
        L(   H )
      } 
  +
      \left\|
        ( \Hess \V )( Y_s )
        -
        ( \Hess \V )( x )
      \right\|_{
        \mathcal{L}^4( \P; L( H )
        )
      } 
    \big]
    \\
    &
    \nonumber
    \quad
    +
    \|
    ( \Hess \V )( Y_s )
    -
    ( \Hess \V )( x )
    \|_{
    	\mathcal{L}^2( \P; L(   H )
    	)
    }
    \left\|
    B(x)
    \right\|_{
    	\HS( \U, H )
    }^2
 \\   
 \nonumber
 & \leq
  \big[
    c^2 s^{2\gc}
    +
    2
    c
    s^{\gc}
    c s^{- \delta}
  \big]
  \big[ 
    c
    (
      1 + \V(x)
    )
    +
      2^{ c + 3} 
     c^{
         c +4 
     }
      [ \min\{s, 1\} ]^{
          \nicefrac{1}{2} -   \delta  -\varsigma
     }
      \! \max\{s,1\} 
  \big]
  \\
  &
  \quad
  +
  2^{ c + 2}
  c^{
  	c + 4
  }
  [ \min\{s, 1\} ]^{
  	\nicefrac{1}{2}  - \delta  -\varsigma
  }
  \!\max\{s,1\} 
  c^2 s^{-2 \delta}
 \\  
 &
 \nonumber
 \leq 
  3  c^2 
  [\min \{ s, 1\}]^{\gc - \delta}
  [\max\{s,1\}]^{2\gc} 
  \big[
  2   c^2[\min\{s, 1\}
  ]^{-\varsigma}
  +
      2^{ c + 3} 
     c^{
       c + 4
     }
      [ \min\{s, 1\} ]^{
          \nicefrac{1}{2} -  \delta  -\varsigma
     }
     \!
      \max\{s,1\} 
  \big]
  \\
  \nonumber
  &
  \quad
  +
   2^{ c + 2} 
   c^{
   	c + 6
   } 
   [ \min\{s, 1\} ]^{
   	\nicefrac{1}{2} -   3 \delta  -\varsigma
   }
   \! \max\{s,1\}  
  \\
  \nonumber
  &
 \leq 
  3   c^2  
  [\max\{s,1\}]^{1+2\gc}  
  \big[
  2  c^2  [\min \{s, 1\}]^{\gc - \delta -\varsigma}
  +
      2^{ c + 3} 
     c^{
      c + 4
     }
      [ \min\{s, 1\}
      ]^{
          \nicefrac{1}{2} + \gc -    2\delta  -\varsigma
     } 
  \big]
  \\
  \nonumber
  &
  \quad
  +
   2^{ c + 2} 
   c^{
   	c + 6  
   }
   [ \min\{s, 1\} ]^{
   	\nicefrac{1}{2} -  3\delta  -\varsigma
   }
   \!\max\{s,1\} 
  \\
  &
  \nonumber
  \leq
  2^{ c + 5} 
  c^{ c + 6}  
  [\max\{s,1\}]^{1+2\gc} 
     [ \min\{s, 1\}
      ]^{
           \gc -    3\delta  -\varsigma
} 
  . 
  \end{align}
  Furthermore, the fact that 
  $ \forall\,a,b\in \U\colon  |
      \| a \|_\U^2 - \| b \|_\U^2
     |
    \leq
    \| a - b \|_\U
    \left(
      2 \left\| b \right\|_\U
      +
       \left\| a - b
       \right\|_\U
    \right) $,
  H\"older's inequality, \eqref{eq:UprimePhiy}, 
  Lemma~\ref{lemma:equivalent000},
  and the assumption that
  $ \forall\,s\in (0,h] \colon\! \max\{\|F(x)\|_H $, 
  $ \|B(x)\|_{\HS(\U,H)} \}  \leq c h^{-\delta} \leq c s^{-\delta} $
  show that for all
  $ s \in (0,h] $ 
  it holds that
  \begin{align}  
  \label{eq:Uprime.Phiy.square}
  \nonumber
    &
      \left\|
        \big\|
          \big[
            \big(
              \tfrac{ \partial }{ \partial y } \Phi
            \big)(  s, W_s)
          \big]^{ * }
          ( \nabla \V )( Y_s )
        \big\|_\U^2
        -
        \left\|
          B(x)^*
          ( \nabla \V )( x )
        \right\|_\U^2
      \right\|_{
        \mathcal{L}^2( \P; \R )
      }
  \\ 
  & 
  \nonumber
  \leq
    \big\|
      \big[
        \big(
          \tfrac{ \partial }{ \partial y } \Phi
        \big)( s, W_s )
      \big]^*
      ( \nabla \V )( Y_s )
      -
      B(x)^*
      (\nabla \V )(x)
    \big\|_{
      \mathcal{L}^4( \P; \U )
    }
    \Big\|
     2
      \!
      \left\|
        B(x)^*
      \right\|_{
        \HS( H, \U)
      }
      \!
      \left\|
        ( \nabla \V )( x )
      \right\|_H
      \\
      &
      \quad
      +
      \big\| 
        \big[
          \big( 
            \tfrac{ \partial }{ \partial y } \Phi
          \big)( s, W_s)
        \big]^* 
        (\nabla \V) ( Y_s )
        -
        B(x)^* 
        ( \nabla \V )( x )
      \big\|_\U
    \Big\|_{
      \mathcal{L}^4( \P; \R )
    }
  \\ 
  \nonumber
  & 
  \leq
    2^{ c + 6}
     c^{
        c + 5
     }
      [ \min\{s, 1\}
      ]^{
          \gc -    2\delta  -\varsigma
     }
      [\max\{s,1\}]^{1+\gc} 
	\\
	&
	\nonumber
	\quad
	\cdot
	\big[
     2 c^2
     s^{- \delta}
      [
       1 + c s^{ - \varsigma }
       ]
       +
       2^{ c + 6}
     c^{
         c + 5
     }
     [ \min\{s, 1\}
     ]^{
          \gc -   2\delta  -\varsigma
     }
      [\max\{s,1\}]^{1+\gc} 
   \big]
  \\ 
  & 
  \nonumber
  \leq
   2^{2c + 13} 
     c^{
       2 c + 10
     }
      [ \min\{s, 1\} ]^{
          \gc - 4 \delta - 2 \varsigma  
     }
     [\max\{s,1\}]^{2+ 2\gc}
.    
  \end{align}
  In addition, note that
  H\"older's inequality,
  Lemma~\ref{lemma:equivalent000},
  \eqref{eq:assumption.on.phi.sigma2},
  and \eqref{eq:increment.Ui} imply 
  for all
  $ s \in (0,h] $ 
  that
  \begin{equation}  
  \begin{split}  
  \label{eq:third.summand}
  &
    \left\|
    \smallsum_{u\in \mathbb{\U}} 
      \V'(Y_s)
      \big(
      \big(
      \tfrac{\partial^2}{\partial y^2}
      \Phi
      \big)
      (s, W_s)
      \big)(u,u)
    \right\|_{
      \mathcal{L}^2( \P; \R )
    }
    \\
    &
  \leq
    \left\|
         \V'(Y_s)
    \right\|_{\mathcal{L}^{4}(\P;L(  H,\R))}
    \left\|
      \smallsum_{u\in \mathbb{\U}}
    \left(
      \big(
      \tfrac{\partial^2}{\partial y^2}
      \Phi
      \big)
      (s, W_s)
      \right)\!(u,u)
    \right\|_{\mathcal{L}^{4}(\P;H)}
  \\ 
  & 
  \leq
    \left(
      \left\|
        \V'(Y_s)-\V'(x)
      \right\|_{
        \mathcal{L}^4( \P; L(  H, \R) )
      }
      +
      \left\|
        \V'(x)
      \right\|_{
        L(   H, \R )
      }
    \right)
    \left\|
      \sum_{u\in \mathbb{\U}}
      \!
    \left(
      \big(
      \tfrac{\partial^2}{\partial y^2}
      \Phi
      \big)
      (s, W_s)
      \right)\!(u,u)
    \right\|_{\mathcal{L}^{4}(\P;H)}
  \\
   & 
   \leq
    \left(
     2^{ c + 3} 
     c^{
       c + 4
     }
      [ \min\{s, 1\}
      ]^{
          \nicefrac{1}{2} -    \delta  -\varsigma
     }
     \!
      \max\{s,1\} 
      +
      c
      [
        1 +
        \V( x )
       ] 
    \right)
     c s^{ \gc }
  \\ &
  \leq
   \left(
     2^{ c + 3} 
     c^{
      c + 4 
     }
     [ \min\{s, 1\}
      ]^{
          \nicefrac{1}{2} -    \delta  -\varsigma
     }
     \!
      \max\{s,1\} 
      +
      2 c^2
       [
        \min\{s, 1\}
       ]^{ - \varsigma }
    \right)
     c s^{ \gc }
  \\ 
  &
  \leq
  2^{c + 4}
     c^{
        c + 5
     } 
     [
      \min\{s, 1\}
     ]^{ 
         \gc  
         - \delta - \varsigma  
    } 
    [\max\{s,1\}]^{1+\gc}
      .
  \end{split}     
  \end{equation}
%
  %
 %
  Moreover, the assumption that 
  $ \forall\,y, z \in H \colon
     | \bar{\V}( y ) - \bar{\V}( z ) |
  \leq
    c
    \left( 1 +
      | \V (y) |^{ \ga}
      +
      | \V (z) |^{ \ga }
    \right)
    \| y - z \|_H $,
  \eqref{eq:before.estimate.fist.term2},
  and \eqref{eq:increment.Y}
  show
  for all
  $ s \in (0,h] $ 
  that
\begin{equation}  \begin{split}
\label{eq:estimate.Ubar}
   &
   \left\|
     \bar{\V}( Y_s ) - \bar{\V}(x)
   \right\|_{\mathcal{L}^2(\P;\R)}
 \leq
   \left\|
     c
     \left( 1 +
       \left| \V(x) \right|^{ \ga}
       +
       \left| \V(Y_s) \right|^{\ga}
     \right)
     \left\|
       Y_s - x
     \right\|_H
   \right\|_{
     \mathcal{L}^2( \P; \R)
   }
   \\ 
   & 
   \leq
   c
   \,
   \big( 1 +
     |\V(x)|^{\ga}
     +
     \left\| \V( Y_s) \right\|^{ \ga }_{
       \mathcal{L}^{ \infty }( \P; \R )
     }
   \big)
   \| Y_s - x \|_{\mathcal{L}^2(\P; H)}
   \\
   &
    \leq
    c
    \,
    \big(
      1 +
      | \V(x) |^{ \ga }
      +
      \big[
        2^{ c + 1}c^{ c + 2} [\min \{s, 1\}]^{  - \varsigma}
      \big]^{ \ga }
    \big)
    2c^2 [\min \{s, 1\}]^{\nicefrac{1}{2}-\delta  }
  \max\{s,1\}
 \\ 
 & 
 \leq
 2c^3 [\min \{s, 1\}]^{\nicefrac{1}{2}-\delta  }
  \max\{s,1\}
    \Big[
      1 +
      c^{ \ga }
      s^{ - \varsigma \ga }
      +
      2^{\ga( c +1)}c^{\ga( c+2)} [\min \{s, 1\}]^{  - \varsigma\ga}
    \Big]
 \\ 
 & 
 \leq
 2^{\ga( c +1) + 2}c^{\ga(  c +  2)+3}
  \max\{s,1\}
 [\min \{s, 1\}]^{\nicefrac{1}{2} -\delta - \varsigma \ga}
    .
  \end{split}     
  \end{equation}
In the next step
we insert~\eqref{eq:L2.expU},
  \eqref{eq:increment.Ui},
  \eqref{eq:second.summand},
  \eqref{eq:fourth.summand},
  \eqref{eq:Uprime.Phiy.square},
  \eqref{eq:third.summand},
  and
  \eqref{eq:estimate.Ubar}
  into~\eqref{eq:after.Fubini1}
  to obtain for all
  $ t \in (0,h] $ 
  that
  \begin{equation}  
  \begin{split}
    &
    \E\!\left[
      \exp\!\left(
        e^{ - \rho t }
        \V( Y_t )
        +
        \smallint\nolimits_0^t
          e^{ - \rho r }
          \bar{\V}( Y_r )
        \, dr
      \right)
    \right]
    - e^{ \V(x) }
   \\
   &
   \leq
    \int_0^t
  \sqrt{ 2 }
    \exp\!\left(
      \tfrac{
        2^{(2c+3) \gb } \, c^{(2c+6)\gb} 
       s  }
       {
     [\min\{s, 1\}]^{2\delta + \max\{2,\gb\}   \varsigma }
     }
  \right)
  e^{ \V (x) }
%
  \bigg[
   \rho
   \,
  2^{ c + 2}  
     \,
     c^{
       c + 4
     }
     \left[ \min\{s, 1\}
     \right]^{
          \nicefrac{1}{2}   - \delta  -\varsigma
     }
      \max\{s,1\} 
  \\ 
  & 
  \quad
    +
   2^{ c+5}
    \,
     c^{
         c + 5
     }
     \left[ \min\{s, 1\}
     \right]^{
            \gc -   2\delta  -\varsigma
     }     
   [\max\{s,1\}]^{1+\gc}
    \\ 
    & 
    \quad
      +
      \tfrac{1}{2}\cdot
    2^{c+5}
    \,
  c^{ c + 6}
  \left[\max\{s,1\} \right]^{1+2\gc} 
     \left[ \min\{s, 1\}
     \right]^{
           \gc  -   3\delta  -\varsigma
}
    \\
    & 
    \quad
      +
      \tfrac{1}{2}\cdot
 2^{2c+13}
    \,
     c^{
       2c+10
     }
     \left[ \min\{s, 1\}
     \right]^{
          \gc - 4 \delta - 2 \varsigma  
     }
     [\max\{s,1\}]^{2+ 2\gc}
    \\ 
    & 
    \quad +
    \tfrac{1}{2}\cdot
     2^{c+4}
     \,
     c^{
        c + 5
     } 
    \left[
      \min\{s, 1\}
    \right]^{
         \gc 
         - \delta - \varsigma 
    }
    [\max\{s,1\}]^{1+\gc}
  \\ 
  & 
  \quad
    +
 2^{\ga(c+1) + 2} \, c^{\ga( c+2)+3}
  \max\{s,1\}
 \left[\min \{s, 1\} \right]^{\nicefrac{1}{2}  -\delta - \varsigma \ga}
  \bigg]
    \, ds
  \\
   & 
   \leq
    e^{ \V(x) } 
    \int_0^t 
    \sqrt{ 2 }
    \exp\!\left(
      \tfrac{
       2^{(2c+3) \gb } \, c^{(2c+6)\gb} 
        s
      }{
        \left[
          \min\{s, 1\}
        \right]^{  
          2\delta + \max\{2, \gb\}  \varsigma
        }
      }
  \right)
%
      \max\{\rho, 1\}
      \,
       c^{
       	(2c + 10 ) \ga
       }
       \, 
    2^{(2c+\nicefrac{25}{2}) \ga }
  \\ 
  & 
  \quad
  \cdot
   \left [
      \max\{s,1\}
     \right ]^{  2 + 2\gc }
     \left[
      \min\{ s, 1\}
     \right]^{
        \gc - 4 \delta
        - \varsigma  - \varsigma \ga
    }
    \,
    ds .
  \end{split}
  \end{equation}
%
%
%
This implies for all
  $ t \in (0,h] $ that
  \begin{equation} 
   \begin{split}
  \label{eq:estimate.last.term1}
&
    \E\!\left[
      \exp\!\left(
        \tfrac{
        \V(Y_t)}{e^{  \rho t }}
        +
        \smallint_0^t
        \tfrac{
        \bar{\V}(Y_r)}{e^{ \rho r }}
        \, dr
      \right)
    \right]
\\
&
  \leq
    e^{ \V(x) }
    \Bigg[
    1
    +
    \max\{ \rho, 1 \}
    \,
    2^{
    	(2c + 13) \ga 
    }
    \!
    \int_0^t  
%
  \tfrac{
   \exp \left(
      \frac{
        2^{(2c+3) \gb }\, c^{(2c+6)\gb}
        s
      }{
        \left[
          \min\{s, 1\}
        \right]^{   
          2\delta + \max\{2, \gb\}   \varsigma 
        }
      }
  \right)
      c^{
        (2c+10) \ga 
      }
    \left[
      \max\{s,1\}
    \right]^{
       2 + 2\gc 
    }
  }{
    \left[
      \min\{ s, 1\}
          \right]^{
    \varsigma + \varsigma \ga
     + 4 \delta
    - \gc
    }
  }
  \, ds
  \Bigg]
  .
  \end{split}
  \end{equation}
  The proof of
  Lemma~\ref{l:exp.mom.abstract.one-step} is thus completed.
\end{proof}
\subsection{Exponential moments for tamed approximation schemes} 
\label{subsection:use_big_lemma}
In this subsection we apply
Lemma~\ref{Cor:exp.mom.abstract}
and Lemma~\ref{l:exp.mom.abstract.one-step}
above
to establish
in Proposition~\ref{proposition:Euler.bounded.increments}
below
exponential moment bounds for 
an appropriate
tamed
exponential Euler-type approximation scheme
(cf., e.g.,   \cite{HutzenthalerJentzenKloeden2012,
HutzenthalerJentzenMemoires2015,
Sabanis2013ECP,
Sabanis2014Arxiv,
HutzenthalerJentzenWang2014}
for related schemes in the case of finite
dimensional SODEs and,
e.g., 
\cite{MasterRyan,
GyongySabanisSiska2016,
JentzenPusnik2015,
BeckerJentzen2016,
HutzenthalerJentzenSalimova2016}
for related schemes in the case of infinite dimensional SPDEs).
\begin{prop}
\label{proposition:Euler.bounded.increments}
 Let
 $ \left( H, \left< \cdot, \cdot \right> _H, \left \| \cdot \right\|_H \right) $
 and 
 $ \left( \U, \left< \cdot, \cdot \right> _\U, \left\| \cdot \right \|_\U \right) $
 be separable $ \R $-Hilbert spaces with $ \#_H > 1 < \#_U $,
  let $ T \in (0,\infty) $, 
  $ \rho \in [0,\infty) $,
  $ \delta \in [0, \nicefrac{1}{14}) $,
  $ c, \gamma \in [1,\infty) $, 
    $ \varsigma \in \big( 0,  \tfrac{1 - 14 \delta}{2 + 2\gamma }  \big) $,
  $ F \in \mathcal{M} \big( \mathcal{B}( H), \mathcal{B}(H) \big) $,
$ B \in \mathcal{M} \big( \mathcal{B}( H), \mathcal{B}(\HS(\U,H)) \big) $,
  $ \V \in \mathcal{C}_{c}^3(  H, [0,\infty) ) $,
  $ \bar{\V}\in \mathcal{C}( H, \R) $,
	$ S \in \mathbb{M}( (0,T], L(  H)) $,    
  $ D \in \mathbb{M}( (0,T ], \mathcal{B}(H) ) $,
   let
   $ ( \Omega, \mathcal{ F }, \P, ( \mathcal{ F }_t )_{t \in [ 0, T]} ) $
   be a filtered probability space, 
   let
   $ W \colon [0,T] \times \Omega \to \U $
   be an
   $ \operatorname{Id}_\U $-cylindrical $ ( \mathcal{F}_t )_{t \in [0, T]} $-Wiener process 
   with continuous sample paths,
assume for all
$ h \in (0,T] $,
$ x, y \in H $
that
 $ \V(S_h x )\leq \V(x) $,
  $ \bar \V(S_h x )\leq \bar \V(x) $,
  $ | \bar{\V}(x) - \bar{\V}(y) |
  \leq
  c \left( 1 + |\V(x)|^{\gamma} + |\V(y)|^{\gamma}\right)
  \| x - y \|_H $,  
   $ |\bar{\V}(x)|    
   \leq c \left( 1 + \left|\V(x)\right|^{\gamma} \right) $,
   and
   $ D_h  \subseteq \{ v \in H \colon \V(v) \leq c h^{ -\varsigma } \} $,
  assume for all $ h \in (0,T] $, 
  $ x \in D_h $ 
   that
  $ \max\{ \|F(x)\|_H, $  $\|B(x)\|_{\HS(\U,H)} \}
    \leq c h^{-\delta} $ 
  and  
  $ (\mathcal{G}_{F,B}\V)(x)
     +\frac{1}{2} \,\|B(x)^{*}(\nabla \V)(x) \|_\U^2
     +\bar{ \V }(x)
   \leq
    \rho  \V(x) $,
    and
    let
    $ Y^{ \theta } \colon
    [0,T] \times \Omega \to  H $,
    $ \theta \in \varpi_T $,
    be $ (\mathcal{F}_t)_{t\in [0,T]} $-adapted
    stochastic processes 
    with continuous sample paths
    which satisfy
    for all
    $ \theta \in \varpi_T $,
    $ t \in ( 0, T ] $
    that
    \begin{equation} 
    \begin{split}  
    \label{eq:thm:Y}
    Y_t^{ \theta }
    &
    =
    S_{t-\llcorner  t \lrcorner_\theta}
    \!
    \left(
    Y_{
    	\llcorner  t \lrcorner_{ \theta }
    }^{ \theta }
    +
    \1_{ D_{ |\theta|_T } }\!(
    Y_{ \llcorner  t \lrcorner_{ \theta } }^{ \theta }
    )\!
    \left[
    F(
    Y_{ \llcorner  t \lrcorner_{ \theta } }^{ \theta }
    ) \,
    (
    t - \llcorner  t \lrcorner_{ \theta }
    ) + 
    \tfrac{
    	B( Y_{ \llcorner  t \lrcorner_\theta }^\theta )
    	(W_t - W_{ \llcorner  t \lrcorner_\theta}) 
    }{
    1 + 
     \| 
    B( Y_{ \llcorner  t \lrcorner_\theta }^\theta )
    (W_t - W_{ \llcorner  t \lrcorner_\theta }) 
     \|_H^2
}
\right]
\right)
.
\end{split}
\end{equation}
  Then 
  \begin{enumerate}[(i)]
  	\item it holds that
\begin{equation}  
\begin{split}  
\label{eq:exp.mom.quadratic.exponent} 
    \limsup_{
      \left | \theta \right |_T
      \searrow 0
    }
    \sup_{ t \in [0,T] }
    \E\!\left[
      \exp\!\left(
        \tfrac{
          \V( Y_t^{ \theta } )
        }{
          e^{ \rho t }
        }
        +
        \smallint_0^t
          \tfrac{
            \1_{ D_{  | \theta  |_T } }\!(
              Y_{ \lfloor s \rfloor_{ \theta } }^{ \theta }
            )
            \,
            \bar{\V}( Y_s^{ \theta } )
          }{
            e^{ \rho s }
          }
        \, ds
      \right)
    \right]
    \leq
      \limsup_{
        \left | \theta
        \right |_T
        \searrow 0
      }
      \E \big[
        e^{
          \V( Y_0^{ \theta } )
        }
       \big]
  \end{split}
  \end{equation}
  and 
  \item \label{item:Exp.Mom.Bound} it holds for all $ \theta \in \varpi_T $ that
\begin{equation}
\begin{split}
 \label{eq:Exp.Mom.Bound}
&
\sup_{t\in [0,T]}\!\!
  \E\Big[
  \!
       \exp\Big(  
         \tfrac{
           \V( Y^{ \theta }_t )
         }{ e^{ \rho t } } 
         +  
         \smallint_0^t
           \tfrac{
             \1_{ D_{ |\theta|_T } }\!(
               Y^{ \theta }_{\lfloor s \rfloor_{ \theta } }
             )
             \,
             \bar{\V}( Y^{ \theta }_s )
           }{
             e^{ \rho s }
           }
           \,
          ds 
             \Big) 
          \Big]
          \\ 
          &
  \leq
    \exp\!\left(
    \tfrac{    
      \exp \left(\!
                  2
                [720 \max\{T, \rho, 1\}   c^3]^{ 
                    (720 c^3\!\max\{T, 1\} + 7)
                    \gamma
                }
          \right)
    }{
       [
        \min\{ |\theta|_T, 1\}
      ]^{       
            \varsigma + \varsigma \gamma
          + 7 \delta -\nicefrac{1}{2}
      }
    }
    \right)
     \E \big[ 
     e^{ \V( Y_0^{ \theta } ) }
     \big]
     .
\end{split}
\end{equation}
\end{enumerate}
\end{prop}
\begin{proof}[Proof of Proposition~\ref{proposition:Euler.bounded.increments}]
Throughout this proof 
 let
 $ \hat c \in [1,\infty) $
 and
 $ \varrho_h \in (0,\infty) $,
 $ h \in (0,T] $,
 be the real numbers 
 which satisfy 
 for all 
 $ s \in (0,T] $ 
 that
  $ \hat{c} =  360 c^3\!\max\{T, 1\} $
  and
 \begin{equation} 
 \begin{split}
 \varrho_s
 &
 =
 \exp\!\left(
 \tfrac{ 
 	[2\hat c]^{
 		( 2 \hat c + 6)
 		\gamma
 	}
 	\,
 	s
 }{ 
 \min\{ s^{  
 	2\delta + \max\{2, \gamma\}  \varsigma
 },
  1\} 
}
\right)
\tfrac{
	[ 2 \hat c]^{
		(2 \hat c + 13 ) \gamma
	}  
	[
	\max\{s,\rho,1\}
	]^4
}{
\left[
\min\{ s, 1\}
\right]^{  \varsigma  + \varsigma \gamma +
		7 \delta
	- \nicefrac{1}{2}  
}
}
,
\end{split} 
\end{equation}
let $ \psi \colon H \to H $
be the mapping
which satisfies 
for all $ x \in H $ that
$ \psi(x) = \frac{x}{1+\| x \|_H^2} $,
and let
$ \Psi \colon   H\times [0,T] \times \U \to  H $
and
$ \Phi_h^x \colon
 [0,h]\times\U \to H $, $ (x,h) \in H \times (0,T] $,
  be the mappings
  which satisfy
  for all    
  $ h \in (0,T] $,
  $ x \in H $,  
  $ s \in [0,h] $, $ y \in  \U $ that
  \begin{equation}
  \begin{split}
    \Psi (x,s,y) 
             =
             x+
    		 F(x)s
    		 +
             \tfrac{B(x)y}
             {1+ 
             \left\| B(x) y\right\|_H^2}
             \qquad
             \text{and}
             \qquad
             \Phi_h^x( s, y) = \Psi(x, s, y)
             .
             \end{split}
  \end{equation}
  We now verify step by step
  the assumptions of
  Lemma~\ref{l:exp.mom.abstract.one-step}.
First,
  note that
  for all
  $ h \in ( 0, T] $,
  $ x\in H $
  it holds
  that
  \begin{equation}
  \label{eq:trivial}
   \Phi^x_h(0,0) = x 
   .
   \end{equation}
  Furthermore, observe that
  for all $ h \in (0,T] $, 
  $ x \in  H $,
  $ s \in  (0, h] $,
    $ y \in \U $ it holds that
  \begin{equation}  
  \begin{split}
  \label{eq:cond_1}
    &\left(\tfrac{\partial}{\partial s}\Phi_h^x\right) \! ( s,y)
    =  F(x) 
    .
  \end{split}     
  \end{equation}
Next we 
note that
$ \psi \in \mathcal{C}^2(H,H) $ and we observe that for all 
  $ z, u, v \in H $ it holds that
  \begin{equation} 
  \begin{split}
  \label{eq:first_der}
  &
    \psi' (z) u
    = 
      \tfrac{u}{1 +  
      \|z\|_H^2}
      -\tfrac{
        2  
        z \left< z,u\right>_H}
            { (1 + 
            \|z\|_H^2 )^2} 
      \end{split}
      \end{equation}
and
  \begin{equation} 
  \label{eq:second_derivative_psi}
  \begin{split}
    &
    \psi''(z) (u, v)
  = 
      -
      \tfrac{
        2 
         [
          u
          \left< z, v \right>_H
          +
          v
          \left< z, u \right>_H
          +
          z
          \left< u, v \right>_H
         ]
      }{
         ( 1 +   
        \|z\|_H^2  )^2
      }
      +
      \tfrac{
        8 
        z \left< z, u \right>_H  \left< z,v\right>_H
      }{
         (1+ 
        \|z\|_H^2 )^3
      }   
      .
  \end{split} 
  \end{equation}
Moreover,
  note that~\eqref{eq:first_der} ensures for all
  $ h \in (0, T] $, $ x \in H $, $ s \in (0, h] $, 
   $ y, u \in \U $ 
  that
  \begin{equation}  
  \begin{split}
  \label{eq:der_y}
    &
    \big(\tfrac{\partial}{\partial y}\Phi_h^x\big)( s,y)u
    =  
    \tfrac{B(x)u}
    {1+ 
    \|B(x)y\|_H^2}
      -\tfrac{
        2   
        B(x) y \left< B(x)y,B(x) u \right>_H 
        }
            { (1+ 
            \|B(x)y\|_H^2 )^2} 
    .
  \end{split}     
  \end{equation}
  The Cauchy-Schwarz inequality hence implies for all 
  $ h \in (0, T] $, $ x \in H $,
  $ s \in (0, h] $, 
   $ y \in \U $ 
  that
  \begin{equation}  
  \begin{split}
  &
    \big\|
      \big(\tfrac{\partial}{\partial y}\Phi_h^x\big)( s,y)-B(x)
    \big\|_{\HS(\U,H)}
    \\
    &
   \leq
     \left\| \tfrac{ B(x) }{ 1 + \| B(x) y \|_H^2} - B(x) \right\|_{\HS(U,H)}
     +
     \left\| \tfrac{2 B(x) y \langle B(x)y, B(x)(\cdot) \rangle_H}{ ( 1+ \| B(x) y\|_H^2)^2}
     \right\|_{\HS(U,H)}
    \\
    &
    \leq
    \left\|
    \left( \tfrac{1}{1+\| B(x) y \|_H^2} - 1 \right) B(x) \right\|_{\HS(U,H)}
    +
    \tfrac{
    \| 2 B(x) y \|_H \| B(x) y \|_H \| B(x) \|_{\HS(U,H)}
    }{ ( 1 + \| B(x) y \|_H^2 )^2 }
	\\
	&
	\leq
	\left| \tfrac{1}{1+\| B(x) y \|_H^2} - 1   \right|
	\| B(x) \|_{\HS(U,H)}
	+
	2 \| B(x) y \|_H^2 \| B(x) \|_{\HS(U, H)}
   .
  \end{split}     
  \end{equation}
  This ensures for all $ h \in (0,T] $,
  $ x \in H $,
  $ s \in (0,h] $, 
  $ y \in U $
  that 
  \begin{equation}
  \begin{split}
  &
  \big\|
  \big( \tfrac{\partial}{\partial y } \Phi_h^x \big)( s,y) - B(x)
  \big\|_{\HS(U,H)}
  \\
  &
  \leq
  \left| 1 - \tfrac{1}{1+\| B(x) y \|_H^2}   \right|
	\| B(x) \|_{\HS(U,H)}
	+
	2 \| B(x) y \|_H^2 \| B(x) \|_{\HS(U, H)}
\\
&
=
\tfrac{ \| B(x) y \|_H^2 \| B(x) \|_{\HS(U,H)} }{ 1 + \| B(x) y \|_H^2}
+
2 \| B(x) y \|_H^2 \| B(x) \|_{\HS(U,H)}
\leq
3 \| B(x) y \|_H^2 \| B(x) \|_{\HS(U, H)} .
  \end{split}
  \end{equation}
  The
  Burkholder-Davis-Gundy type inequality in Lemma~7.7 in Da Prato \& Zabczyk~\cite{dz92}
  therefore proves that for all 
  $ h \in (0,T] $, $ x \in D_h $,
  $ s \in (0, h] $
  it holds that
  \begin{equation}  
  \begin{split}
  &
    \big\|\big(\tfrac{\partial}{\partial y}\Phi_h^x\big)( s,W_s)-B(x)
    \big\|_{\mathcal{L}^8(\P;\HS(\U,H))}
   \leq
	3 \| B(x) \|_{\HS(U,H)}
	\| \| B(x) W_s \|_H^2 \|_{\mathcal{L}^8(\P; \R)}
	\\
	&
	=
	3 \| B(x) \|_{\HS(U,H)}
	\|  B(x) W_s   \|_{\mathcal{L}^{16}(\P; H)}^2 
	\leq
	3 \| B(x) \|_{\HS(U,H)}
	\big( \tfrac{ 16 \cdot 15}{2} \big)
	\| B(x) \|_{\HS(U, H)}^2 s
	\\
	&
	=
	360 \| B(x) \|_{\HS(U,H)}^3 s
	\leq
	360 (ch^{-\delta})^3 s
	\leq
	 360 
   c^3 s^{1-3 \delta}
    .
  \end{split}    
  \end{equation}
  This and~\eqref{eq:cond_1} show that for all 
  $ h \in (0,T] $, $ x \in D_h $,
  $ s \in (0,h] $  
  it holds that
  \begin{equation}
  \begin{split}
  \label{eq:cond1}
  &
   \max\!\Big\{ \big\|
  \big(\tfrac{\partial}{\partial s}\Phi_h^x\big) ( s,W_s) - F(x) \big\|_{\mathcal{L}^4(\P;H)},
  \big\|\big(\tfrac{\partial}{\partial y}\Phi_h^x\big)( s,W_s)-B(x)
    \big\|_{\mathcal{L}^8(\P;\HS(\U,H))}
  \Big \}
  \\
  &
  \leq
  360 c^3 s^{1-3\delta}
  \leq
  \hat c s^{\nicefrac{1}{2} - 3 \delta}
  .
  \end{split}
  \end{equation}
  Next observe that~\eqref{eq:second_derivative_psi} implies that for all $ z, u \in H $
  it holds that
  \begin{equation} 
  \begin{split}
  &
    \psi''(z) (u, u)
=
      \tfrac{
        8  
        z
         |
          \left< z, u \right>_H
         |^2
      }{
        \left(1+ 
        \|z\|_H^2\right)^3
      }
      -
      \tfrac{
      2  
        \left[
          2 u
          \left< z, u \right>_H
          +
          z
          \left\| u \right\|_H^2
        \right]
      }{
        \left( 1 +  
        \|z\|_H^2 \right)^2
      } 
    .
  \end{split} 
  \end{equation}
  Therefore, we obtain that 
  for all
  $ x \in H $, 
    $ y, u \in \U $
  it holds
  that
\begin{equation}
  \begin{split}
  \label{eq:der_yy} 
   \big( \tfrac{ \partial^2 }{ \partial y^2 }
     \psi (B(x)y )\big)
    (u,u)
&
    = 
    \psi '' \big( B(x) y \big) \big( B( x )u, B( x )u \big)
    \\ 
    & 
    =
    \tfrac{ 
      8  
       B(x) y
      |
        \left<  B(x) y, B(x)u
        \right>_H
       |^2
    }{
      \left(
        1 +  
        \left\|
         B(x) y
        \right\|_H^2
      \right)^3
    }
    -
    \tfrac{ 
      2  
      \big[
      2
        B(x)u
        \left<
          B(x) y ,
          B(x) u
        \right>_H
        +
       B(x) y 
        \|B(x)u \|_H^2
      \big]
   }{
     \left(1 +  
      \|B(x) y\|_H^2\right)^2
   }
  .
  \end{split}    
   \end{equation}   
  The Cauchy-Schwarz inequality hence shows for all
  $ x \in H $,
  $ y, u \in \U $ 
  that
  \begin{equation}
  \begin{split}
    \left\| \big( 
      \tfrac{ \partial^2 }{ \partial y^2 }
      \psi  ( B(x) y  )\big)
     (u,u)
    \right\|_H
    &
    \leq
    \tfrac{
      8  
      \left\|
         B(x) y
      \right\|_{H}^3
      \left\|
        B( x )u
      \right\|^2_H
    }{
      \left(
        1 + 
        \left\|
         B(x) y
        \right\|_H^2
      \right)^3
    }
    +
    \tfrac{
      6 
      \left\|
        B(x) y
      \right\|_H
      \left\|
        B( x ) u
      \right\|^2_H
   }{
     \left(1+ 
     \|B(x) y\|_H^2\right)^2
   }
\\
&
=
\left[
\tfrac{ 
	8 \| B(x) y \|_H^2
	+
	6 (1 + \| B(x) y \|_H^2 )
}
{
	\left( 1 + \| B(x) y \|_H^2 \right)^3
	}
\right]
\!
\| B(x) y \|_H
\| B(x) u \|_H^2
\\
&
\leq
6  
\|B(x)y\|_H
        \|
          B( x )u
        \|^2_H
  .
  \end{split}     
  \end{equation}
  This, the triangle inequality, and the Burkholder-Davis-Gundy type inequality in Lemma~7.7 in Da Prato \& Zabczyk~\cite{dz92} imply that for all
 $ h \in (0, T] $, $ x \in D_h $,
  $ s \in (0, h] $ it holds that
\begin{equation}
  \begin{split}
  \label{eq:cond2}
  &
   \left\|
      \sum\nolimits_{u\in \mathbb{\U}}
    \big(
      \big(
      \tfrac{\partial^2}{\partial y^2}
      \Phi_h^x
      \big)
      ( s, W_s)
      \big)(u,u)
    \right\|_{\mathcal{L}^{4}(\P;H)}
  \leq
 6 
 \!
  \left\|B(x)W_s \right\|_{\mathcal{L}^4(\P; H)}
  \sum\nolimits_{u\in \mathbb{\U}}
  \|
  B(x) u
  \|_H^2
  \\
  &
  \leq
 6  
  \|B(x)\|_{\HS(\U,H)}^3
  \sqrt{6s} 
  \leq
  6 \sqrt{6}  
  c^3 s^{\nicefrac{1}{2}-3 \delta}
  \leq 
  \hat c s^{\nicefrac{1}{2}-3 \delta}
  .
  \end{split}
  \end{equation}
  Next observe that for all $ h \in (0,T] $,
  $ x \in D_h $, 
   $ s \in (0, h] $, 
   $ r \in [1, \infty) $
   it holds that
  \begin{equation}
  \begin{split}
  \label{eq:cond3} 
&
    \left\|\Phi_h^x( s,W_s)-x \right\|_{\mathcal{L}^r(\P; H)}
     \leq
     \left\|
     \| F(x) \|_H s
     +
		\tfrac{ \| B(x) W_s \|_H }{
		1 +  
		 \| B(x) W_s \|_H^2 }     
     \right\|_{\mathcal{L}^r(\P; \R)} 
     \\ 
     & 
 \leq 
 \min \! \big\{
 \big\| \| F(x) \|_H s + \tfrac{1}{2} \big\|_{\mathcal{L}^r(\P; \R)},
 \big\|
 \| F(x) \|_H s + \| B(x) W_s \|_H
 \big\|_{\mathcal{L}^r(\P; \R)}
  \big\} 
 \\ 
 &
\leq
\min \! \big\{  
\tfrac{1}{2} + ch^{-\delta} s,
\big\|
\| F(x) \|_H s + \| B(x) W_s \|_H
\big\|_{\mathcal{L}^r(\P; \R)}
\big\}
\\
&
\leq
\min \! \big\{  
\tfrac{1}{2} + cT^{1-\delta},
\big\|
\| F(x) \|_H s + \| B(x) W_s \|_H
\big\|_{\mathcal{L}^r(\P; \R)}
\big\}
\\
&
\leq
\hat c \min \! \big\{  
1,
\big\|
\| F(x) \|_H s + \| B(x) W_s \|_H
\big\|_{\mathcal{L}^r(\P; \R)}
\big\}
     .     
  \end{split}
  \end{equation}
Moreover, note that
the fact that $ \psi \in \mathcal{C}^2(H, H) $
implies that for all $ h \in (0,T] $, $ x \in H $ 
it holds that
$ \Phi_{h}^x \in \mathcal{C}^{1,2}
(   [0,h]\times\U,  H) $.
Combining  
this,
\eqref{eq:trivial},
 \eqref{eq:cond1}, \eqref{eq:cond2},
and \eqref{eq:cond3}
allows us to apply
Lemma~\ref{l:exp.mom.abstract.one-step}
(with 
$ \varsigma = \varsigma $,
$ h = h $,
$ c = \hat c $,
$ \gamma_0 = \gamma $,
$ \gamma_1 = \gamma $,
$ \rho = \rho $,
$ \delta = \delta $,
$ \gamma_2 = \nicefrac{1}{2}-3\delta $, 
$ x = x $,
$ F = F $,
$ B = B $,
$ \bar V = \bar V $,
$ V = V $, 
$ \Phi = \Phi_h^x $
for $ x \in D_h $, 
$ h \in (0,T] $ 
in the notation of Lemma~\ref{l:exp.mom.abstract.one-step})
to obtain that
for all $ h\in(0,T] $,  
$ x \in D_h $,
$ t\in  (0,h] $ it holds that
\begin{equation}  
\begin{split}
\label{eq:prepare_big_lemma}
&
\E\!\left[
\exp\!\left(
\tfrac{
	\V(
	\Phi_h^x( t, W_t)
	)
}{
e^{\rho t}
}
+
\smallint_0^t
\tfrac{ \bar{\V}(\Phi_h^x( s,W_s))}{e^{\rho s}}
\, ds
\right)
\right]
\leq
\left(
1 + \smallint\nolimits_0^t \varrho_s \, ds
\right)
e^{ \V( x ) } 
.
\end{split}     
\end{equation}
  Next note that
  the estimates
  $ 1 - 2 \delta - \max\{2, \gamma \} \varsigma \geq 0 $
  and
  $ \nicefrac{1}{2} 
  -  \varsigma
  -   \varsigma   \gamma  
  - 7 \delta 
  > 0 $
  ensure that the function 
  $ (0,T] \ni h \mapsto \varrho_h \in (0,\infty) $
  is non-decreasing
  and that
  $ \limsup_{ h \searrow 0 } \varrho_h = 0 $.
  Combining
  this
  with~\eqref{eq:prepare_big_lemma}
  implies
  that
  for all $ h\in(0,T] $, 
  $ x \in D_h $,
  $ t\in  (0,h] $ it holds that
  \begin{equation}  
  \begin{split}
    &
    \E\!\left[
      \exp\!\left(
        \tfrac{
          \V(
            \Phi_h^x( t, W_t)
          )
        }{
          e^{\rho t}
        }
        +
        \smallint_0^t
        \tfrac{ \bar{\V}(\Phi_h^x( s,W_s))}{e^{\rho s}}
        \, ds
      \right)
    \right]
    \leq
    \left(
      1 + \smallint\nolimits_0^t \varrho_s \, ds
    \right)
    e^{ \V( x ) }
    \leq
    ( 1 + \varrho_h t)
    \, e^{ \V(x) } 
    .
  \end{split}     
  \end{equation}
  This ensures for all
  $ \theta \in \varpi_T $,
  $ x \in D_{|\theta|_T} $,
  $ t \in (0, |\theta|_T] $
  that
  \begin{equation}  
  \begin{split}
    &
    \E\!\left[\exp\!\left(
      \tfrac{ 
        \V(
          \Phi_{|\theta|_T}^x (t,W_t)
        )
      }{
        e^{\rho t}
      }
        +
        \smallint_0^t \tfrac{   
        \bar{\V}(\Phi_{ |\theta|_T }^x (s,W_s))}{e^{\rho s}}\,ds
        \right)\right]
     \leq e^{\varrho_{ |\theta|_T} t + \V(x)}.
  \end{split}     
  \end{equation}
  Hence, we obtain for all
  $ \theta \in \varpi_T $,
  $ x \in D_{|\theta|_T} $,
  $ t \in (0, |\theta|_T] $ 
  that
   \begin{equation}  
   \begin{split}
   &
   \E\!\left[\exp\!\left(
   \tfrac{ 
   	\V(
   	\Psi (x,t,W_t)
   	)
   }{
   e^{\rho t}
}
+
\smallint_0^t \tfrac{  
	\bar{\V}(\Psi (x,s,W_s))}{e^{\rho s}}\,ds
\right)\right]
\leq e^{\varrho_{ |\theta|_T} t + \V(x)}.
\end{split}     
\end{equation}
  Corollary~\ref{Cor:exp.mom.abstract}
  (with 
  $ T = T $,
  $ \theta = \theta $,
  $ \rho = \rho $,
  $ c = \varrho_{|\theta|_T} $,
  $ V = V $,
  $ \bar V = \bar V $,
  $ \Phi = \Psi $,
  $ E = D_h $,
  $ S = S $,
  $ W = W $,
  $ Y = Y^\theta $ 
  for 
  $ \theta \in \varpi_T $,
  $ h \in (0,T] $
  in the notation of Corollary~\ref{Cor:exp.mom.abstract})
therefore
  yields that for all  
  $ \theta \in \varpi_T $,
  $ t \in [0,T] $ 
  it holds that
  \begin{equation}  
  \begin{split}
  \label{eq:first.exp.mom.bound}
    &
    \E\!\left[\exp\!\left( \tfrac{\V(Y^{\theta}_{t})}{e^{\rho t}}
          +\smallint_0^{t}
                         \tfrac{\1_{D_{|\theta|_T}}
                         (Y^{\theta}_{\lfloor s \rfloor_{\theta}})
                         \, \bar{\V}(Y^{\theta}_s)}{e^{\rho s}}\,ds
                 \right)
          \right]
    \leq
      e^{\varrho_{|\theta|_T} t }
      \,
      \E \! \left[e^{{\V}(Y^{\theta}_0)}\right].
  \end{split}    
   \end{equation}
  This assures that for all $ \theta \in \varpi_T $ it holds that
  \begin{equation}  
  \begin{split}
  \label{eq:first.exp.mom.bound.before.final}
    &
    \sup_{t\in[0,T]} \E\!\left[\exp\!\left( \tfrac{\V(Y^{\theta}_{t})}{e^{\rho t}}
          +\smallint_0^{t}
                         \tfrac{\1_{D_{|\theta|_T}}
                         (Y^{\theta}_{\lfloor s\rfloor_{\theta}}) \, \bar{\V}(Y^{\theta}_s)}{e^{\rho s}}\,ds
                 \right)
          \right]
    \leq
      e^{\varrho_{|\theta|_T} T }
      \,
      \E \! \left[e^{{\V}(Y^{\theta}_0)}\right].
  \end{split}     
  \end{equation}
  This and
  the fact that
  $ \limsup_{h \searrow 0}\varrho_h = 0 $ 
  establish~\eqref{eq:exp.mom.quadratic.exponent}. It thus remains to prove~\eqref{eq:Exp.Mom.Bound}.
  For this observe that the fact that  $ \forall \, x \in [ 2^{20}, \infty ) \colon x \leq \exp\!\left( x^{ 1 / 4 } \right) $
  and the fact that
  $ \forall \, \theta \in \varpi_T \colon
  ( |\theta|_T )^{1 - 2\delta - \max \{ 2, \gamma\} \varsigma}
  \leq \max\{ 1, T\} $
  show that for all $ \theta \in \varpi_T $
  it holds that
  \begin{equation} 
  \begin{split}
  \label{eq:vartheta.h.estimate}
    \varrho_{|\theta|_T}    T 
   &
=
       \exp\!\left(
      \tfrac{ 
        [ 720 c^3 \!\max\{ T, 1 \}  ]^{
          ( 720 c^3 \!\max\{T, 1\} + 6)  \gamma
        } 
        |\theta|_T
      }{
        \left[
          \min\{  |\theta|_T, 1\}
        \right]^{  
          2\delta + \max\{2, \gamma\}  \varsigma
        }
      }
  \right)
\tfrac{
      [ 720 c^3 \! \max\{T, 1\}  ]^{
       ( 720 c^3 \! \max\{T, 1\} + 13 ) \gamma 
      }  
       [
            \max\{|\theta|_T, \rho, 1\}
            ]^4
            T
}{
     [
      \min\{ |\theta|_T, 1\}
     ]^{ 
         \varsigma  + \varsigma \gamma + 7 \delta
        - \nicefrac{1}{2}  
    }
    }
  \\
  &
  \leq
   \exp\!\left(
        [720 c^3\!\max\{T, 1\} ]^{
              ( 720 c^3\!\max\{T, 1\} + 7 )
            \gamma
        }
  \right)
  \tfrac{    
      [720  c^3\!\max\{ T, \rho, 1\}   ]^{
          ( 720 c^3 \! \max\{T, 1\} + 18)  \gamma
      } 
  }{
    \left[
      \min\{ |\theta|_T, 1\}
    \right]^{
        \varsigma
      + \varsigma \gamma
      + 7 \delta
      - \nicefrac{1}{2}
    }
  }
  \\
  &
  \leq
   \exp\!\left(
          2
        [720  c^3\!\max\{T, \rho, 1\} ]^{ 
           (720 c^3\!\max\{T, 1\} + 7)
            \gamma
        }
  \right)
  \tfrac{    
   1
  }{
    [
      \min\{ |\theta|_T, 1\}
    ]^{       
        \varsigma + \varsigma \gamma + 7 \delta
        -\nicefrac{1}{2}
    }
  }
      .
  \end{split}
  \end{equation}
  Combining \eqref{eq:first.exp.mom.bound.before.final} with \eqref{eq:vartheta.h.estimate} 
establishes \eqref{eq:Exp.Mom.Bound}.   
  The proof of  Proposition~\ref{proposition:Euler.bounded.increments} is thus completed.
\end{proof}
\section{Exponential moments
 for space-time-noise discrete approximation schemes}
 \label{section:3}
 In Proposition~\ref{proposition:Euler.bounded.increments}
 in Section~\ref{section:2}
 above
 we established
 exponential
 moment bounds
 for a class
 of time
 discrete approximation schemes.
 In this section we extend this result
 in Theorem~\ref{theorem:full_discrete_scheme_moments}
 and Corollary~\ref{Corollary:full_discrete_scheme_convergence} below to obtain
 exponential moments 
 for a class of space-time-noise
 discrete approximation schemes.
Theorem~\ref{theorem:full_discrete_scheme_moments}
below
proves exponential moment bounds
for numerical approximations of
SPDEs
whose coefficients
satisfy a general Lyapunov-type condition.
Corollary~\ref{Corollary:full_discrete_scheme_convergence}
below
specialises 
Theorem~\ref{theorem:full_discrete_scheme_moments}
to the case
where the considered Lyapunov-type
function is an affine linear transformation
of the squared Hilbert space norm.
Our proof of Theorem~\ref{theorem:full_discrete_scheme_moments}
uses two well-known
auxiliary lemmas
(see Lemma~\ref{lemma:switch_modification}
and Lemma~\ref{lemma:Change_variables} below).
\subsection{Setting}
\label{setting:exponential_moments_full_discrete}
 Let
 $ \left( H, \left< \cdot, \cdot \right> _H, \left \| \cdot \right\|_H \right) $
 and 
 $ \left( \U, \left< \cdot, \cdot \right> _\U, \left\| \cdot \right \|_\U \right) $
 be separable $ \R $-Hilbert spaces,
 let 
 $ \H \subseteq H $
 be a
 non-empty orthonormal basis of $ H $, 
 let $ \mathbb U \subseteq U $
 be a non-empty orthonormal basis of $ U $,  
 let 
 $ T \in (0,\infty) $, 
 $ \gamma \in [0,\infty) $,
 $ \delta \in [0,\nicefrac{1}{14}) $,
 $ \lambda \in \mathbb{M}( \H, \R ) $
 satisfy that
 $ \sup ( \operatorname{im}(\lambda)   ) < 0 $,
 let
 $ ( \Omega, \mathcal{ F }, \P, ( \mathcal{F}_t )_{t\in [0,T]} ) $
 be a filtered probability space, 
 let
$ (W_t)_{t\in [0,T]} $
be an
$ \operatorname{Id}_\U $-cylindrical 
$ ( \mathcal{F}_t )_{t \in [0, T]} $-Wiener process,
let
$ A \colon D( A ) \subseteq H \to H $ 
be the linear operator 
which satisfies for all 
$ v \in D(A) $
that
$ D(A) = \big \{
w \in H \colon \sum_{ h \in \H} \left | \lambda_h \left< h, w \right>_H \right | ^2 <  \infty
\big \} $
and  
$ A v = \sum_{ h \in \H} \lambda _h \left< h, v \right> _H h $,
let
$ (H_r, \left< \cdot, \cdot \right>_{H_r}, \left \| \cdot \right \|_{H_r} ) $, $ r \in \R $, be a family of interpolation spaces 
associated to $ - A $
(see, e.g., Definition 3.6.30 in~\cite{j16}),
let
$ \xi \in \mathcal{M}( \mathcal{F}_0, \mathcal{B}(H_\gamma) ) $,
$ F \in \mathcal{M}(  \mathcal{B}(H_\gamma), \mathcal{B} (H) ) $,
$ B \in \mathcal{M}( \mathcal{B}(H_\gamma), \mathcal{B}(\HS(U, H ) )) $, 
$ D =
( D_h^I )_{(I,h) \in \mathcal{P}(\H) \times (0,T] }
\in \mathbb{M} ( \mathcal{P}(\H)\times (0,T], \mathcal{B}(H_\gamma) ) $,
and let $ P_I \in L(H) $, 
$ I \in \mathcal{P} (\H) $, 
and
$ \hat P_J \in L(U) $, 
$ J \in \mathcal{P} (\mathbb{U}) $,
be the linear operators which satisfy
for all $ I \in \mathcal{P}(\H) $, 
$ J \in \mathcal{P}(\mathbb{U}) $,
$ x \in H $, 
$ y \in U $
that
$ P_I(x) = \sum_{h \in I} \left< h, x \right>_H h $
and
$ \hat P_J (y) = \sum_{u\in J}
\left< u, y \right>_U u $.
\subsection{Exponential moments for tamed approximation schemes}
\begin{lemma}[cf., e.g.,  Lemma~1 in Da Prato et al.~\cite{DaPratoJentzenRoeckner2012}]
	\label{lemma:switch_modification}
	Let $ ( \Omega, \mathcal{F}, \mu ) $
	be a sigma-finite measure space
	and
	let $ T \in (0,\infty) $,
	$ Y, Z \in \mathcal{M}(\mathcal{F} \otimes \mathcal{B}([0,T]), \mathcal{B}(\R)) $ 
	satisfy for all $ t \in [0,T] $
	that
	$ 	
	\mu(  Y_t \neq Z_t  ) 
	= 
	\mu \big( \int_0^T  | Y_s |\,ds = \infty \big)
	=
	0 $.
	Then   
	$ \mu \big( 
	\Omega \backslash \big\{ 
	\omega \in \Omega \colon
	\int_0^T | Y_s (\omega) | + | Z_s(\omega) | \, ds < \infty
	\text{ and }
	\int_0^T Y_s(\omega)\, ds = \int_0^T Z_s(\omega) \, ds 
	\big \} \big) = 0 $.
\end{lemma}
\begin{proof}[Proof of Lemma~\ref{lemma:switch_modification}]
	First, note that the Tonelli theorem implies that
	\begin{equation}
	\label{eq:Tonelli_first}
	\int_\Omega
	\left(  
	\smallint\limits_0^T
	|Y_s  - Z_s  | \,ds 
	\right) d\mu
	=
	\int_0^T
	\left(
	\smallint\limits_\Omega  
	| Y_s - Z_s |\,d\mu
	\right)  ds
	=
	0
	.
	\end{equation}
	This shows that 
  $ \mu  \big(    \smallint_0^T
	|Y_s  - Z_s  | \,ds  > 0     \big) = 0 $. 
	Therefore,
	we obtain that 
	$ \mu\big( \int_0^T
	|Y_s  - Z_s  | \,ds  = \infty \big) = 0 $.
	This
	and the assumption that
	$ \mu \big( \int_0^T | Y_s |\, ds = \infty 
	\big) = 0 $
	proves that
	\begin{equation}
	\begin{split}
	&
	\mu\Big(
	\smallint_0^T | Y_s | \, ds
	+
	\smallint_0^T | Y_s - Z_s | \, ds
	= \infty
	\Big)
	=
	\mu\Big(
	\Big\{
	\smallint_0^T |Y_s|\, ds = \infty
	\Big\}
	\cup
	\Big\{
	\smallint_0^T | Y_s - Z_s | \, ds
	= \infty
	\Big\}
	\Big)
	\\
	&
	\leq
	\mu\Big( 
	\smallint_0^T |Y_s|\, ds = \infty 
	\Big)
	+
	\mu
	\Big( 
	\smallint_0^T | Y_s - Z_s | \, ds
	= \infty 
	\Big)
	= 
	0.
	\end{split}
	\end{equation}
	The triangle
	inequality hence proves that 
	\begin{equation}
	\label{eq:finite_after}
	\mu\Big(
	\smallint_0^T | Z_s | \, ds = \infty \Big)
	\leq
	 \mu\Big(
	 \int_0^T | Z_s - Y_s |\, ds
	 +
	 \int_0^T | Y_s |\,ds = \infty
	 \Big)
	 =
	 0.
	\end{equation}
Next note 
that~\eqref{eq:Tonelli_first}
ensures that
\begin{equation}
\begin{split}
\int_\Omega  
\left|\smallint\limits_0^T
\1_{ \{ \int_0^T | Y_u - Z_u | \, du < \infty \} }
(
Y_s  - Z_s
)  \,ds 
\right| d\mu
&
\leq
\int_\Omega
\left(  
\smallint\limits_0^T
|Y_s  - Z_s  | \,ds 
\right) d\mu
=
0
.
\end{split}
\end{equation}
Hence,
we obtain that
\begin{equation} 
\mu\Big( 
\smallint_0^T
\1_{ \{ \int_0^T | Y_u - Z_u | \, du < \infty \} }
(
Y_s  - Z_s
)  \,ds 
\neq 0 
\Big)
=
0
.
\end{equation}
This shows that
\begin{equation}
	\label{eq:ensure_equality}
\mu\Big( 
\smallint_0^T
\1_{ \{ \int_0^T | Y_u | + | Z_u | \, du < \infty \} }
(
Y_s  - Z_s
)  \,ds 
\neq 0 
\Big)
=
0
.
\end{equation}
Combining~\eqref{eq:finite_after}
and~\eqref{eq:ensure_equality}
completes
the proof of 
Lemma~\ref{lemma:switch_modification}.
\end{proof}	
\begin{lemma}
\label{lemma:Change_variables}
Let 
$ ( H, \langle \cdot, \cdot \rangle_H,
\left \| \cdot \right \|_H) $
and
$ ( U, \langle \cdot, \cdot \rangle_U,
\left \| \cdot \right \|_U) $
be separable $ \R $-Hilbert spaces,
let $ T \in (0,\infty) $,
let $ Q \in L(U) $ be a non-negative and symmetric linear operator,
let
$ R \in \HS( Q^{\nicefrac{1}{2}} ( U ),H) $,
let
$ ( \Omega, \mathcal{F}, \P, 
(\mathcal{F}_t)_{t\in [0,T]} ) $
be a filtered probability space,
let $ (W_t)_{t\in [0,T]} $
be a $ Q $-cylindrical
$ (\mathcal{F}_t)_{t\in [0,T]} $-Wiener process,
let
$ \mathcal{G}_t \subseteq \mathcal{F} $, 
$ t \in [0,T] $,
satisfy for all $ t \in [0,T] $ that
$ \mathcal{G}_t =
\sigma_\Omega( \mathcal{F}_t
\cup \{
C \in \mathcal{F} \colon
\P(C)
=
0
\}) $,
and let
$ \tilde W \colon [0,T] \times \Omega \to H $
be a stochastic process with continuous sample paths 
which satisfies for all 
$ t \in [0,T] $ that
$ [ \tilde W_t ]_{\P, \mathcal{B}(H)} = \int_0^t R \, dW_s $.
Then it holds that
$ \tilde W $
is an $ R   R^* $-standard
$ (\mathcal{G}_t^+)_{t\in [0,T]} $-Wiener process.
\end{lemma}
\begin{proof}[Proof of Lemma~\ref{lemma:Change_variables}]
Throughout this proof 
let  
$ \mathbb{U}_0 \subseteq U $
be an orthonormal basis of $ \operatorname{Kern}( Q^{\nicefrac{1}{2}} ) $
and
let
$ \mathbb{U}_1 \subseteq U $
be an orthonormal basis of $ \operatorname{Kern}( Q^{\nicefrac{1}{2}} )^\perp $.
Next
note that for all
$ v, w \in H $,
$ s \in [0, T) $,
$ t \in (s, T] $
it holds that
\begin{equation}
\begin{split}
&
\E[ \langle v, \tilde W_t - \tilde W_s \rangle_H
\langle w, \tilde W_t - \tilde W_s \rangle_H ]
=
\E\bigg[ \bigg\langle v, \smallint_s^t R \, dW_r 
\bigg \rangle_{\! \! H}
\bigg\langle w, \smallint_s^t R \, dW_r 
\bigg \rangle_{\! \! H} \bigg ]
\\
&
=
\E\bigg[
\bigg(  \smallint_s^t  \langle v,  R \, dW_r   \rangle_H
\bigg)
\bigg(
\smallint_s^t  \langle w,  R \, dW_r   \rangle_H 
\bigg) \bigg ]
\\
&
=
\E\bigg[
\bigg(  \smallint_s^t  \langle R^* v,   dW_r   \rangle_{ Q^{1/2}(U) }
\bigg)
\bigg(
\smallint_s^t  \langle R^* w, dW_r   \rangle_{ Q^{1/2}(U) } 
\bigg) \bigg ]
.
\end{split}
\end{equation}
It\^o's isometry  
hence shows 
for all
$ v, w \in H $,
$ s \in [0, T) $,
$ t \in (s, T] $
that
\begin{equation}
\begin{split}
\label{eq:covariance}
&
\tfrac{1}{(t-s)}
\,
\E[ \langle v, \tilde W_t - \tilde W_s \rangle_H
\langle w, \tilde W_t - \tilde W_s \rangle_H ]
\\
&
= 
\big\langle
( Q^{\nicefrac{1}{2}}(U) \ni z
\mapsto
   \langle R^* v, z  \rangle_{ Q^{1/2}(U) } \in \R ),
   ( Q^{\nicefrac{1}{2}}(U) \ni z
   \mapsto
 \langle R^* w,  z \rangle_{ Q^{1/2}(U) } \in \R )   
 \big \rangle_{\HS( Q^{ 1 / 2 }(U), \R)}
 \\
 &
 = 
 \big\langle
 ( U \ni z
 \mapsto
 \langle R^* v, Q^{\nicefrac{1}{2}} z  \rangle_{ Q^{1/2}(U) } \in \R ),
 ( U \ni z
 \mapsto
 \langle R^* w, Q^{\nicefrac{1}{2}} z \rangle_{ Q^{1/2}(U) } \in \R )   
 \big \rangle_{\HS( U, \R)}
 \\
 &
 = 
 \sum\nolimits_{ u \in \mathbb{U}_0 \cup \mathbb{U}_1 }
  \langle R^* v,  Q^{\nicefrac{1}{2}} u \rangle_{ Q^{1/2}(U) }
  \langle R^* w,  Q^{\nicefrac{1}{2}} u \rangle_{ Q^{1/2}(U) } 
 \\
 &
 = 
 \sum\nolimits_{ u \in \mathbb{U}_1 }
 \langle R^* v,  Q^{\nicefrac{1}{2}} u \rangle_{ Q^{1/2}(U) }
 \langle R^* w,  Q^{\nicefrac{1}{2}} u \rangle_{ Q^{1/2}(U) } 
\\
&
=
\sum\nolimits_{ u \in \mathbb{U}_1 }
\langle   Q^{-\nicefrac{1}{2}} (  R^* v ),   
Q^{-\nicefrac{1}{2}}
( Q^{\nicefrac{1}{2}}
u ) \rangle_U
\langle   Q^{-\nicefrac{1}{2}} ( R^* w ),   
Q^{-\nicefrac{1}{2}}
( Q^{\nicefrac{1}{2}}
u )  \rangle_U 
\\
&
=
\sum\nolimits_{ u \in \mathbb{U}_1 }
\langle   Q^{-\nicefrac{1}{2}} (  R^* v ), 
u  \rangle_U
\langle   Q^{-\nicefrac{1}{2}} ( R^* w ),   
u  \rangle_U 
\\
&
=
\sum\nolimits_{ u \in \mathbb{U}_0 \cup \mathbb{U}_1 }
\langle   Q^{-\nicefrac{1}{2}} (  R^* v ), 
u  \rangle_U
\langle   Q^{-\nicefrac{1}{2}} ( R^* w ),   
u  \rangle_U 
\\
&
=
\langle
Q^{ - \nicefrac{1}{2} }
(R^* v),
Q^{-\nicefrac{1}{2}}
( R^* w)
\rangle_U
=
\langle
R^* v, R^* w 
\rangle_{Q^{\nicefrac{1}{2}}(U)}
=
\langle v, R R^* w \rangle_H.
\end{split}
\end{equation}
Next observe that 
the assumption that
$ W $
is a $ Q $-cylindrical
$ ( \mathcal{F}_t )_{t\in[0,T]} $-Wiener
process
and,
e.g.,
Proposition 6.1.16 in~\cite{j16}
ensure that
$ ( \Omega, \mathcal{F}, \P,
( \mathcal{G}^+_t)_{t\in [0,T]} ) $
is a stochastic basis
and that $ W $ is a
$ Q $-cylindrical
$ (\mathcal{G}_t^+)_{t\in [0,T]} $-Wiener
process.
This 
implies
that
for all
$ s \in [0,T) $, 
$ t \in (s,T] $
it holds that
$ \sigma_\Omega( \tilde W_t- \tilde W_s ) $
and
$ \mathcal{G}_s^+ $ 
are $ \P $-independent
and that
$ \tilde W $ is 
$ (\mathcal{G}_r^+ )_{r \in [0,T]} $-adapted.
Combining this
with~\eqref{eq:covariance} 
completes the proof of Lemma~\ref{lemma:Change_variables}.
\end{proof}
\begin{theorem}
\label{theorem:full_discrete_scheme_moments}
Assume the setting in Subsection~\ref{setting:exponential_moments_full_discrete},
  let $ \rho \in [0,\infty) $,  
  $ c, \iota \in [1, \infty) $,   
 $ \varsigma \in \big( 0, \tfrac{1 - 14 \delta}{ 2 + 2\iota } \big) $,
  $ \V \in \mathcal{C}_{c}^3( H, [0,\infty) ) $,
  $ \bar V \in \mathcal{C}(H, \R) $,  
assume for all $ h \in (0,T] $, $ x, y \in H $ that
 $ \V( e^{h A} x ) \leq \V(x) $,
    $ \bar \V( e^{h A} x ) \leq \bar \V(x) $,
     $ | \bar{\V}(x) - \bar{\V}(y) |
     \leq
     c \left( 1 + |\V(x)|^{\iota} + |\V(y)|^{\iota}\right)
     \| x - y \|_H $,
     $
     |\bar{\V}(x)|    
     \leq c \left( 1 + \left|\V(x)\right|^{\iota} \right) $,
    and
    $ \sup_{I\in \mathcal{P}_0 (\H)  } 
   \E\!\left[
   e^{
   	\V( P_I \xi )
   }
   \right] < \infty $,
  assume for all $ I \in \mathcal{P}_0 (\H) $, 
  $ J \in \mathcal{P}_0( \mathbb{U} ) $,
  $ h \in (0,T] $, 
  $ x \in D_h^I $ 
   that 
  $ D_h^I \subseteq \{ v \in H \colon \V(v) 
  \leq ch^{- \varsigma } \} $,
  $ \max\{\| P_I F(x)\|_{H } $,
$ \|P_I B(x) \hat P_J \|_{\HS(\U,H )} \}
    \leq c h^{-\delta} $,  
    and
$ ( \mathcal{G}_{P_I F,P_I B  \hat P_J } V )( x ) 
     +\frac{1}{2} \, \|(P_I B( x ) \hat P_J )^{*}(\nabla \V)( x ) \|_\U^2
     +\bar{ \V }( x )
   \leq
    \rho \V( x ) $,
  and 
  let
  $ Y^{\theta,I, J} \colon [0, T] \times \Omega \to P_I(H_\gamma) $,
  $ \theta \in \varpi_T $, 
  $ I\in \mathcal{P}_0 (\H) $, 
  $ J \in \mathcal{P}_0( \mathbb{U} ) $,
  be $ (\mathcal{F}_t)_{t\in [0,T]} $-adapted
stochastic processes  
with continuous sample paths
which satisfy for all
 $ \theta \in \varpi_T $, 
 $ I\in \mathcal{P}_0 (\H) $,
 $ J \in \mathcal{P}_0( \mathbb{U} ) $,
$ t \in (0,T] $
that
$ Y^{\theta, I, J}_0 = P_I(\xi) $ and
\begin{equation}
\begin{split}
 \label{eq:scheme} 
[ Y_t^{\theta, I, J} ]_{\P, \mathcal{B}(P_I(H_\gamma))} 
& 
= 
\left[
e^{(t-\llcorner t \lrcorner_\theta)A} 
\big( 
Y_{
	\llcorner t \lrcorner_\theta
}^{\theta, I, J} 
+
\1_{ D_{ | \theta |_T }^I } \! (
Y_{ \llcorner t \lrcorner_{ \theta }}^{\theta, I, J}
) 
P_I
F(
Y_{ \llcorner t \lrcorner_\theta }^{\theta, I, J}
)  \,
(
t - \llcorner t \lrcorner_\theta
) 
\big) 
\right]_{\P, \mathcal{B}(P_I(H_\gamma))}  
\\
&
\quad
+
\tfrac{
	\int_{ \llcorner t \lrcorner_\theta }^t
	e^{(t-\llcorner t \lrcorner_\theta )A} 
	\,
		\1_{ D_{ | \theta |_T }^I }\!\! (
		Y_{ \llcorner t \lrcorner_\theta }^{\theta, I, J}
		)
		\,
	P_I
	B( Y_{ \llcorner t \lrcorner_\theta }^{\theta, I, J} )
	\hat P_J
	\, dW_s
}{
1 + 
 \|
\int_{\llcorner t \lrcorner_\theta}^t 
P_I
B( Y_{ \llcorner t \lrcorner_\theta }^{\theta, I, J} )
\hat P_J
\, dW_s
 \|_H^2
} 
.
\end{split}
\end{equation}
Then it holds that
\begin{equation}  
\begin{split}  
\label{eq:exp.mom.full}
&
    \limsup_{ 
      \left | \theta \right |_T
      \searrow 0
    }    
    \sup_{ I \in \mathcal{P}_0 (\H) }
    \sup_{ J \in \mathcal{P}_0( \mathbb{U} ) }
    \sup_{ t \in [0,T]}
    \!
    \E\!\left[
      \exp\!\left( 
        \tfrac{
          \V( Y_t^{ \theta, I, J} )
        }{
          e^{ \rho t }
        }
        +
        \smallint_0^t
        \1_{ D_{  | \theta  |_T }^I}\!(
        Y_{ \lfloor s \rfloor_{ \theta } }^{ \theta, I, J }
        )
        \,
          \tfrac{ 
            \bar{\V}( Y_s^{ \theta, I, J } )
          }{
            e^{ \rho s }
          }
        \, ds 
      \right) 
    \right]
    \\
    &
    \leq 
      \sup_{I\in \mathcal{P}_0 (\H) } 
      \E\!\left[
        e^{
          \V( P_I \xi )
        }
      \right]
 \\
  &
  \leq
  \sup_{\theta \in \varpi_T}
  \sup_{I \in \mathcal{P}_0 (\H) }
  \sup_{ J \in \mathcal{P}_0( \mathbb{U} ) } 
  \sup_{t\in [0,T]}
  \E\!\left[
       \exp\!\left(
         \tfrac{  
        V ( Y^{ \theta, I, J }_t )
         }
         {
         e^{ \rho t }
         }
         +
         \smallint_0^t
         \1_{ D_{  | \theta  |_T }^I}\!(
         Y_{ \lfloor s \rfloor_{ \theta } }^{ \theta, I, J }
         ) 
         \,
             \tfrac{ 
           \bar{\V}( Y^{ \theta, I, J}_s )
           }
           {e^{   \rho s }}
           \,
          ds
       \right)
     \right]
< \infty
    .
\end{split}    
 \end{equation}
\end{theorem}
\begin{proof}[Proof of Theorem~\ref{theorem:full_discrete_scheme_moments}]
Throughout this proof 
let 
$ \mathcal{G}_t \subseteq \mathcal{F} $,
$ t \in [0,T] $,
be the sets which satisfy for all $ t \in [0,T] $ that
\begin{equation} 
\mathcal{G}_t =  
\sigma_\Omega 
\big(
\mathcal{F}_t
\cup 
\{ C \in \mathcal{F}
\colon \P(C) = 0 \}
\big),
\end{equation} 
let
$ u_0 \in \mathbb{U} $,
let
$ W^J \colon [0,T] \times \Omega 
\to \hat P_{J \cup \{u_0\} }(U) $, 
$ J \in \mathcal{P}_0( \mathbb{U} ) $,
be 
stochastic processes 
with continuous sample paths
which satisfy
for all 
$ J \in \mathcal{P}_0( \mathbb{U} ) $,
$ t \in [0,T] $ that
$ [W_t^J]_{\P, \mathcal{B}( U )}
=
\int_0^t \hat P_{J \cup \{u_0 \} } \, dW_s $
and
$ W_0^J = 0 $,
let
$ \tilde F_I \colon H \to H $,
$ I \in \mathcal{P}_0 (\H) $,
and
$ \tilde B_{I,J} \colon H \to \HS( \hat P_{J \cup \{u_0\} }(U),H) $,
$ I \in \mathcal{P}_0 (\H) $,
$ J \in \mathcal{P}_0( \mathbb{U} ) $,
be the functions 
which satisfy for all 
$ I \in \mathcal{P}_0 (\H) $,
$ J \in \mathcal{P}_0( \mathbb{U} ) $,
$ x \in H $,
$ u \in \hat P_{J \cup \{u_0\} }(U) \subseteq U $
that
\begin{equation}
\begin{split}
\tilde F_I(x)
=
\begin{cases}
P_I(F(x)) &\colon x \in H_\gamma
\\
0 &\colon x \in H \backslash H_\gamma
\end{cases}
\qquad
\text{and}
\qquad
\tilde B_{I,J}(x)u
=
\begin{cases}
P_I( B(x) \hat P_J u) & \colon x \in H_\gamma
\\
0 &\colon x \in H \backslash H_\gamma
\end{cases}
,
\end{split}
\end{equation}
and
let
 $ \tilde Y^{\theta,I,J} \colon [0, T] \times \Omega \to H $,
 $ \theta \in \varpi_T $, 
 $ I\in \mathcal{P}_0 (\H) $, 
 $ J \in \mathcal{P}_0( \mathbb{U} ) $,
 be  
 the 
 functions
 which satisfy
 for all
  $ \theta \in \varpi_T $, 
  $ I\in \mathcal{P}_0 (\H) $, 
  $ J \in \mathcal{P}_0( \mathbb{U} ) $,
 $ t \in [0,T] $
 that
 $ \tilde Y_0^{\theta, I, J} = P_I(\xi) $
 and
 \begin{equation}
 \begin{split} 
 \label{eq:scheme.tilde}
 &
 \tilde Y^{\theta,I,J}_t  
 \\
 & 
 =
  e^{(t - \llcorner t \lrcorner_\theta) A}  
  \bigg( 
  \tilde Y^{\theta, I, J}_{\llcorner t \lrcorner_\theta}
 +  
 \1_{  D_{|\theta|_T}^I  }
 \!  (\tilde Y_{\llcorner t \lrcorner_{\theta}}^{\theta, I, J} )
 \bigg[
  \tilde F_I ( \tilde Y^{\theta,I,J}_{\llcorner t \lrcorner_{\theta}} )
  \,
 (t - \llcorner t \lrcorner_\theta) 
 +  
 \tfrac{
 	 \tilde B_{I,J}( \tilde Y^{\theta,I,J}_{\llcorner t \lrcorner_{\theta}}  ) (W^J_t - W^J_{\llcorner t \lrcorner_\theta} )}
 {
 	1 +   \|  \tilde B_{I,J}( \tilde Y^{\theta,I,J}_{\llcorner t \lrcorner_{\theta}}  ) (W^J_t - W^J_{\llcorner t \lrcorner_\theta} )\|_{H }^2
 }  
 \bigg]
 \bigg).
 \end{split}
 \end{equation} 
In the next step observe that, e.g., 
Theorem~2.4 in Chapter V in Parthasarathy~\cite{Parthasarathy67}
ensures that
$ \mathcal{B}(H_\gamma) \subseteq \mathcal{B}(H) $.
This implies that for all
$ I \in \mathcal{P}_0 (\H) $,
$ J \in \mathcal{P}_0( \mathbb{U} ) $,
$ h \in (0,T] $
it holds that   
\begin{equation} 
\label{eq:condition_measurability}
D_h^I \in \mathcal{B}(H),
\quad
 \tilde F_I \in \mathcal{M}( \mathcal{B}(H), \mathcal{B}(H) ),
 \quad
 \text{and}
 \quad
\tilde B_{I,J} \in \mathcal{M}( \mathcal{B}(H), \mathcal{B}(\HS( \hat P_{J \cup \{u_0\}}(U), H)) )
.
\end{equation}
In addition, note that, e.g., Proposition~6.1.16 in~\cite{j16}
ensures that 
$ ( \mathcal{G}_t^+ )_{t\in [0,T]} $
is a normal filtration
on 
$ ( \Omega, \mathcal{F}, \P ) $
and that
$ W $ is an $ \operatorname{Id}_U $-cylindrical
$ ( \mathcal{G}_t^+ )_{t\in [0,T]} $-Wiener process.
Lemma~\ref{lemma:Change_variables}  
(with
$ H = \hat P_{J \cup \{u_0\}}(U) $,
$ U = U $,
$ R = ( U \ni u \mapsto \hat P_{J \cup \{u_0\}}(u) \in \hat P_{J \cup \{u_0\}}(U) ) $,
$ Q = \operatorname{Id}_U $,
$ ( \mathcal{F}_t)_{t\in [0,T]} 
= ( \mathcal{F}_t)_{t\in [0,T]} $,
$ W = W $,
$ \tilde W = W^J $
for $ J \in \mathcal{P}_0(\mathbb{U}) $
in the notation of Lemma~\ref{lemma:Change_variables})
hence assures that
for all $ J \in \mathcal{P}_0(\mathbb{U}) $
it holds that
$ W^J $ is an
$ (
( U \ni u \mapsto \hat P_{J \cup \{u_0\}} (u) \in\hat P_{J \cup \{u_0\}}(U) )
( U \ni u \mapsto \hat P_{J \cup \{u_0\}}(u) \in\hat P_{J \cup \{u_0\}}(U) )^*
 ) $-standard
$ ( \mathcal{G}_t^+ )_{t\in [0,T]} $-Wiener process.
This 
shows that for all
$ J \in \mathcal{P}_0( \mathbb{U} ) $
it holds that
$ W^J $ is an 
$ \operatorname{Id}_{\hat P_{J \cup \{u_0\}}(U)} $-standard
$ (\mathcal{G}_t^+)_{t\in [0,T]} $-Wiener process 
with continuous sample paths. 
Combining  
the fact that
for all
$ \theta \in \varpi_T $, 
$ I\in \mathcal{P}_0 (\H) $, 
$ J \in \mathcal{P}_0( \mathbb{U} ) $ 
it holds that  
$ \tilde Y^{\theta, I, J} $
is a $ ( \mathcal{G}_t^+  )_{t\in [0,T]} $-adapted
stochastic process with continuous sample paths,
the fact 
that  
$ \forall \, I \in \mathcal{P}_0 (\H) $,
$ h \in (0,T] \colon D_h^I \subseteq \{ v \in H \colon \V(v) 
\leq ch^{- \varsigma } \} $,
\eqref{eq:condition_measurability},
and
item~\eqref{item:Exp.Mom.Bound} of
Proposition~\ref{proposition:Euler.bounded.increments}
(with 
$ H = H $, 
$ U = \hat P_{J \cup \{u_0\}}(U) $,
$ T = T $,
$ \rho = \rho $,
$ \delta = \delta $,
$ c = c $,
$ \gamma = \iota $,
$ \varsigma = \varsigma $,
$ F = \tilde F_I $,
$ B = \tilde B_{I,J} $, 
$ V = V $,
$ \bar V = \bar V $,
$ S = 
 ( 
(0,T]
\ni t
\mapsto
(
H \ni x \mapsto
e^{tA} 
x
\in H
)
\in L(H)  
 ) $, 
$ D_h = D_h^I $, 
$ ( \mathcal{F}_t)_{t\in [0,T]}
= ( \mathcal{G}_t^+ )_{t\in [0,T]} $,
$ W = W^J $, 
$ Y^\theta = \tilde Y^{\theta, I, J} $ 
for 
$ h \in (0,T] $,
$ \theta \in \varpi_T $,
$ I \in \mathcal{P}_0 (\H) $,
$ J \in \mathcal{P}_0( \mathbb{U} ) $  
in the notation of Proposition~\ref{proposition:Euler.bounded.increments})
hence proves that
for all $ \theta \in \varpi_T $,
$ I \in \mathcal{P}_0(\H) $,
$ J \in \mathcal{P}_0(\mathbb{U}) $
it holds that
\begin{equation}
\begin{split}
&
\sup_{t\in [0,T]}\!\!
\E\Big[
\!
\exp\Big(  
\tfrac{
	\V( \tilde Y^{ \theta, I, J }_t )
}{ e^{ \rho t } } 
+  
\smallint_0^t
\1_{ D_{ |\theta|_T }^I }\!(
\tilde Y^{ \theta, I, J }_{\lfloor s \rfloor_{ \theta } }
	)
	\,
\tfrac{ 
	\bar{\V}( \tilde Y^{ \theta, I, J }_s )
}{
e^{ \rho s }
}
\,
ds 
\Big) 
\Big]
\\ 
&
\leq
\exp\!\left(
\tfrac{    
	\exp \left(\!
	2
	[720 \max\{T, \rho, 1\}   c^3]^{ 
		(720 c^3\!\max\{T, 1\} + 7)
		\iota
	}
	\right)
}{
[
\min\{ |\theta|_T, 1\}
]^{       
	 \varsigma + \varsigma \iota
	+ 7 \delta -\nicefrac{1}{2}
}
}
\right)
\E \big[ 
e^{ \V( Y_0^{ \theta, I, J } ) }
\big]
.
\end{split}
\end{equation}
This, the fact that
$ \nicefrac{1}{2} 
- \varsigma 
- \varsigma \iota
- 7 \delta > 0 $,
and the assumption
that
$ \sup_{I\in \mathcal{P}_0(\H) } 
\E[ e^{V(P_I(\xi))} ] < \infty $
imply that
\begin{equation}  
\begin{split}  
\label{eq:exp.mom.full.tilde}
&
\limsup_{ 
	\left | \theta \right |_T
	\searrow 0
}    
\sup_{ I \in \mathcal{P}_0 (\H) }
\sup_{ J \in \mathcal{P}_0( \mathbb{U} ) }
\sup_{ t \in [0,T]} 
\E\!\left[
\exp\!\left( 
\tfrac{
	\V( \tilde Y_t^{ \theta, I, J} )
}{
e^{ \rho t }
}
+
\smallint_0^t
	\1_{ D_{  | \theta  |_T }^I} \!(
	\tilde Y_{ \lfloor s \rfloor_{ \theta } }^{ \theta, I, J }
	)
	\,
\tfrac{ 
	\bar{\V}( \tilde Y_s^{ \theta, I, J } )
}{
e^{ \rho s }
}
\, ds 
\right) 
\right]
\\
&
\leq 
\sup_{ I\in \mathcal{P}_0 (\H) } 
\E\!\left[
e^{
	\V( P_I \xi )
}
\right]
\\
& 
\leq
\sup_{\theta \in \varpi_T}
\sup_{I \in \mathcal{P}_0 (\H) }
\sup_{J \in \mathcal{P}_0( \mathbb{U} ) } 
\sup_{t\in [0,T]}
\E\!\left[
\exp\!\left(
\tfrac{  
	V ( \tilde Y^{ \theta, I, J }_t )
}
{
	e^{ \rho t }
}
+
\smallint_0^t
\1_{ D_{  | \theta  |_T }^I} \!(
\tilde Y_{ \lfloor s \rfloor_{ \theta } }^{ \theta, I, J }
)
\,
\tfrac{
	\bar{\V}(\tilde Y^{ \theta, I, J}_s )
}
{e^{   \rho s }}
\,
ds
\right)
\right]
< \infty
.
\end{split}    
\end{equation}
Furthermore, note that~\eqref{eq:scheme}
and~\eqref{eq:scheme.tilde}
ensure that for all 
$ \theta \in \varpi_T $,
$ I \in \mathcal{P}_0 (\H) $,
$ J \in \mathcal{P}_0( \mathbb{U} ) $,
$ t \in [0,T] $
it holds that
$ [ Y_t^{\theta, I, J}]_{\P, \mathcal{B}(P_I(H_\gamma))}
=
[ \tilde Y_t^{\theta, I, J}]_{\P, \mathcal{B}(P_I(H_\gamma))} $.
Combining this, Lemma~\ref{lemma:switch_modification},
and~\eqref{eq:exp.mom.full.tilde}  establishes~\eqref{eq:exp.mom.full}.
The proof of Theorem~\ref{theorem:full_discrete_scheme_moments}
is thus completed.
\end{proof}
\begin{corollary}
\label{Corollary:full_discrete_scheme_convergence}
Assume the setting in Subsection~\ref{setting:exponential_moments_full_discrete},
let 
$ \vartheta \in
[ \sup_{x\in H_\gamma}\| B(x) \|_{\HS(U, H )}^2,
\infty] \cap \R $,
$ b_1, b_2 \in [0,\infty) $,
$ \varepsilon \in (0,\infty) $, 
$ \varsigma \in \big( 0, \tfrac{ 1 - 14 \delta  }{ 4}  \big) $,   
$ c\in [ 
2\max\{ 1, 
\varepsilon b_1,
\varepsilon \sqrt{\vartheta}, \varepsilon \},\infty) $,  
assume that
$ \E[e^{\varepsilon \| \xi \|_{H}^2}]<\infty $,
assume for all $ h \in (0,T] $, 
$ I \in \mathcal{P}_0(\H) $, 
$ J \in \mathcal{P}_0(\mathbb{U}) $, 
$ x \in D_h^I $ that
$ D_h^I \subseteq \{ v \in H \colon \sqrt{\vartheta} +   \varepsilon\|  v \|_H^2 \leq  c  h^{- \varsigma }  \} $,
$ \max\{\| P_I F(x) \|_H,
\|  P_I B(x) \hat P_J \|_{\HS(U, H  )}
\} 
\leq c h^{-\delta} $,
and
$ \left< x, P_I F(x) \right>_H
\leq b_1 + b_2 \|   x\|_H^2 $,
 and
 let
 $ Y^{\theta, I, J }\colon 
 [0, T] \times \Omega \to P_I(H_\gamma) $,
 $ \theta \in \varpi_T $, 
 $ I \in \mathcal{P}_0(\H) $, 
 $ J \in \mathcal{P}_0(\mathbb{U}) $,
 be
 $ (\mathcal{F}_t)_{t\in [0,T]}$-adapted
stochastic processes 
with continuous sample paths
which satisfy for all 
$ \theta \in \varpi_T $, 
$ I \in \mathcal{P}_0(\H) $, 
$ J \in \mathcal{P}_0(\mathbb{U}) $,
$ t\in (0,T] $ 
that
$ Y^{\theta, I, J}_0 = P_I(\xi) $ and
\begin{equation}
\begin{split} 
[ Y_t^{\theta, I, J} ]_{\P, \mathcal{B}(P_I(H_\gamma))} 
& 
= 
\left[
e^{(t-\llcorner t \lrcorner_\theta)A} 
\big( 
Y_{
	\llcorner t \lrcorner_\theta
}^{\theta, I, J} 
+
\1_{ D_{ | \theta |_T }^I } 
\!(
Y_{ \llcorner t \lrcorner_{ \theta }}^{\theta, I, J}
) 
P_I
F(
Y_{ \llcorner t \lrcorner_\theta }^{\theta, I, J}
) \,
(
t - \llcorner t \lrcorner_\theta
) 
\big) 
\right]_{\P, \mathcal{B}(P_I(H_\gamma))}  
\\
&
\quad
+
\tfrac{
	\int_{ \llcorner t \lrcorner_\theta }^t
		e^{(t - \llcorner t \lrcorner_\theta )A} 
		\,
	\1_{ D_{ | \theta |_T }^I } \!\!(
	Y_{ \llcorner t \lrcorner_\theta }^{\theta, I, J}
	)\,
	P_I
	B( Y_{ \llcorner t \lrcorner_\theta }^{\theta, I, J} )
	\hat P_J
	\, dW_s
}{
1 + 
 \|
\int_{ \llcorner t \lrcorner_\theta}^t 
P_I
B( Y_{  \llcorner t \lrcorner_\theta }^{\theta, I, J} )
\hat P_J
\, dW_s
 \|_H^2
} 
.
\end{split}
\end{equation}
  Then it holds that
  \begin{equation}
  \begin{split} 
  \label{eq:exponential_moments} 
  &
  \limsup_{|\theta|_T \searrow 0}
  \sup_{I \in \mathcal{P}_0(\H)}
  \sup_{J \in \mathcal{P}_0(\mathbb{U})} 
  \sup_{t\in [0,T]}
  \E\!\left[
       \exp \! \left (
         \tfrac{  
        \sqrt{\vartheta} +
         \varepsilon 
        \| Y^{ \theta, I, J}_t  \|_H^2
         }
         {
         e^{ 2 (b_2 + \varepsilon \vartheta ) t }
         }
         -  
         \smallint_0^t
           \,
\1_{ D_{  | \theta  |_T }^I}\!(
              Y_{ \lfloor s \rfloor_{ \theta } }^{ \theta, I, J }
            )
            \,
            \tfrac{
            \varepsilon 
         ( 2 b_1 + \vartheta )
         }
         {
           e^{  2 (b_2 + \varepsilon \vartheta )  s }
           }
          \, ds
       \right)
     \right]
\\
&
\leq 
\E\big[e^{ \sqrt{\vartheta} + \varepsilon \| \xi \|_H^2} \big] 
\\
&
\leq
  \sup_{ \theta \in \varpi_T }
  \sup_{ I \in \mathcal{P}_0(\H) }
  \sup_{ J \in \mathcal{P}_0(\mathbb{U})} 
  \sup_{t\in [0,T]}
  \E\!\left[ 
       \exp \! \left(
         \tfrac{   
         \sqrt{\vartheta} +
       \varepsilon \| Y^{ \theta, I, J}_t \|_H^2
         }
         {
         e^{ 2(b_2 + \varepsilon \vartheta) t }
         }
         -  
         \smallint_0^t
           \,
\1_{ D_{  | \theta  |_T }^I}\!(
              Y_{ \lfloor s \rfloor_{ \theta } }^{ \theta, I, J}
            ) 
            \,
            \tfrac{
            \varepsilon 
         ( 2 b_1 + \vartheta )
         }
         {
            e^{   2(b_2 + \varepsilon \vartheta) s }  
            }
          \, ds
       \right) 
     \right]
     \\
     &
     \leq
     e^{ \sqrt{ \vartheta } }
     \sup_{ \theta \in \varpi_T }
  \sup_{ I \in \mathcal{P}_0(\H) }
  \sup_{ J \in \mathcal{P}_0(\mathbb{U})} 
  \sup_{t\in [0,T]}
  \E\!\left[ 
       \exp \! \left(
         \tfrac{    
       \varepsilon \| Y^{ \theta, I, J}_t \|_H^2
         }
         {
         e^{ 2(b_2 + \varepsilon \vartheta) t }
         } 
       \right) 
     \right]
     < \infty.
  \end{split}
  \end{equation}
\end{corollary}
\begin{proof}[Proof of Corollary~\ref{Corollary:full_discrete_scheme_convergence}]
Throughout  this  
proof  
let $ V \colon H \to [0,\infty) $
and
$ \bar V \colon H \to \R $
be the functions with the property that 
for all $ x \in H $ it holds
that
$ V(x) = \sqrt{\vartheta} + \varepsilon \|  x \|_H^2 $
and
$ \bar V (x) = - 2\varepsilon b_1 - \varepsilon \vartheta $.
First of all, observe that for all $ x, y \in H $ it holds that
$ | V(x) - V(y) | \leq   
2 \sqrt{\varepsilon}
\|   x - y   \|_H ( 1 + 
\sup_{r\in [0,1]}|V(rx + (1-r)y) |)^{\nicefrac{1}{2}} $, 
$ \| V'(x) - V'(y) \|_{L^{(1)}(H, \R)} 
\leq
2   \varepsilon \|  x - y   \|_H $,
and
$ |V''(x) - V''(y) \|_{L^{(2)}(H, \R)} = 0 $.
Hence, we obtain that 
\begin{equation}
\label{eq:fun_regularity}
V \in \mathcal{C}^3_{ 2 \max\{1, \varepsilon\}}(H, [0,\infty)). 
\end{equation}
Next note that the assumption  that
$ \forall\, h\in (0,T] $, 
$ I \in \mathcal{P}_0(\H) $, 
$ J \in \mathcal{P}_0(\mathbb{U}) $,
$ x \in D_h^I   \colon 
\langle x,  P_I F(x)\rangle_{H } 
\leq b_1 + b_2 \| x \|_H^2 $ 
shows that for all $ h \in (0,T] $, 
$ I \in \mathcal{P}_0(\H) $, 
$ J \in \mathcal{P}_0(\mathbb{U}) $,
$ x \in D_h^I $ it holds that
\begin{equation}
\begin{split}
&
(\mathcal{G}_{P_I F, P_I B  \hat P_J}\V)(x)
     +\tfrac{1}{2}\big \|(P_I B(x) \hat P_J )^{*}(\nabla \V)(x)\big \|_\U^2
     +\bar{ \V }(x)
    \\
& 
=
  2 \varepsilon
  \left<   x,   P_I F( x ) \right>_H
  +
  \varepsilon
  \sum\nolimits_{u \in \mathbb{U} } 
  \langle  P_I B(x) \hat P_J  u, P_I B(x) \hat P_J u \rangle_U 
  \\
  &
  \quad
  +
  2 \varepsilon^2
  \|
     (P_I B(x) \hat P_J )^*
     x
  \|^2_\U
  -2\varepsilon b_1
  - \varepsilon \vartheta
\\
&
=
2 \varepsilon
\left<    x,    P_I F( x ) \right>_H
+
\varepsilon
\|   P_I B(x) \hat P_J   \|_{\HS(U,H)}^2
+
2 \varepsilon^2
\|
(P_I B(x) \hat P_J )^*
  x
\|^2_\U
-2\varepsilon b_1
- \varepsilon \vartheta
\\
&
\leq
2\varepsilon
 (
 b_2  
 +
 \varepsilon
 \vartheta ) \| x \|_H^2
 \leq
 2
 (b_2 + \varepsilon\vartheta)V(x)
  .
\end{split}
\end{equation}
Combining
this,
\eqref{eq:fun_regularity},
the fact that
$ \sup_{I \in \mathcal{P}_0(\H) }
\E[ e^{V(P_I \xi)}]
\leq
e^{\sqrt{\vartheta}}
\,
\E[ e^{\varepsilon \| \xi \|_H^2 } ] $,
the assumption that
$ \E[ e^{\varepsilon \| \xi \|_H^2 } ] < \infty $,
the fact that
$ \forall \, x \in H \colon
| \bar V(x) | \leq c ( 1 + | V(x) | ) $,
and Theorem~\ref{theorem:full_discrete_scheme_moments}
(with 
$ \rho = 2 (b_2+\varepsilon \vartheta) $,
$ c = c $, 
$ \iota = 1 $,
$ \delta = \delta $,
$ \varsigma = \varsigma $,
$ V = V $,
$ \bar V = \bar V $,   
$ Y^{\theta, I, J} = Y^{\theta, I, J} $
for $ \theta \in \varpi_T $,
$ I \in \mathcal{P}_0(\H) $,
$ J \in \mathcal{P}_0(\mathbb{U}) $
in the notation of Theorem~\ref{theorem:full_discrete_scheme_moments})
establishes~\eqref{eq:exponential_moments}.
The proof of Corollary~\ref{Corollary:full_discrete_scheme_convergence}
is thus completed.
\end{proof}
\section{Examples}
\label{Section:4}
In this section we illustrate 
Corollary~\ref{Corollary:full_discrete_scheme_convergence}
by some examples.
In particular, we prove in the case of
a class of stochastic Burgers equations (see Subsection~\ref{sec:Burgers}),
stochastic Kuramoto-Sivashinsky equations (see Subsection~\ref{sec:Kuramoto}), 
and
two-dimensional stochastic Navier-Stokes equations (see Subsection~\ref{sec:2DNavier})
that a certain
tamed
and space-time-noise
discrete approximation scheme
(see~\eqref{scheme:full_discrete} below)
has bounded exponential moments.
\subsection{Setting}
\label{setting:Examples}
Let $ d \in \N $, 
$ \D =(0,1)^d $, 
$ \eta, \gamma \in [0,\infty) $, 
$ T, \varepsilon \in (0,\infty) $,   
$ \delta \in (0, \nicefrac{1}{18}) $,  
$ (U, \langle \cdot, \cdot \rangle_U,
\left \| \cdot \right \|_U)  =
(L^2(\mu_\D; \R^d), 
\langle \cdot, \cdot 
\rangle_{L^2(\mu_\D;\R^d)},
\left \| \cdot \right \|_{L^2(\mu_\D; \R^d)}
) $, 
let $ \mathbb{U} \subseteq U $ be an orthonormal basis of $ U $,
let $ H \subseteq U $ be a closed subvector space
of $ U $,
  let 
$ \H \subseteq H $
be a non-empty orthonormal basis of $ H $, 
  let
  $ \lambda \in \mathbb{M}( \H, \R) $
  satisfy that
  $ \sup ( \operatorname{im}(\lambda) ) < 0 $,
  let
 $ ( \Omega, \mathcal{ F }, \P, ( \mathcal{F}_t)_{t\in [0,T]} ) $
 be a  
 filtered probability space,
 let
$ (W_t)_{t\in [0,T]} $
be an
$ \operatorname{Id}_\U $-cylindrical 
$( \mathcal{F}_t)_{t\in [0,T]} $-Wiener process,
let $ Q \in L(U) $
be a non-negative
symmetric trace class operator, 
let
$ A \colon D( A ) \subseteq H \to H $ 
be the linear operator 
which satisfies for all $ v \in D(A) $ that
$ D(A) = \big \{
w \in H \colon \sum_{ h\in \H} \left | \lambda_h \left<  h, w \right>_H \right | ^2 <  \infty
\big\} $
and 
$ A v = \sum_{ h\in \H } \lambda_h \left<  h, v \right> _H  h $,
let
$ (H_r, \left< \cdot, \cdot \right>_{H_r}, \left \| \cdot \right \|_{H_r} ) $, $ r\in \R $, be a family of interpolation spaces associated to 
$ - A $
(see, e.g., Definition 3.6.30 in~\cite{j16}),
let  
$ r\in \mathcal{M}( \mathcal{B}(H_\gamma), \mathcal{B}([0,\infty) ) ) $,
$ b \in \mathcal{M}( \mathcal{B}(\D  \times  \R^d), \mathcal{B}(\R^{d \times d})) $ 
satisfy 
$  
\sup\nolimits_{x \in \D, y\in \R^d, z \in \R^d \backslash \{ y \} }
\big(
 \| b(x,y) \|_{ \R^{ d \times d } }
+
\tfrac{ \| b(x,y) - b(x,z) \|_{ \R^{ d \times d } } }
{\| y - z \|_{\R^d}}
\big)
< \infty
 $,
let
$ \vartheta =   
\tr_U(Q) 
(\sup_{x\in \D, y\in \R^d }
\|b(x,y)\|_{ \R^{ d \times d } }^2)
$,
$ c \in [2 \max\{1, 
\varepsilon \sqrt{\vartheta}, 
\varepsilon  \}, \infty) $,
let 
$ P_I \in L(H) $,
$ I\in \mathcal{P} (\H) $,
and   
$ \hat P_J \in L(U) $,
$ J \in \mathcal{P} (\mathbb{U}) $,
be the linear operators which satisfy
for all 
$ I\in \mathcal{P}(\H) $, 
$ J \in \mathcal{P}(\mathbb{U}) $, 
$ v \in H $,
$ w \in U $
that
$ P_I(v) =\sum_{h \in I} \left< h ,v \right>_H h $ 
and
$ \hat P_J (w) = \sum_{u\in \mathbb{U}}
\left< u, w \right>_U u $,
for every
$ I \in \mathcal{P}_0(\H) $,
$ h\in (0,T] $ 
let
$ D_h^I \in \mathcal{P}(H_\gamma) $
be the set given by 
$ D_h^I =\{ x\in P_I(H_\gamma) \colon r(x) \leq
c h^{-\delta}\} $,
let $ R \in L(U) $  
be the orthogonal projection
of $ U $ on $ H $,
for every 
$ n \in \N $,
$ v \in W^{1,2}(\D, \R^n ) $
let
$ \partial v = ( \partial_1 v, \ldots, \partial_d v) 
 \in L^2(\mu_\D; \R^{n \times d})  $
be the 
vector which satisfies for all
$ i \in \{ 1,\ldots, d\} $,
$ \phi  
\in 
\mathcal{C}_{cpt}^\infty(\D, \R^n) $
that
$ \langle \partial_i v, 
[ \phi ]_{\mu_\D, \mathcal{B}(\R^n)} \rangle_{L^2(\mu_\D; \R^n)}
=
-
\langle v , 
[ \frac{\partial}{\partial x_i} \phi  ]_{\mu_\D, \mathcal{B}(\R^n)}
\rangle_{L^2(\mu_\D; \R^n)} $,
let $ F \in \mathbb{M}(H_\gamma, H) $,
$ B \in \mathbb{M}(H_\gamma, \HS(U,H)) $,
$ \xi \in \mathcal{M}(\mathcal{F}_0, \mathcal{B}(H_\gamma)) $
satisfy
for all 
$ u \in U $, 
$ v \in \mathcal{M}(\mathcal{B}(\D), \mathcal{B}(\R^d) ) $,
$ w \in [ H_\gamma 
\cap  W^{1,2}( \D, \R^d) 
\cap
L^\infty( \mu_{\mathcal{D}}; \R^d) ] $
with
$ [ v ]_{\mu_\D, \mathcal{B}(\R^d)} \in  H_\gamma $
that
$ \E[ e^{\varepsilon \| \xi \|_H^2} ] < \infty $,
  $ B(
  [ v ]_{\mu_{\D}, \mathcal{B}(\R^d)}
  )
  u
  = R
  \big(
  [
  \{ b(x, v(x))  \}_{x\in \D} 
  ]_{\mu_{\D}, \mathcal{B}(\R^{d \times d})} 
  ( \sqrt{Q} u ) \big) $,
  and
  $ F(w) =
  R( 
  \eta w
  - 
  \sum_{i=1}^d w_i \partial_i w 
  ) $,
and
let 
$ Y^{\theta,I,J}\colon [0, T] \times \Omega \to P_I(H) $,
$ \theta \in \varpi_T $, 
$ I\in \mathcal{P}_0(\H) $, 
$ J \in \mathcal{P}_0(\mathbb{U}) $,
be  
$ (\mathcal{F}_t)_{t\in [0,T]} $-adapted
stochastic processes
with continuous sample paths
which satisfy for all 
$ t \in (0,T] $, 
$ \theta \in \varpi_T $, 
$ I \in \mathcal{P}_0 (\H) $, 
$ J \in \mathcal{P}_0(\mathbb{U}) $ 
that
$ Y^{\theta, I, J}_0 = P_I(\xi) $ and
\begin{equation}
\begin{split}
\label{scheme:full_discrete}
[ Y_t^{\theta, I, J} ]_{\P, \mathcal{B}(P_I(H_\gamma))} 
& 
= 
\left[
e^{(t-\llcorner t \lrcorner_\theta)A} 
\left( 
Y_{
	\llcorner t \lrcorner_\theta
}^{\theta, I, J} 
+
\,
\1_{  \{ 
	r (Y^{\theta,I,J}_{\llcorner t \lrcorner_{\theta}}  )
	\leq 
	c
	[|\theta|_T]^{-\delta}   \}}
P_I
F(
Y_{ \llcorner t \lrcorner_\theta }^{\theta, I, J}
) \,
(
t - \llcorner t \lrcorner_\theta
) 
\right) 
\right]_{\P, \mathcal{B}(P_I(H_\gamma))}  
\\
&
\quad
+
\tfrac{
	\int_{ \llcorner t \lrcorner_\theta }^t
		e^{(t-\llcorner t \lrcorner_\theta )A}  
		\,
	\1_{  \{ 
		r (Y^{\theta,I,J}_{ \llcorner t \lrcorner_{\theta}}  )
		\leq 
		c
		[|\theta|_T]^{-\delta}   \}}
	\,
	P_I
	B( Y_{ \llcorner t \lrcorner_\theta }^{\theta, I, J} )
	\hat P_J
	\, dW_s
}{
1 + 
\|
\int_{ \llcorner t \lrcorner_\theta}^t 
P_I
B( Y_{ \llcorner t \lrcorner_\theta }^{\theta, I, J} )
\hat P_J
\, dW_s
\|_H^2
} 
.
\end{split}
\end{equation}
\subsection{Properties of the nonlinearities}
In this subsection we establish
a few elementary
properties
for the nonlinearities $ F $ and $ B $
in Subsection~\ref{setting:Examples}
(see Lemma~\ref{lemma:F_well_defined},
Lemma~\ref{lemma:estimate_B_1},
Lemma~\ref{lemma:zero_coercivity},
Lemma~\ref{lemma:zero_coercivity2},
and Corollary~\ref{corollary_continuous} below).
To do so, we also recall in this subsection
some well-known
properties of the involved
Sobolev and interpolation spaces 
(see Lemmas~\ref{lemma:Linfty}--\ref{lemma:zero_coercivity} below).
\begin{lemma}
	\label{lemma:F_well_defined}
	Assume the setting in Subsection~\ref{setting:Examples} 
	and let  
	$ v, w \in [ H_\gamma 
	\cap  W^{1,2}( \D, \R^d) 
	\cap
	L^\infty( \mu_\D; \R^d) ] $.
	Then 
	it holds that
	\begin{equation}
	\begin{split} 
	\| F(v)  \|_{L^2(\mu_\D; \R^d)}
	& 
	\leq
	\eta 
	\| v \|_{ L^2(\mu_\D; \R^d) } 
	+
	d  
	\| v \|_{ L^\infty( \mu_\D; \R^d)}
	\|  \partial v \|_{ L^2(\mu_\D; \R^{d\times d} )}
	< \infty
	\end{split}
	\end{equation}
	and 
	\begin{equation}
	\begin{split}
	&
	\| F(v) - F(w) \|_{L^2(\mu_\D; \R^d)} 
	\leq
	\eta
	\| v - w \|_{ L^2(\mu_\D; \R^d) } 
	\\
	&
	+
	d
	\big(
	\| \partial v \|_{L^2(\mu_\D; \R^{d\times d} )} 
	\| v - w \|_{L^\infty(\mu_\D; \R^d)}
	+ 
	\| w \|_{L^\infty(\mu_\D; \R^d)}
	\| \partial( v - w ) \|_{L^2(\mu_\D; \R^{d\times d} )} 
	\big)
	<
	\infty 
	.
	\end{split}
	\end{equation}
\end{lemma}
\begin{proof}[Proof of Lemma~\ref{lemma:F_well_defined}]
	Note that
	the triangle inequality and
	H\"older's inequality imply
	that
	\begin{equation}
	\begin{split} 
	\| F(v)  \|_{ L^2(\mu_\D; \R^d) } 
	& \leq
	\eta
	\| v \|_{ L^2(\mu_\D; \R^d) } 
	+
	\big\|
	\smallsum_{j=1}^d
	v_j \partial_j v
	\big\|_{L^2(\mu_\D; \R^d)} 
	\\
	&
	\leq
	\eta
	\| v \|_{ L^2(\mu_\D; \R^d) } 
	+
	\smallsum_{j=1}^d \| v_j \|_{L^\infty( \mu_\D; \R ) }
	\| \partial_j v
	\|_{ L^2(\mu_\D; \R^d) }  
	\\
	&
	\leq
	\eta
	\| v \|_{ L^2(\mu_\D; \R^d) } 
	+
	\sqrt{
		\smallsum_{j=1}^d
		\| v_j \|_{L^\infty( \mu_\D; \R ) }^2
	}
	\sqrt{
		\smallsum_{j=1}^d
		\| \partial_j v
		\|_{ L^2(\mu_\D; \R^d) }^2
	} 
	\\
	& 
	\leq
	\eta
	\| v \|_{ L^2(\mu_\D; \R^d) } 
	+
	d 
	\| v \|_{ L^\infty( \mu_\D; \R^d)}
	\| \partial v \|_{ L^2(\mu_\D;  \R^{d\times d} ) }
	.
	\end{split}
	\end{equation}
	In addition, observe
	that
	\begin{equation}
	\begin{split}
	& 
	\| F(v) - F(w) \|_{L^2(\mu_\D; \R^d)}
	-
	\eta
	\| v - w \|_{ L^2(\mu_\D; \R^d) } 
	\leq  
	\big \| \smallsum_{j=1}^d (\partial_j v) v_j  -  \smallsum_{j=1}^d (\partial_j w) w_j 
	\big \|_{L^2(\mu_\D; \R^d)}
	\\ 
	&
	\leq  
	\,
	\big \|
	\smallsum_{j=1}^d (\partial_j v) (v_j - w_j)
	\big \|_{L^2(\mu_\D; \R^d)}
	+
	\,
	\big \| \smallsum_{j=1}^d 
	(\partial_j v - \partial_j w )  w_j
	\big \|_{L^2(\mu_\D; \R^d)}
	\\
	& 
	\leq 
	\smallsum_{j=1}^d 
	\| ( \partial_j v ) ( v_j - w_j ) 
	\|_{L^2(\mu_\D; \R^d)}
	+
	\smallsum_{j=1}^d
	\|
	(\partial_j v - \partial_j w ) w_j
	\|_{L^2(\mu_\D; \R^d)}
	\\
	& 
	\leq
	\smallsum_{j=1}^d 
	\|\partial_j v  
	\|_{L^2(\mu_\D; \R^d)}
	\| v_j - w_j 
	\|_{L^\infty(\mu_\D; \R )}
	+ 
	\smallsum_{j=1}^d
	\|
	\partial_j v - \partial_j w  
	\|_{L^2(\mu_\D; \R^d)}
	\|
	w_j
	\|_{L^\infty(\mu_\D; \R )}
	.
	\end{split}
	\end{equation}
	H\"older's inequality hence proves  
	that
	\begin{equation}
	\begin{split}
	& 
	\| F(v) - F(w) \|_{L^2(\mu_\D; \R^d)}  
	\leq
	\sqrt{
		\smallsum_{j=1}^d 
		\|\partial_j v  
		\|_{L^2(\mu_\D; \R^d)}^2
	}
	\sqrt{
		\smallsum_{j=1}^d
		\| v_j - w_j 
		\|_{L^\infty(\mu_\D; \R )}^2
	}
	\\
	&
	\quad
	+
	\sqrt{ 
		\smallsum_{j=1}^d
		\|
		\partial_j v - \partial_j w  
		\|_{L^2(\mu_\D; \R^d)}^2
	}
	\sqrt{
		\smallsum_{j=1}^d
		\|
		w_j
		\|_{L^\infty(\mu_\D; \R )}^2
	}
	+
	\eta
	\| v - w \|_{ L^2(\mu_\D; \R^d) } 
	\\ 
	&
	\leq
	d
	\big(
	\| \partial v \|_{L^2(\mu_\D; \R^{d\times d} )} 
	\| v - w \|_{L^\infty(\mu_\D; \R^d)}
	+ 
	\| w \|_{L^\infty(\mu_\D; \R^d)}
	\| \partial(v-w) \|_{L^2(\mu_\D; \R^{d\times d} )} 
	\big)
	\\
	&
	\quad
	+
	\eta
	\| v - w \|_{ L^2(\mu_\D; \R^d) } 
	.
	\end{split}
	\end{equation}
	The proof of Lemma~\ref{lemma:F_well_defined}
	is thus completed.
\end{proof}
\begin{lemma}
\label{lemma:estimate_B_1}
Assume the setting in Subsection~\ref{setting:Examples}.
Then it holds for all 
$ v, w \in H_\gamma $ that 
\begin{equation}
\begin{split}
&
\| B( v ) \|_{\HS( U, H )} 
\leq 
\left(
\sup\nolimits_{x\in \D, y\in \R^d}
\| b(x,y) \|_{  \R^{d \times d}  } 
\right)
\sqrt{ \operatorname{trace}_U(Q) }
=
\sqrt{\vartheta}
 < \infty
\end{split}
\end{equation} 
and
\begin{equation}
\begin{split} 
&
\| 
B( v )
-
B( w ) 
\|_{\HS( U, H )}
\\
&
\leq 
\Big(
\sup\nolimits_{x \in \D, y\in \R^d, z \in \R^d \backslash \{ y \} }
\tfrac{ \| b(x,y) - b(x,z) \|_{ \R^{ d \times d } } }{\| y - z \|_{\R^d}}
\Big)
\| v - w \|_{ L^\infty(\mu_\D; \R^d) }
\sqrt{ \operatorname{trace}_U(Q) } 
.
\end{split}
\end{equation} 
\end{lemma}
\begin{proof}[Proof of Lemma~\ref{lemma:estimate_B_1}]
First of all, note for all   
$ v  \in \mathcal{M}( \mathcal{B}(\D), \mathcal{B}(\R^d) ) $
with $ [ v ]_{\mu_\D, \mathcal{B}(\R^d)} \in H_\gamma $
that
\begin{equation}
\begin{split} 
\label{eq_needded1}
&
\| B( [ v ]_{\mu_\D, \mathcal{B}(\R^d) } ) \|_{\HS( U, H )}^2
=
\sum\nolimits_{u \in \mathbb{U}}
\| B( [ v ]_{\mu_\D, \mathcal{B}(\R^d) } ) u\|_H^2
\\
&
\leq
\sum\nolimits_{ u \in \mathbb{U} }
\big\| 
[
\{ b(x, v(x)) \}_{x\in \D}
]_{\mu_\D, \mathcal{B}(\R^{d \times d})}
( Q^{\nicefrac{1}{2}} u )
\big\|_U^2
\\
&
\leq
\left(
\sup\nolimits_{x\in \D, y\in \R^d}
\| b(x,y) \|_{ \R^{ d \times d } }^2
\right)
\sum\nolimits_{u \in \mathbb{U} } 
\|  
Q^{\nicefrac{1}{2}} u 
\|_U^2 
=
\left(
\sup\nolimits_{x \in \D, y\in \R^d}
\| b(x,y) \|_{ \R^{ d \times d } }^2
\right)
 \operatorname{trace}_U(Q).
\end{split}
\end{equation} 
Next observe for all 
$ v, w \in \mathcal{M}( \mathcal{B}(\D), \mathcal{B}(\R^d) ) $
with 
$ [ v ]_{\mu_\D, \mathcal{B}(\R^d)}, 
[ w ]_{\mu_\D, \mathcal{B}(\R^d)} 
\in  H_\gamma $
that
\begin{equation}
\begin{split} 
\label{eq_needded2}
&
\| 
B( [ v ]_{\mu_\D, \mathcal{B}(\R^d) } )
-
B( [ w ]_{\mu_\D, \mathcal{B}(\R^d) } ) 
\|_{\HS( U, H )}^2
\\
&
\leq
\sum\nolimits_{ u \in \mathbb{U} }
\big\| 
[
\{ b(x, v(x)) - b ( x, w(x) ) \}_{x \in \D}
]_{\mu_\D, \mathcal{B}(\R^{d \times d})}
( Q^{\nicefrac{1}{2}} u )
\big\|_U^2
\\
&
\leq  
\|  
\{ b(x,v(x))-b(x,w(x))
\}_{ x \in \mathcal{D} }  
\|_{
\mathcal{L}^\infty( \mu_{\mathcal{D}}; \R^{ d \times d } ) }^2  
\sum\nolimits_{u \in \mathbb{U} } 
\|  Q^{\nicefrac{1}{2}} u \|_U^2 
\\
&
\leq
\Big[
\sup\nolimits_{x \in \D, y\in \R^d, z \in \R^d \backslash \{ y \} }
\tfrac{ 
	\| b(x,y) - b(x,z) \|_{ \R^{ d \times d } }
	}{\| y - z \|_{\R^d}}
\Big]^2
\| v - w \|_{ \mathcal{L}^\infty(\mu_\D; \R^d) }^2
\operatorname{trace}_U(Q)
.
\end{split}
\end{equation} 
Combining~\eqref{eq_needded1}
and~\eqref{eq_needded2}
completes the proof of Lemma~\ref{lemma:estimate_B_1}.
\end{proof}
\begin{lemma}
	\label{lemma:Linfty}
	Assume the setting in Subsection~\ref{setting:Examples}
	and
	let
	$ \rho \in [0,\infty) $,
	$ v \in H_\rho $.
	Then 
	\begin{equation}
	\begin{split}
	\label{eq:estimate_infty_norm}
	&
	\| v \|_{L^\infty( \mu_\D; \R^d)} 
	\leq
	\| v \|_{H_\rho}
	\bigg(
	\sup_{ h \in \H } 
	\| h \|_{L^\infty( \mu_\D; \R^d)}
	\bigg)
	\Bigg[
	\sum_{ h \in \H }
	| \lambda_h |^{-2\rho}
	\Bigg]^{\nicefrac{1}{2}}
.
\end{split}
\end{equation}
\end{lemma}
\begin{proof}[Proof of Lemma~\ref{lemma:Linfty}]
	Note that H\"older's inequality proves  
	that
	\begin{equation}
	\begin{split}
	\sum_{ h \in \H }
	\| 
	\langle h, v \rangle_H
	\,
	h 
	\|_{L^\infty( \mu_\D; \R^d)}
	&
	\leq
	\bigg(
	\sup_{ h \in \H } 
	\| h 
	\|_{L^\infty( \mu_\D; \R^d)}
	\bigg)
	\sum_{ h \in \H }
	| \langle h, v \rangle_H |
	\\
	& 
	\leq
	\bigg(
	\sup_{ h \in \H } 
	\| h 
	\|_{L^\infty( \mu_\D; \R^d)}
	\bigg)
	\sum_{ h \in \H }
	|  \lambda_h |^\rho
	\,
	| \langle h, v \rangle_H |
	\,
	| \lambda_h |^{-\rho}
	\\
	&
	\leq
	\| v \|_{H_\rho}
	\bigg(
	\sup_{h \in \H } 
	\| h
	\|_{L^\infty( \mu_\D; \R^d)}
	\bigg) 
	\Bigg[
	\sum_{ h \in \H }
	|   \lambda_h |^{-2\rho}
	\Bigg]^{\nicefrac{1}{2}}
.
\end{split}
\end{equation}
This 
completes the proof
 of Lemma~\ref{lemma:Linfty}.
\end{proof}
\begin{lemma}
	\label{lemma:weak_exists}
		Assume the setting in Subsection~\ref{setting:Examples},
		let 
		$ \rho \in [0,\infty) $,   
		and assume
		for all
		$ j \in \{1, \ldots, d \} $,
		$ v, w \in \H $
		that   
		$ \H \subseteq W^{1,2} (\D, \R^d) $, 
		$ \sup_{h \in \H  }
		\big(
		\|  \partial_j h  \|_U
		\, | \lambda_h |^{-\rho} 
		\big) < \infty $, 
		$ \langle \partial_j v,  
		\partial_j  w \rangle_U
		\, \1_{ \H \backslash \{ v \} }(w) = 0 $.
		Then 
\begin{enumerate}[(i)]
	\item \label{item:Hrho_inclusion}	
	it holds that 
		$ H_\rho \subseteq W^{1,2}( \D, \R^d ) $,
	\item
	\label{item:conv_partial_L2} it holds for all
	$ u \in H_\rho $,
	$ j \in \{1, \ldots, d \} $
	that
		\begin{equation} 
		\| \partial_j u \|_U 
		=
		\Big(
		\sum\nolimits_{ h \in \H }
		\| \langle h, u \rangle_H \partial_j h \|_U^2
		\Big)^{\!\nicefrac{1}{2}}
		\leq
		\Big(
		\sup\nolimits_{ h \in \H } 
		\tfrac{ \| \partial_j h \|_{U} } 
		{ |  \lambda_h |^{ \rho}  }
		\Big)
		\| u \|_{H_\rho}
		< \infty 
		\end{equation}
		and
		$ \partial_j u = \sum_{h \in \H}
		\langle h, u \rangle_H
		\partial_j h $, and
	\item  it holds for all
	$ u \in H_\rho $ that
	\begin{equation}
	\begin{split}
	\label{eq:estimate_partialL2} 
	& 
	\| \partial u \|_{L^2(\mu_\D; \R^{d\times d} ) } 
	\leq
	\bigg[
		\smallsum\limits_{j=1}^d
		\| \partial_j u \|_{L^2(\mu_\D; \R^d) }^2
	\bigg]^{\nicefrac{1}{2}}
	\!
	\leq
	\sqrt{d}
	\bigg[
	\sup\limits_{h \in \H}
	\sup\limits_{ j \in  \{1,\ldots, d\} }
	\tfrac{
		\| \partial_j h \|_U
	}{
	|  \lambda_h |^{ \rho}
}
\bigg]
\| u \|_{H_\rho}
< \infty.
\end{split}
\end{equation}
\end{enumerate}
\end{lemma}
\begin{proof}[Proof of Lemma~\ref{lemma:weak_exists}]
Note
that
for all
$ u \in H_\rho $,
$ j \in \{1,\ldots, d \} $
it holds that
\begin{equation}
\begin{split}
\label{eq:inDomain}
\sum\nolimits_{h \in \H }   
\| \langle h, u \rangle_H \partial_j h \|_{U}^2
&
\leq
\Big(
\sup\nolimits_{ h \in \H }
\tfrac{ \| \partial_j h \|_{U}^2 } 
{ |  \lambda_h |^{2\rho}  }
\Big)
\sum\nolimits_{ h \in \H }  
|  \lambda_h |^{2\rho}
| \langle h, u \rangle_H |^2 
\\
&
=
\Big(
\sup\nolimits_{ h \in \H }
\tfrac{ \| \partial_j h \|_{U}^2 } 
{ |  \lambda_h |^{2\rho}  }
\Big)
\| u \|_{H_\rho}^2
<
\infty.
\end{split}
\end{equation}
The fact that
for all
$ j \in \{1, \ldots, d \} $,
	$ v, w \in \H $ 
 with
	$ v \neq w $
	it holds that 
	$ \langle \partial_j v,  
	\partial_j w \rangle_U = 0 $
hence shows that 
 for all
 $ u \in H_\rho $,
$ \phi \in \mathcal{C}_{cpt}^\infty(\D, \R^d) $,
$ j \in \{1,\ldots, d \} $
it holds that
\begin{equation}
\begin{split}
&
\langle u, [ \tfrac{\partial}{\partial x_j} \phi ]_{\mu_\D, \mathcal{B}(\R^d)} \rangle_U
=
\left< \sum_{ h \in \H }
\langle h, u \rangle_U
h,
[ \tfrac{\partial}{\partial x_j} \phi ]_{\mu_\D, \mathcal{B}(\R^d)} 
\right>_{\!\! U}
=
\sum_{ h \in \H }
\langle h, u \rangle_U
\langle h, 
[ \tfrac{\partial}{\partial x_j} \phi ]_{\mu_\D, \mathcal{B}(\R^d)}
 \rangle_U
\\
&
=
-
\sum_{ h \in \H }
\langle h, u \rangle_U
\langle \partial_j h, 
[ \phi ]_{\mu_\D, \mathcal{B}(\R^d)}
 \rangle_U
 =
 -
 \left<
 \sum_{ h \in \H }  
 \langle h, u \rangle_U 
  \partial_j h, 
 [ \phi ]_{\mu_\D, \mathcal{B}(\R^d)}
 \right>_{\!\! U}
.
\end{split}
\end{equation}
This 
and~\eqref{eq:inDomain}
complete the proof of Lemma~\ref{lemma:weak_exists}.
\end{proof}
\begin{lemma}[Weak product rule (cf., e.g., Proposition 7.1.11 in Atkinson \& Han~\cite{AtkinsonWeimin2009})]
\label{lemma:chain_rule}
Let $ d \in \N $,
$ u, v \in  
[ W^{1,2}( (0,1)^d, \R)
\cap L^\infty( \mu_{(0,1)^d}; \R) ] $,  
$ j \in \{1,\ldots, d \} $.
Then it holds that  
$ u \cdot v 
\in  
[ W^{1,2}( (0,1)^d, \R)
\cap L^\infty( \mu_{(0,1)^d}; \R) ] $ 
and 
$ \partial_j(u v) 
= 
u \, \partial_j v
+
v \, \partial_j u $.
\end{lemma}
\begin{proof}[Proof of Lemma~\ref{lemma:chain_rule}]
Throughout this proof let
$ \tilde u_n,
\tilde v_n
\in
\mathcal{C}^\infty([0,1]^d, \R) $,
$ n \in \N $,
and
$ u_n, v_n \in  
W^{1,2}((0,1)^d, \R) 
$,
$ n \in \N $, 
satisfy
for all
$ n \in \N $
that
 $ u_n = 
 [ \tilde u_n |_{(0,1)^d}]_{\mu_{(0,1)^d}, \mathcal{B}(\R)} $,
 $ v_n = 
 [ \tilde v_n |_{(0,1)^d}]_{\mu_{(0,1)^d}, \mathcal{B}(\R)} $,
 and
$ \limsup_{m \to \infty} 
\big( 
\| u  - 
  \tilde u_m  \|_{W^{1,2}(  (0,1)^d , \R)} 
+
\| v  - 
 \tilde v_m  \|_{W^{1,2}( (0,1)^d, \R)} 
\big) = 0 $
(see, e.g., Theorem 7.3.2 in Atkinson \& Han~\cite{AtkinsonWeimin2009}).
Observe that for all
$ f, g \in L^2(\mu_{(0,1)^d}; \R ) $,
$ h \in L^\infty(\mu_{(0,1)^d}; \R ) $ it holds that
\begin{equation}
\begin{split}
\| (f - g) h \|_{L^2(\mu_{(0,1)^d}; \R ) }
\leq
\| f - g \|_{L^2(\mu_{(0,1)^d}; \R ) }
\| h \|_{L^\infty(\mu_{(0,1)^d}; \R ) }
< \infty .
\end{split}
\end{equation} 
Hence, we obtain for all
$ \phi \in \mathcal{C}_{cpt}((0,1)^d, \R) $,
$ i \in \{1,\ldots, d \} $
that
\begin{equation}
\begin{split}
&
-
\Big\langle u v, 
\big[ \tfrac{\partial}{\partial x_i} \phi
\big ]_{\mu_{(0,1)^d}, \mathcal{B}(\R)} 
\Big\rangle_{\! L^2(\mu_{(0,1)^d}; \R ) }
=
-
\lim_{n \to \infty}
\Big \langle 
u_n v,  
\big[ 
\tfrac{\partial}{\partial x_i} \phi 
\big ]_{\mu_{(0,1)^d}, \mathcal{B}(\R)}
\Big \rangle_{\! L^2(\mu_{(0,1)^d}; \R ) }
\\
&
=
-
\lim_{n \to \infty}
\bigg[
\lim_{m \to \infty}
\Big \langle 
u_n  v_m,
\big[ 
\tfrac{\partial}{\partial x_i} \phi 
\big ]_{\mu_{(0,1)^d}, \mathcal{B}(\R)}
\Big \rangle_{\! L^2(\mu_{(0,1)^d}; \R ) }
\bigg]
\\
&
=
-
\lim_{n \to \infty}
\bigg[
\lim_{m \to \infty}
\int_\D
\tilde u_n (x) 
\,
\tilde v_m (x)
\,
\big(
\tfrac{\partial}{\partial x_i} \phi
\big)(x)
\, dx
\bigg]
\\
&
= 
\lim_{n \to \infty}
\bigg[
\lim_{m \to \infty}
\int_\D
\Big[
\big( \tfrac{\partial}{\partial x_i}
\tilde u_n \big)(x)
\,
\tilde v_m (x) 
+
\tilde u_n(x)
\,
\big( 
\tfrac{\partial}{\partial x_i}
\tilde v_m
\big)(x)
\Big]
\phi(x)
\, dx
\bigg]
\\
&
=
\lim_{n \to \infty}
\bigg[
\lim_{m \to \infty}
\langle 
v_m \, \partial_i u_n
+
u_n \, \partial_i v_m,  
\phi \rangle_{ L^2(\mu_{(0,1)^d}; \R ) }
\bigg]
\\
&
=
\lim_{n \to \infty} 
\langle 
v \, \partial_i u_n
+
u_n \, \partial_i v,  
\phi 
\rangle_{L^2(\mu_{(0,1)^d}; \R ) } 
= 
\langle 
v \, \partial_i u
+
u \, \partial_i v,  
\phi 
\rangle_{L^2(\mu_{(0,1)^d}; \R ) } 
.
\end{split}
\end{equation}
This completes the proof of Lemma~\ref{lemma:chain_rule}. 
\end{proof}
\begin{lemma}[Weak integration by parts]
	\label{lemma:int_parts}
	Let $ d \in \N $,
	$ u, v \in W_P^{1,2}( (0,1)^d, \R) $,  
	$ j \in \{1,\ldots, d \} $.
	Then it holds that 
	\begin{equation}
	\left< \partial_j u, v \right>_{\! L^2(\mu_{(0,1)^d}; \R)} 
	=
	\!
	- 
	\!
	\left< u, \partial_j v \right>_{\! L^2(\mu_{(0,1)^d}; \R)} 
	.
	\end{equation}
\end{lemma}
\begin{proof}[Proof of Lemma~\ref{lemma:int_parts}]
	Throughout this proof let
	$ \tilde u_n,
	\tilde v_n
	\in
	\mathcal{C}_P^\infty([0,1]^d, \R) $,
	$ n \in \N $,
	and
	$ u_n, v_n \in  
	W_P^{1,2}((0,1)^d, \R) 
	$,
	$ n \in \N $,
	satisfy for all $ n \in \N $
	that
	$ u_n = 
	[ \tilde u_n |_{(0,1)^d}]_{\mu_{(0,1)^d}, \mathcal{B}(\R)} $,
	$ v_n = 
	[ \tilde v_n |_{(0,1)^d}]_{\mu_{(0,1)^d}, \mathcal{B}(\R)} $,
	and
	$ \limsup_{m \to \infty} 
	\big( 
	\| u  - 
	[ \tilde u_m |_{ (0,1)^d } ]_{\mu_{ (0,1)^d }, \mathcal{B}(\R)} \|_{W^{1,2}(  (0,1)^d , \R)} 
	+
	\| v  - 
	[ \tilde v_m |_{ (0,1)^d } ]_{\mu_{ (0,1)^d }, \mathcal{B}(\R)} \|_{W^{1,2}( (0,1)^d, \R)} 
	\big) = 0 $.
	Observe that
	integration by parts
	and the fact that
	$ \forall \, n \in \N \colon
	\tilde u_n, \tilde v_n \in \mathcal{C}_P^\infty([0,1]^d, \R) $
	prove that
	\begin{equation}
	\begin{split}
	& 
	\langle \partial_j u, v \rangle_{L^2(\mu_{(0,1)^d}; \R ) }
	=
	\lim_{n \to \infty}
	\langle \partial_j u_n, v \rangle_{L^2(\mu_{(0,1)^d}; \R ) }
	= 
	\lim_{n \to \infty}
	\left(
	\lim_{m \to \infty}
	\langle \partial_j u_n, v_m \rangle_{L^2(\mu_{(0,1)^d}; \R ) }
	\right)
	\\
	&
	= 
	\lim_{n \to \infty}
	\left(
	\lim_{m \to \infty}
	\int_\D
	\tilde v_m(x)
	\,
	\big(
	\tfrac{\partial}{\partial x_j} 
	\tilde u_n
	\big)(x) 
	\, dx
	\right)
	=
	-
	\lim_{n \to \infty}
	\left(
	\lim_{m \to \infty}
	\int_\D
	\tilde u_n(x)
	\,
	\big(
	\tfrac{\partial}{\partial x_j} 
	\tilde v_m
	\big)(x) 
	\, dx
	\right)
	\\
	&
	=
	-
	\lim_{n \to \infty}
	\left(
	\lim_{m \to \infty}
	\langle  u_n, \partial_j v_m \rangle_{L^2(\mu_{(0,1)^d}; \R ) }
	\right)
	=
	-
	\lim_{n \to \infty} 
	\langle  u_n, \partial_j v \rangle_{L^2(\mu_{(0,1)^d}; \R ) } 
	\\
	&
	=
	-
	\langle  u, \partial_j v \rangle_{L^2(\mu_{(0,1)^d}; \R ) } 
	.
	\end{split}
	\end{equation}
	The proof of Lemma~\ref{lemma:int_parts}
	is thus completed.
\end{proof}
\begin{lemma}[Weak integration by parts revisited] 
	\label{lemma:int_parts2}
	Let $ d \in \N $,
	$ u, v, w \in [ W_P^{1,2}( (0,1)^d, \R)
	\cap L^\infty( \mu_{(0,1)^d}; \R) ] $,
	$ j \in \{1,\ldots, d \} $.
	Then it holds that 
	$ u \cdot v \in [ W^{1,2}( (0,1)^d, \R ) \cap L^\infty( \mu_{(0,1)^d}; \R) ] $
	and
	\begin{equation}
	\left< \partial_j( u v ), w \right>_{\! L^2(\mu_{(0,1)^d}; \R)} 
	=
	\!
	- 
	\!
	\left< u v, \partial_j w \right>_{\! L^2(\mu_{(0,1)^d}; \R)} 
	.
	\end{equation}
\end{lemma}
\begin{proof}[Proof of Lemma~\ref{lemma:int_parts2}]
	Throughout this proof let
	$ \tilde u_n,
	\tilde v_n,
	\tilde w_n
	\in
	\mathcal{C}_P^\infty([0,1]^d, \R) $,
	$ n \in \N $,
	and 
	$ u_n, v_n, w_n \in  
	W_P^{1,2}((0,1)^d, \R) 
	$,
	$ n \in \N $,
	satisfy 
	for all
	$ n \in \N $ 
	that
	$ u_n = 
	[ \tilde u_n |_{(0,1)^d}]_{\mu_{(0,1)^d}, \mathcal{B}(\R)} $,
	$ v_n = 
	[ \tilde v_n |_{(0,1)^d}]_{\mu_{(0,1)^d}, \mathcal{B}(\R)} $,
	$ w_n = 
	[ \tilde w_n |_{(0,1)^d}]_{\mu_{(0,1)^d}, \mathcal{B}(\R)} $,
	and
	$ \limsup_{m \to \infty} 
	\big( 
	\| u  - 
	u_m  \|_{W^{1,2}(  (0,1)^d , \R)} 
	+
	\| v  -   v_m   \|_{W^{1,2}( (0,1)^d, \R)}
	+
	\| w  -   w_m   \|_{W^{1,2}( (0,1)^d, \R)}  
	\big) = 0 $.
	Observe that
	Lemma~\ref{lemma:chain_rule}
	(with $ d = d $,
	$ u = u $,
	$ v = v $,
	$ j = j $
	in the notation of Lemma~\ref{lemma:chain_rule})
and the product rule for differentiation	 
		prove that  
		$ u \cdot v \in [ W^{1,2}( (0,1)^d, \R ) 
		\cap L^\infty( \mu_{ (0,1)^d }; \R ) ] $
		and
	\begin{equation}
	\begin{split}
	& 
	\langle \partial_j ( u v ), w \rangle_{L^2(\mu_{(0,1)^d}; \R ) }
	=
	\langle 
	u
	\,
	\partial_j v
	+
	v
	\,
	\partial_j u,
	w
	\rangle_{L^2(\mu_{(0,1)^d}; \R ) }
	\\
	&
	=
	\lim_{l \to \infty}
	\Big(
	\langle 
	u_l
	\,
	\partial_j v,
	w
	\rangle_{L^2(\mu_{(0,1)^d}; \R ) }
	+ 
	\langle
	v
	\,
	\partial_j u_l,
	w
	\rangle_{L^2(\mu_{(0,1)^d}; \R ) }
	\Big)
	\\
	&
	=
	\lim_{l \to \infty}
	\Big( 
	\lim_{n \to \infty}
	\Big(
	\langle 
	u_l
	\,
	\partial_j v,
	w_n
	\rangle_{L^2(\mu_{(0,1)^d}; \R ) }
	+
	\langle
	v
	\,
	\partial_j u_l,
	w_n
	\rangle_{L^2(\mu_{(0,1)^d}; \R ) }
	\Big)
	\Big)
	\\
	&
	=
	\lim_{l \to \infty}
	\Big( 
	\lim_{n \to \infty}
	\Big(
	\lim_{m \to \infty}
	\Big(
	\langle 
	u_l
	\,
	\partial_j v_m,
	w_n
	\rangle_{L^2(\mu_{(0,1)^d}; \R ) }
	+
	\langle
	v_m 
	\,
	\partial_j u_l,
	w_n
	\rangle_{L^2(\mu_{(0,1)^d}; \R ) }
	\Big)
	\Big)
	\Big)
	\\
	&
	=
	\lim_{l \to \infty}
	\Big( 
	\lim_{n \to \infty}
	\Big(
	\lim_{m \to \infty}
	\langle  
	u_l
	\,
	\partial_j v_m
	+
	v_m
	\,
	\partial_j u_l,
	w_n
	\rangle_{L^2(\mu_{(0,1)^d}; \R ) }
	\Big)
	\Big)
	\\
	&
	=
	\lim_{l \to \infty}
	\Big( 
	\lim_{n \to \infty}
	\Big(
	\lim_{m \to \infty}
	\langle  
	\partial_j ( u_l v_m ),
	w_n
	\rangle_{L^2(\mu_{(0,1)^d}; \R ) }
	\Big)
	\Big)
	.
	\end{split}
	\end{equation}
	Integration by parts 
	and the fact that
		$ \forall \, n \in \N \colon
		\tilde u_n, \tilde v_n, \tilde w_n \in \mathcal{C}_P^\infty([0,1]^d, \R) $
		hence show that
	\begin{equation}
	\begin{split}
	& 
	\langle \partial_j ( u v ), w \rangle_{L^2(\mu_{(0,1)^d}; \R ) }
	=
	\lim_{l \to \infty}
	\bigg( 
	\lim_{n \to \infty}
	\bigg(
	\lim_{m \to \infty}
	\int_{ (0,1)^d }
	\big[ 
	\tfrac{\partial}{\partial x_j} 
	(
	\tilde u_l ( x )
	\cdot
	\tilde v_m ( x )
	)
	\big]
	\,
	\tilde w_n(x)
	\,
	dx
	\bigg)
	\bigg)
	\\
	&
	=
	-
	\lim_{l \to \infty}
	\bigg( 
	\lim_{n \to \infty}
	\bigg(
	\lim_{m \to \infty}
	\int_{ (0,1)^d } 
	\tilde u_l(x)
	\,
	\tilde v_m(x) 
	\,
	\big(
	\tfrac{\partial}{\partial x_j} 
	\tilde w_n
	\big)(x)
	\,
	dx
	\bigg)
	\bigg)
	\\
	&
	=
	-
	\lim_{l \to \infty}
	\Big( 
	\lim_{n \to \infty}
	\Big(
	\lim_{m \to \infty}
	\langle  
	u_l v_m,
	\partial_j 
	w_n
	\rangle_{L^2(\mu_{(0,1)^d}; \R ) }
	\Big)
	\Big)
	\\
	&
	=
	-
	\lim_{l \to \infty}
	\Big( 
	\lim_{n \to \infty} 
	\langle  
	u_l v,
	\partial_j 
	w_n
	\rangle_{L^2(\mu_{(0,1)^d}; \R ) }
	\Big) 
	\\
	&
	=
	-
	\lim_{l \to \infty}  
	\langle  
	u_l v,
	\partial_j 
	w 
	\rangle_{L^2(\mu_{(0,1)^d}; \R ) } 
	=
	-   
	\langle  
	u v,
	\partial_j 
	w 
	\rangle_{L^2(\mu_{(0,1)^d}; \R ) } 
	.
	\end{split}
	\end{equation}
	The proof of Lemma~\ref{lemma:int_parts2}
	is thus completed.
\end{proof}
\begin{lemma}
	\label{lemma:zero_coercivity}
	Assume the setting in Subsection~\ref{setting:Examples},
	let  
$ \rho \in [\gamma, \infty) $,  	
	$ u \in  H_\rho $,
	and  
	assume for all
	$ j \in \{1,\ldots, d\} $,
	$ v, w \in \H $
	that
	 $ \H
		\subseteq
		W^{1,2}(\D, \R^d)  $, 
		$  
		\big(
		\sum_{ h \in \H }
		| \lambda_h |^{-2\rho }
		\big)
		+
				\sup_{ h \in \H }
		\big(
		\| h
		\|_{L^\infty( \mu_\D; \R^d)}
		+
		\| \partial_j h 
		\|_U
		| \lambda_h |^{-\rho} 
		\big)   
		< \infty $,  
		$ \langle \partial_j v,  
		\partial_j w  \rangle_U \, \1_{\H \backslash \{ v \} } (w)= 0 $.
		Then
		 it holds  		
		 that
		$ u \in [ W^{1,2}(\D, \R^d) \cap L^\infty (\mu_\D; \R^d) ] $
		and
		\begin{equation}
\begin{split}
\label{eq:estimate_W12} 
&
\| \partial u \|_{L^2(\mu_\D; \R^{d\times d} ) } 
	\leq
	\bigg[
		\smallsum\limits_{j=1}^d
		\| \partial_j u \|_{L^2(\mu_\D; \R^d) }^2
	\bigg]^{\nicefrac{1}{2}}
	\!
	\leq
	\sqrt{d}
	\bigg[
	\sup\limits_{h \in \H}
	\sup\limits_{ j \in  \{1,\ldots, d\} }
	\tfrac{
		\| \partial_j h \|_U
	}{
	|  \lambda_h |^{ \rho}
}
\bigg]
\| u \|_{H_\rho} 
< \infty,
\end{split}
\end{equation}
			\begin{equation}
			\begin{split}
			\label{eq:estimate_infiniti}
			\| u \|_{L^\infty( \mu_\D; \R^d)} 
			&
			\leq
			\| u \|_{H_\rho}
			\bigg(
			\sup_{ h \in \H } 
			\| h
			\|_{L^\infty( \mu_\D; \R^d)}
			\bigg) 
		\Bigg[ 
		\sum_{h \in \H }
		| \lambda_h |^{- 2 \rho}
		\Bigg]^{\nicefrac{1}{2}}
		< \infty,
		\end{split}
		\end{equation} 
		\begin{equation}
	   	\begin{split}
	   	\label{eq:estimateF}
	   	&
	   	\| F( u ) \|_H  
	   	\leq 
	   	\eta \| u \|_H
	   	\\
	   	&
	   	+
	   	d \sqrt{d}
	   	\| u \|_{H_\rho}^2
	   	\bigg[
	   	\sup_{h\in \H}
	   	\sup_{j \in 
	   	 \{1,\ldots, d\}}
	   	\tfrac{
	   		\| \partial_j h \|_U 
	   	}{
	   	| \lambda_h |^{ \rho}
	   }
	   \bigg]
	   	\bigg(  
	   	\sup_{h \in \H } 
	   	\| h
	   	\|_{L^\infty( \mu_\D; \R^d)}
	   	\bigg) 
	   	\Bigg[ 
	   	\sum_{ h \in \H }
	   	| \lambda_h |^{- 2 \rho}
	   	\Bigg]^{\nicefrac{1}{2}}
	    < \infty
	   .
	   	\end{split}
	   	\end{equation}
\end{lemma}
\begin{proof}[Proof of Lemma~\ref{lemma:zero_coercivity}]
First, note that Lemma~\ref{lemma:Linfty}
(with 
$ \rho = \rho $,
$ v = u $
in the notation of
Lemma~\ref{lemma:Linfty})
proves~\eqref{eq:estimate_infiniti}.
Moreover, observe that
Lemma~\ref{lemma:weak_exists} 
(with
$ \rho = \rho $
in the notation of
Lemma~\ref{lemma:weak_exists}) 
establishes  
that
$
u \in W^{1,2}(\D, \R^d) 
$
and~\eqref{eq:estimate_W12}.
This
and~\eqref{eq:estimate_infiniti}
ensure that 
$ u \in [ W^{1,2}(\D, \R^d) \cap
L^\infty ( \mu_D; \R^d ) ] $.
Combining
Lemma~\ref{lemma:F_well_defined}
(with $ v = u $,
$ w = u $
in the notation of Lemma~\ref{lemma:F_well_defined}),
\eqref{eq:estimate_W12},
and~\eqref{eq:estimate_infiniti}
hence proves~\eqref{eq:estimateF}.
The proof of Lemma~\ref{lemma:zero_coercivity}
is thus completed.
\end{proof}
\begin{lemma}
	\label{lemma:zero_coercivity2}
	Assume the setting in Subsection~\ref{setting:Examples},
	let  
$ \rho \in [\gamma, \infty) $,  	
	$ u = ( u_1, \ldots, u_d ) \in  H_\rho $,
	and  
	assume for all
	$ j \in \{1,\ldots, d\} $,
	$ v, w \in \H $
	that
	 $ \H
		\subseteq
		W_P^{1,2}(\D, \R^d)  $, 
		$  
		\big(
		\sum_{ h \in \H }
		| \lambda_h |^{-2\rho }
		\big)
		+
				\sup_{ h \in \H }
		\big(
		\| h
		\|_{L^\infty( \mu_\D; \R^d)}
		+
		\| \partial_j h 
		\|_U
		| \lambda_h |^{-\rho} 
		\big)   
		< \infty $,  
		$ \langle \partial_j v,  
		\partial_j w  \rangle_U \, \1_{\H \backslash \{ v \} } (w)= 0 $.
		Then it holds that
		 $ u \in [ W_P^{1,2}(\D, \R^d) \cap L^\infty (\mu_\D; \R^d) ] $
		 and
		\begin{equation}
	    2 \langle u, F( u ) \rangle_H
	    = 
	    2
	    \eta \| u \|_H^2
	    +  
	    \smallsum_{j=1}^d \langle  u , u \, \partial_j u_j  \rangle_U
	    =
	    2
	    \eta \| u \|_H^2
	    + 
	    \big\langle  \smallsum_{i=1}^d ( u_i )^2 , 
	    \smallsum_{j=1}^d \partial_j u_j  
	    \big\rangle_{L^2(\mu_\D; \R)} 
	    .
		\end{equation} 
\end{lemma}
\begin{proof}[Proof of Lemma~\ref{lemma:zero_coercivity2}]
First, note that Lemma~\ref{lemma:zero_coercivity}
(with $ \rho = \rho $,
$ u = u $ 
in the notation of
Lemma~\ref{lemma:zero_coercivity})
ensures that
\begin{equation}
\label{eq:correct_subset}
 u \in [ W^{1,2}(\D, \R^d) \cap L^\infty (\mu_\D; \R^d) ] 
 .
 \end{equation}
Moreover,
observe that
\begin{equation}
\label{eq:L2_conv}
\limsup\nolimits_{\mathcal{P}_0( \H) \ni I \to \H}
 \| u - \smallsum_{h \in I} 
\langle h, u \rangle_H h   \|_{ L^2( \mu_\D; \R^d ) } = 0 
.
\end{equation}
In addition, note that
item~\eqref{item:conv_partial_L2} in Lemma~\ref{lemma:weak_exists}
(with
$ \rho = \rho $,
$ u = u $,
$ j = j $
for  
$ j \in \{ 1, \ldots, d \} $
in the notation of
Lemma~\ref{lemma:weak_exists})
proves that for all
$ j \in \{1,\ldots, d \} $
it holds
that
\begin{equation}
\begin{split}
\label{eq:W12_conv}
&
\limsup\nolimits_{\mathcal{P}_0( \H) \ni I \to \H}
  \| 
\partial_j u 
- 
\smallsum_{h \in I} \langle h, u \rangle_H \partial_j h 
  \|_{ L^2( \mu_\D; \R^d ) }
=
0
.
\end{split}
\end{equation}
Combining~\eqref{eq:correct_subset}--\eqref{eq:W12_conv} with the fact that
$ \forall \, v \in W^{1,2}(\D, \R^d) 
\colon  \| v \|_{ W^{1,2}(\D, \R^d) }^2
=
\| v \|_{ L^2( \mu_\D; \R^d ) }^2
+
\sum_{j=1}^d
\| \partial_j v 
\|_{ L^2( \mu_\D; \R^d ) }^2 $
proves that
\begin{equation}
\limsup\nolimits_{\mathcal{P}_0( \H) \ni I \to \H}
  \| 
 u 
- 
\smallsum_{h \in I} \langle h, u \rangle_H   h 
  \|_{ W^{1, 2}( \D; \R^d ) }
=
0
.
\end{equation}
The fact that
$ W^{1,2}_P( (0,1)^d, \R^d) $
is a closed subspace of 
$ W^{1,2}( (0,1)^d, \R^d) $, 
\eqref{eq:correct_subset},
and the fact that
 $ \forall \, I \in \mathcal{P}_0(\H) 
 \colon 
 \sum_{h \in I} 
 \langle h, u \rangle_H h 
 \in 
 W_P^{1,2}( (0,1)^d, \R^d ) $
hence show that 
\begin{equation} 
\label{eq:correct_sub}
u \in [ W_P^{1,2}(\D, \R^d) \cap L^\infty (\mu_\D; \R^d) ] 
.
\end{equation}
This 
and
Lemma~\ref{lemma:chain_rule}
(with
$ d = d $,
$ u = u_i$,
$ v = u_j $,
$ j = j $
for $ i,j \in \{1,\ldots, d\} $
in the notation of Lemma~\ref{lemma:chain_rule})
prove that
for all 
$ i,j \in \{1,\ldots, d \} $
it holds that
$ u_i u_j \in W^{1,2}(\D, \R) $
and
$ \partial_j(u_i u_j)
=
u_i \, \partial_j u_j
+
u_j \, \partial_j u_i $.
Combining this
and the fact that
$ \forall \, i \in \{1,\ldots, d \} \colon 
u_i \in [ W_P^{1,2}(\D, \R) \cap L^{\infty} (\mu_\D; \R) \cap H ] $
with Lemma~\ref{lemma:int_parts2}
(with
$ d = d $,
$ u = u_i $,
$ v = u_j $,
$ w = u_i $,
$ j = j $
for 
$ i, j \in \{1,\ldots, d\} $
in the notation of
Lemma~\ref{lemma:int_parts2})
ensures that
\begin{equation}
\begin{split} 
&
\langle u, F(u) \rangle_H 
=
\langle
u,
R( \eta u - \smallsum_{j=1}^d u_j \, \partial_j u ) 
\rangle_H
=
\eta
\| u \|_H^2
-
\smallsum_{j=1}^d
\langle R u,
u_j \, \partial_j u 
\rangle_U
\\
&
=
\eta
\| u \|_H^2
-
\smallsum_{j=1}^d
\langle u,
u_j \, \partial_j u 
\rangle_U
=
\eta \| u \|_H^2
-
\smallsum_{j=1}^d
\smallsum_{i=1}^d
\langle u_i, u_j \, \partial_j u_i \rangle_{L^2(\mu_\D; \R)}
\\
&
=
\eta \| u \|_H^2
-
\smallsum_{j=1}^d
\smallsum_{i=1}^d
\langle u_i u_j, \partial_j u_i \rangle_{L^2(\mu_\D; \R)}
=
\eta \| u \|_H^2
+
\sum_{j=1}^d
\sum_{i=1}^d
\langle  
  \partial_j ( u_i u_j ),
  u_i
\rangle_{L^2(\mu_\D; \R)}
\\
&
=
\eta \| u \|_H^2
+
\smallsum_{j=1}^d
\smallsum_{i=1}^d
\langle  
u_i
\,
\partial_j u_j 
+
u_j
\,
\partial_j u_i,
u_i
\rangle_{L^2(\mu_\D; \R)}
.
\end{split}
\end{equation}
Hence, we obtain that
\begin{equation}
\begin{split}
\label{eq:first_part_coercivity}
&
\langle u, F(u) \rangle_U
=
\eta \| u \|_H^2
+
\smallsum_{j=1}^d
\smallsum_{i=1}^d
\langle    
u_i,
u_i
\,
\partial_j u_j 
+
u_j
\,
\partial_j u_i
\rangle_{L^2(\mu_\D; \R)}
\\
&
=
\eta \| u \|_H^2
+
\smallsum_{j=1}^d
\langle
u,
u \, \partial_j u_j
\rangle_U
+
\smallsum_{j=1}^d
\langle
u,
u_j \, \partial_j u 
\rangle_U
\\
&
=
2 \eta \| u \|_H^2
-
\big[
\eta \| u \|_H^2 
-
\langle
R u,
\smallsum_{j=1}^d
u_j \, \partial_j u
\rangle_U
\big]
+
\smallsum_{j=1}^d
\langle u, u \, \partial_j u_j \rangle_U
\\
&
=
2 \eta \| u \|_H^2
-
\big[
\langle u, R ( \eta u ) \rangle_H
-
\langle u, R ( \smallsum_{i=1}^d u_i \partial_i u ) \rangle_H
\big]
+
\smallsum_{j=1}^d
\langle u, u \, \partial_j u_j \rangle_U
\\
&
=
2
\eta \| u \|_H^2
-
\langle u, F( u ) \rangle_H
+
\smallsum_{j=1}^d 
\langle  
u, 
u \, \partial_j u_j
\rangle_U 
.
\end{split}
\end{equation}
In addition, 
note that
\begin{equation}
\begin{split}
\label{eq:coercivity_equality2}
&
\smallsum_{j=1}^d \langle  u , u \, \partial_j u_j \rangle_U  
=
\smallsum_{j=1}^d
\smallsum_{i=1}^d 
\langle  u_i, 
u_i \, \partial_j u_j 
\rangle_{L^2(\mu_\D; \R)}
=
\smallsum_{j=1}^d
\smallsum_{i=1}^d 
\langle  
(u_i)^2,
\partial_j u_j
\rangle_{L^2(\mu_\D; \R)}
\\
&
=
\big \langle  \smallsum_{i=1}^d ( u_i )^2 , 
\smallsum_{j=1}^d \partial_j u_j   
\big \rangle_{L^2(\mu_\D; \R)} 
.
\end{split}
\end{equation}
Combining this, \eqref{eq:correct_sub},
and~\eqref{eq:first_part_coercivity}
completes the proof of Lemma~\ref{lemma:zero_coercivity2}.
\end{proof}
\begin{corollary}
	\label{corollary_continuous}
	Assume the setting in Subsection~\ref{setting:Examples}
	and 
	assume for all $ j \in \{ 1, \ldots, d \} $,
	$ v, w \in \H $
	that
	$ \H
	\subseteq
	  W^{1,2}(\D, \R^d) $,
	$  
	\big(
	\sum_{ h \in \H }
	| \lambda_h |^{-2\gamma }
	\big)
	+
	\sup_{ h \in \H }
	\big(
	\| h
	\|_{L^\infty( \mu_\D; \R^d)}
	+
	\| \partial_j h 
	\|_U
	| \lambda_h |^{-\gamma} 
	\big)   
	< \infty $, 
	$ \langle \partial_j v,  
	\partial_j w \rangle_U \, \1_{ \H \backslash \{ v \} } (w) = 0 $.
	Then
	it holds that
	$ F \in \mathcal{C}( H_\gamma, H ) $
	and
	$ B \in \mathcal{C}( H_\gamma, \HS(U,H) ) $.
\end{corollary}
\begin{proof}[Proof of Corollary~\ref{corollary_continuous}]
First of all, note that
Lemma~\ref{lemma:zero_coercivity} 
(with
$ \rho = \gamma $
in the notation of 
Lemma~\ref{lemma:zero_coercivity})
assures that 
\begin{equation}
\label{eq:cont_inclusion}
H_\gamma \subseteq   W^{1,2}(D, \R^d)  
\quad
\text{continuously}
\qquad
\text{and}
\qquad
 H_\gamma \subseteq 
L^\infty( \mu_{\mathcal{D}}; \R^d) 
\quad
\text{continuously.}
\end{equation}
This and 
Lemma~\ref{lemma:F_well_defined} 
(with 
$ v = v $,
$ w = w $
for $ v, w \in H_\gamma $ 
in the notation of 
Lemma~\ref{lemma:F_well_defined})
show that for all 
$ v, w \in H_\gamma $
it holds
that
\begin{equation}
\begin{split}
\label{eq:diff_estimate}
&
\| F(v) - F(w) \|_{L^2(\mu_\D; \R^d)} 
\leq
\eta
\| v - w \|_{ L^2(\mu_\D; \R^d) } 
\\
&
+
d
\big(
\| \partial v \|_{L^2(\mu_\D; \R^{d\times d} )} 
\| v - w \|_{L^\infty(\mu_\D; \R^d)}
+ 
\| w \|_{L^\infty(\mu_\D; \R^d)}
\| \partial( v -w ) \|_{L^2(\mu_\D; \R^{d\times d} )} 
\big) 
< \infty
.
\end{split}
\end{equation}
In addition,
Lemma~\ref{lemma:estimate_B_1} 
(with 
$ v = v $,
$ w = w $
for $ v, w \in H_\gamma $
in the notation of 
Lemma~\ref{lemma:estimate_B_1})
proves that for all
$ v, w \in H_\gamma $
it holds that
\begin{equation}
\begin{split} 
&
\| 
B( v )
-
B( w ) 
\|_{\HS( U, H )}
\\
&
\leq 
\Big(
\sup\nolimits_{x \in \D, y\in \R^d, z \in \R^d \backslash \{ y \} }
\tfrac{ \| b(x,y) - b(x,z) \|_{ \R^{ d \times d } } }{\| y - z \|_{\R^d}}
\Big)
\| v - w \|_{ L^\infty(\mu_\D; \R^d) }
\sqrt{ \operatorname{trace}_U(Q) } 
.
\end{split}
\end{equation} 
Combining this
and~\eqref{eq:diff_estimate}
with~\eqref{eq:cont_inclusion}
 completes the proof of
Corollary~\ref{corollary_continuous}.
\end{proof}
\subsection[Stochastic Burgers equations]{Stochastic Burgers equations}
\label{sec:Burgers}
\begin{corollary}
\label{corollary:Burgers}
Assume the setting in 
Subsection~\ref{setting:Examples},
let 
$ (e_n)_{n\in \N} \subseteq \H $,
and
assume 
for all $ n \in \N $,
$ v \in H_\gamma $
that 
$ \eta = 0 $,
$ d=1 $,  
$ \gamma \geq \nicefrac{1}{2} $,  
$ e_n 
= 
[ \{ \sqrt{2} \sin(n \pi x) \}_{x\in \D} ]_{\mu_{\D}, \mathcal{B}(\R)} $,  
$ \lambda_{e_n} = - \pi^2 n^2 $, 
$ r(v) \geq  
\max\!\big\{ 
\sqrt{ \vartheta }
+
\varepsilon
\| v \|_H^2
, 
\frac{ 1 }{\sqrt{3}}
\| v \|_{H_\gamma}^2
\big\} $.
Then  
\begin{equation}
  \begin{split}
  \label{eq:BurgersExpMoments}
  &
  \sup_{ \theta \in \varpi_T}
  \sup_{I \in \mathcal{P}_0(\H)}
  \sup_{J \in \mathcal{P}_0(\mathbb{U})}
  \sup_{t\in [0,T]} 
  \E\!\left[
       \exp\!\left(
         \tfrac{   
         \varepsilon      
           \| Y^{ \theta, I, J  }_t \|_H^2
         }
         {
         e^{ 2 \varepsilon \vartheta  t }
         }
       \right)
     \right] 
<\infty 
.
\end{split}
  \end{equation}
\end{corollary}
\begin{proof}[Proof of Corollary~\ref{corollary:Burgers}]
First of all, note that the fact that
$ \{e_n \colon n \in \N \}
\subseteq \H
\subseteq H \subseteq U $
shows that
\begin{equation} 
\label{eq:isONS}
 \{ e_n \colon n \in \N \} = \H  
 \qquad
 \text{  and  }
 \qquad
 H = U .
\end{equation}
In the next step 
we observe  
that  
\begin{equation} 
\label{eq:sum_conv}
\sum\nolimits_{ h \in \H }
| \lambda_h |^{-2 \gamma} 
=
\sum\nolimits_{n\in \N}
| \pi^2 n^2 |^{-2 \gamma}
=
\pi^{-4 \gamma}
\sum\nolimits_{n \in \N}
 n^{-4 \gamma}
 \leq
 \pi^{-2}
 \sum\nolimits_{n \in \N}
 n^{-2}
 < \infty.
\end{equation}
Moreover,
note that for all 
$ n \in \N $
it holds that
\begin{equation}
\begin{split}
\label{eq:eig_decrease}
\| \partial e_n 
\|_U
| \lambda_{e_n} |^{- \gamma} 
&
=
\| [ \{ \pi n \sqrt{2} \cos( n \pi  x )  \}_{x\in \D} ]_{\mu_\D, \mathcal{B}(\R)} \|_U
| \pi^2 n^2 |^{-\gamma}
\\
&
=
\pi n | \pi^2 n^2 |^{-\gamma}
=
\tfrac{1}{ (\pi n)^{2 \gamma - 1 } } 
\leq
1
.
\end{split}
\end{equation}
Combining~\eqref{eq:isONS}--\eqref{eq:eig_decrease},
the fact that
$ \sup_{h\in \H} \| h
\|_{L^\infty( \mu_\D; \R)}
= \sqrt{2} $,
and
Lemma~\ref{lemma:zero_coercivity}
(with   
$ \rho = \gamma $,
$ u = v $
for 
$ v \in H_\gamma $  
in the notation of 
Lemma~\ref{lemma:zero_coercivity})
proves that for all
	$ v \in H_\gamma $ 
	it holds that 
	$ H_\gamma \subseteq [ W^{1,2}(\D, \R) 
	\cap L^\infty( \mu_{\D}; \R) ] $
	and
	\begin{equation}
	\begin{split}
	\label{eq:F_B_estimate}
	\| F(v) \|_H
\leq
\tfrac{\sqrt{2}}{\pi}
\big(
\smallsum_{n\in \N}
n^{-2}
\big)^{\nicefrac{1}{2}}
\| v \|_{H_\gamma}^2
=
\tfrac{1}{\sqrt{3}}
\| v \|_{H_\gamma}^2
< \infty
	.
\end{split}
\end{equation}
Next note that
\begin{equation}
\begin{split}
\H
\subseteq D(A) 
& = [ W^{2,2}( \D, \R ) \cap W_0^{1,2}( \D, \R ) ]
\subseteq W_0^{1,2} ( \D, \R ) 
=
\overline{ \mathcal{C}_{cpt}^\infty ( \D, \R ) }^{ W^{1,2} ( \D, \R) } 
\\
&
\subseteq
\overline{ \mathcal{C}_{P}^\infty ( \D, \R ) }^{ W^{1,2} ( \D, \R) } 
= W^{1,2}_P (\D, \R) 
.
\end{split}
\end{equation}
This,
\eqref{eq:isONS}--\eqref{eq:eig_decrease},
the fact that
$ \sup_{h\in \H} \| h
\|_{L^\infty( \mu_\D; \R)}
= \sqrt{2} $, 
and
Lemma~\ref{lemma:zero_coercivity2}
(with   
$ \rho = \gamma $,
$ u = x $
for  
$ x \in H_\gamma $ 
in the notation of 
Lemma~\ref{lemma:zero_coercivity2})
ensure that for all  
$ x \in H_\gamma $
it holds that
$ H_\gamma \subseteq [ W_P^{1,2}( \D, \R ) \cap L^\infty( \mu_\D; \R ) ] $
and
	\begin{equation}
	\begin{split}
2 \langle x, F(x) \rangle_H
&
=
2 \eta \| x \|_H^2
+
\langle x, x \, \partial x \rangle_U
=
2 \eta \| x \|_H^2
+
\langle x, R ( x \, \partial x ) \rangle_H
\\
&
=
2 \eta \| x \|_H^2
+
\langle x, F(x) \rangle_H
=
\langle x, F(x) \rangle_H
.
\end{split}
 \end{equation} 
Hence, we obtain that
for all
$ I\in \mathcal{P}_0( \H ) $, 
$ x \in P_I(H) $
it holds that
\begin{equation}
\label{eq:Burger_coercivity}
\langle x , P_I F(x) \rangle_H =
\langle P_I x, F(x) \rangle_H
=
\langle x, F(x) \rangle_H
=
0.
\end{equation}
In the next step
we observe that~\eqref{eq:isONS}--\eqref{eq:eig_decrease},
the fact that
$ \sup_{h\in \H} \| h
\|_{L^\infty( \mu_\D; \R)}
= \sqrt{2} $,
and
Corollary~\ref{corollary_continuous} 
assure that 
$
F \in \mathcal{C}( H_\gamma, H )
$
and
$
B \in \mathcal{C}( H_\gamma, \HS(U,H) )
$.
This proves that
\begin{equation} 
\label{eq:meas}
F \in \mathcal{M}( \mathcal{B}(H_\gamma), \mathcal{B}(H) )
\qquad
\text{and}
\qquad
B \in \mathcal{M}( \mathcal{B}(H_\gamma), \mathcal{B}(\HS(U,H)) )
.
\end{equation}
Moreover, \eqref{eq:F_B_estimate} and Lemma~\ref{lemma:estimate_B_1} 
(with 
$ v = x $,
$ w = x $
for $ x \in 
\cup_{ h \in (0,T] }
\cup_{I \in \mathcal{P}_0(\H)} D_h^I $
in the notation of 
Lemma~\ref{lemma:estimate_B_1})
imply for all 
$ h\in (0,T] $, 
$ I \in \mathcal{P}_0(\H) $, 
$ J \in \mathcal{P}_0 ( \mathbb{U} ) $,
$ x \in D_h^I $ that
\begin{equation}
\begin{split}
\label{eq:growt_estimate_Burgers}
\max\!\big\{
\| P_I F(x) \|_H,
\| P_I B(x) \hat P_J \|_{\HS( U, H )}
\big\}
&
\leq 
\max\!\big\{
\| F(x) \|_H,
\| B(x) \|_{\HS( U, H )}
\big\}
\\
&
\leq 
\max \big\{
\tfrac{ 1 }{\sqrt{3}}
\| x \|_{H_\gamma}^2
,
\sqrt{ \vartheta }
\big\}
\leq r(x) \leq ch^{-\delta}
. 
\end{split}
\end{equation}
Furthermore, we observe that
the fact that
$ \forall \, v \in H_\gamma \colon
\sqrt{\theta} + \varepsilon \| v \|_H^2
\leq 
r(v) $
shows
that
for all $ I\in \mathcal{P}_0(\H) $, $ h\in (0,T] $ 
it holds
that
\begin{equation} 
\begin{split}
\label{eq:subset_consistent_Burgers}
 D_h^I
 = \{ x \in P_I(H ) \colon
 r(x) \leq ch^{-\delta}
 \}
 &
  \subseteq 
  \{x\in P_I(H )\colon \sqrt{\vartheta} + \varepsilon \| x \|_H^2
\leq c h^{-\delta}\}
\\
&
\subseteq
\{ v \in H \colon \sqrt{\vartheta} + \varepsilon \| v \|_H^2
\leq c h^{-\delta}\}
.
\end{split}
\end{equation}
In addition, we note that
Lemma~\ref{lemma:estimate_B_1}
ensures that $ \sup_{x\in H_\gamma} \| B(x) \|_{\HS(U,H)}^2 \leq \vartheta < \infty $.
Combining 
\eqref{eq:Burger_coercivity}--\eqref{eq:subset_consistent_Burgers} 
and
Corollary~\ref{Corollary:full_discrete_scheme_convergence}  
(with 
$ H = H $,
$ U = U $,
$ \H = \H $,
$ \mathbb{U} = \mathbb{U} $,
$ T = T $,
$ \gamma = \gamma $,
$ \delta = \delta $, 
$ \lambda = \lambda $,
$ A = A $,
$ \xi = \xi $,
$ F = F $,
$ B = B $,
$ D_h^I = D_h^I $,
$ \vartheta = \vartheta $,
$ b_1 = 0 $,
$ b_2 = 0 $,
$ \varepsilon = \varepsilon $,
$ \varsigma = \delta $,
$ c = c $,
$ Y^{\theta, I, J} = Y^{\theta, I, J} $
for 
$ h \in (0,T] $,
$ \theta \in \varpi_T $,
$ I \in \mathcal{P}_0(\H) $,
$ J \in \mathcal{P}_0(\mathbb{U}) $ 
in the notation of Corollary~\ref{Corollary:full_discrete_scheme_convergence})
hence
completes the proof of Corollary~\ref{corollary:Burgers}. 
\end{proof}
\begin{remark}
Consider the setting of Corollary~\ref{corollary:Burgers}.
Then
the stochastic processes
$ Y^{\theta, I, J} \colon $ $ [0,T] \times \Omega \to P_I(H) $,
$ \theta \in \varpi_T $,
$ I \in \mathcal{P}_0(\H) $,
$ J \in \mathcal{P}_0(\mathbb{U}) $,
are 
space-time-noise discrete numerical
approximation processes
for the  
 stochastic Burgers equation
\begin{equation}
dX_t(x)
=
\big[
\tfrac{\partial^2}{\partial x^2}
X_t(x)
-  
X_t(x)
\cdot
\tfrac{\partial }{\partial x} X_t(x) 
\big]\,dt
+
b(x, X_t(x)) \, 
d ( \sqrt{Q} W )_t(x) , 
\end{equation}
with
$ X_0(x)=\xi(x) $ 
and 
$ X_t(0)=X_t(1)=0 $
for $ t\in [0,T] $, 
$ x\in (0,1) $
(cf., e.g., 
Section~1 in 
Da Prato et al.~\cite{DaPratoDebusscheTemam1994}
and
Section~2 in Hairer \emph{\&} Voss~\cite{HairerVoss2011}).
\end{remark}
\subsection{Stochastic Kuramoto-Sivashinsky equations}
\label{sec:Kuramoto}
\begin{corollary}
\label{corollary:Kuramoto}
Assume the setting in Subsection~\ref{setting:Examples}, 
let 
$ (e_k)_{ k \in \Z } \subseteq \H $,
and
assume 
for all
$ n\in \N $, 
$ k \in \Z $, 
$ v \in H_\gamma $ 
that
$ \eta \in (0,\infty) $,
$ d=1 $, 
$ \gamma \geq \nicefrac{1}{4} $,  
$ e_0 = [ \{ 1 \}_{x \in \D} ]_{\mu_\D, \mathcal{B}(\R)} $,
$ e_n  = 
[ \{ \sqrt{2} \cos(2\pi n x)  \}_{x\in \D} ]_{\mu_{\D}, \mathcal{B}(\R)}
 $,
$ e_{-n}  = 
[ \{  \sqrt{2} \sin(2\pi n x) \}_{x\in \D} ]_{\mu_{\D}, \mathcal{B}(\R)}
 $, 
$ r(v) \geq
 \max\!
\big\{
\sqrt{\vartheta} 
+
\varepsilon \| v \|_H^2  ,
\eta \| v \|_H
$
$
+ 
 5
 \max\{ 1, \eta^{- \gamma} \}  
\| v \|_{  H_\gamma }^2
\big \}
$,
$ \lambda_{e_k} =   4 k^2 \pi^2 - 16 k^4 \pi^4 - \eta $.
Then  
\begin{equation}
  \begin{split}
  \label{eq:KuramotoExpMoments}
  &
  \sup_{ \theta \in \varpi_T}
  \sup_{I \in \mathcal{P}_0(\H)}
  \sup_{J \in \mathcal{P}_0(\mathbb{U})}
  \sup_{t\in [0,T]} 
  \E\!\left[
       \exp\!\left(
         \tfrac{ 
         \varepsilon      
           \| Y^{ \theta, I ,   J }_t \|_H^2
         }
         {
         e^{ 2 
         (
            \eta
           +
           \varepsilon \vartheta ) t }
         }
       \right)
     \right]
<\infty
.
\end{split}
  \end{equation}
\end{corollary}
\begin{proof}[Proof of Corollary~\ref{corollary:Kuramoto}]
First of all, note that the fact that
$ \{ e_l \colon l \in \Z \}
\subseteq \H 
\subseteq H
\subseteq U $
shows that 
\begin{equation} 
\label{eq:isONS2}
\{ e_l \colon l \in \Z \} = \H
\qquad
\text{and}
\qquad H = U 
.
\end{equation}
Hence, we obtain that
\begin{equation} 
\begin{split}
\label{eq:sum_conv2}
\smallsum_{ h \in \H }
| \lambda_h |^{-2 \gamma} 
&
=
\smallsum_{k \in \Z}
| \lambda_{e_k} |^{-2 \gamma}
= 
\smallsum_{k \in \Z}
| 16 k^4 \pi^4
 -4 k^2 \pi^2
+ \eta
|^{-2 \gamma}
\\
&
=
 \eta^{-2 \gamma} 
 +
 2
\smallsum_{k = 1}^\infty
| 16 k^4 \pi^4
- 
4 k^2 \pi^2
+ 
\eta 
|^{-2 \gamma}
\leq
\eta^{-2 \gamma} 
 +
 2
\smallsum_{k = 1}^\infty
| 12 k^4 \pi^4 
+ 
\eta 
|^{-2 \gamma}
\\
&
\leq
\eta^{-2\gamma}
+
2 \smallsum_{k=1}^\infty | 12 k^4 \pi^4 |^{-2\gamma}
=
\eta^{-2\gamma}
+
\tfrac{ 2 }{ | 12 \pi^4 |^{2\gamma } }
\sum_{k=1}^\infty
\tfrac{1}{ k^{8 \gamma} }
\\
&
\leq
\eta^{-2\gamma}
+
\tfrac{2}{ (12 \pi^4 )^{ \nicefrac{1}{2} } }
\smallsum_{k=1}^\infty
\tfrac{1}{k^2}
\leq
\eta^{-2 \gamma}
+
\smallsum_{k=1}^\infty
\tfrac{1}{k^2} 
= 
\eta^{-2\gamma}
+ 
\tfrac{ \pi^2 }{6} 
< \infty.
\end{split}
\end{equation}
Moreover,
note that for all 
$ n \in \N $
it holds that
\begin{equation}
\begin{split}
\label{eq:eig_decrease2}
\| \partial e_n
\|_U
| \lambda_{e_n} |^{- \gamma} 
&
=
\tfrac{
\|
[ 
\{
2 \pi n \sqrt{2} \sin(2 \pi n x)
\}_{x \in (0,1)}
]_{\mu_{(0,1)}, \mathcal{B}(\R)}
\|_U
}
{
| 16 n^4 \pi^4 - 4 n^2 \pi^2 + \eta |^\gamma
}
=
\tfrac{
	2   \pi n
}{
| 16 n^4 \pi^4 - 4 n^2 \pi^2 + \eta |^\gamma
}
\leq
\tfrac{
	2 \pi n
}{
| 12 n^4 \pi^4 + \eta |^\gamma
}
\\
&
\leq
\tfrac{
	2   \pi n
}{
| 12 n^4 \pi^2 |^\gamma
}
\leq
\tfrac{
	2 \pi n
}{
n 
(12)^{ \nicefrac{1}{4} }
\sqrt{ \pi }
}
=
\tfrac{ 2 \sqrt{ \pi } }{ (12)^{ \nicefrac{1}{4} } }
\leq
2
.
\end{split}
\end{equation}
This shows that
for all $ n \in \N $ it holds that
\begin{equation}
\label{eq:eig_decrease2b}
\| \partial e_{-n}
\|_U
| \lambda_{ e_{-n} } |^{- \gamma} 
=
\tfrac{
\|
[ 
\{
2 \pi n \sqrt{2} \cos(2 \pi n x)
\}_{x \in (0,1)}
]_{\mu_{(0,1)}, \mathcal{B}(\R)}
\|_U
}
{
| 16 n^4 \pi^4 - 4 n^2 \pi^2 + \eta |^\gamma
}
=
\tfrac{
	2   \pi n
}{
| 16 n^4 \pi^4 - 4 n^2 \pi^2 + \eta |^\gamma
}
\leq
2.
\end{equation}
Combining~\eqref{eq:isONS2}--\eqref{eq:eig_decrease2b},
the fact that
$ \sup_{h\in \H} \| h
\|_{L^\infty( \mu_\D; \R)}
= \sqrt{2} $,
and
Lemma~\ref{lemma:zero_coercivity}
(with  
$ \rho = \gamma $,
$ u = v $
for $ v \in H_\gamma $
in the notation of 
Lemma~\ref{lemma:zero_coercivity}) 
proves that for all
$ v \in H_\gamma $
it holds that
$ H_\gamma \subseteq [ W^{1,2}(\D, \R) 
\cap L^\infty( \mu_{\D}; \R) ] $
and
\begin{equation}
\begin{split}
\label{eq:F_estimateKuramoto}
\| F(v) \|_H
&
 \leq
 \eta \| v \|_H
 +
 2 \sqrt{2} 
 \big( 
 \eta^{- 2\gamma }  
 +
 \tfrac{\pi^2}{6}
 \big)^{\nicefrac{1}{2}}
 \| v \|_{H_\gamma }^2
 =
 \eta \| v \|_H
 +
 \big(
 8 \eta^{-2 \gamma}
 +
 \tfrac{ 4 \pi^2 }{ 3 }
 \big)^{ \nicefrac{1}{2} }
 \| v \|_{ H_\gamma }^2
 \\
 &
 \leq
 \eta \| v \|_H
 +
 \max \{ 1, \eta^{-\gamma} \}
 \big(
 8 
 +
 \tfrac{ 4 \pi^2 }{ 3 }
 \big)^{ \nicefrac{1}{2} }
 \| v \|_{ H_\gamma }^2
 \leq
 \eta \| v \|_H
 +
 5
 \max\{ 1, \eta^{- \gamma } \}  
 \| v \|_{H_\gamma }^2
.
\end{split}
\end{equation}
Next note that
\begin{equation}
\begin{split}
\H 
\subseteq
\overline{ \mathcal{C}_{P}^\infty ( \D, \R ) }^{ W^{1,2} ( \D, \R) } 
= W^{1,2}_P (\D, \R) 
.
\end{split}
\end{equation}
This,
\eqref{eq:isONS2}--\eqref{eq:eig_decrease2b},
the fact that
$ \sup_{h\in \H} \| h
\|_{L^\infty( \mu_\D; \R)}
= \sqrt{2} $,  
and
Lemma~\ref{lemma:zero_coercivity2}
(with   
$ \rho = \gamma $,
$ u = x $
for  
$ x \in H_\gamma $ 
in the notation of 
Lemma~\ref{lemma:zero_coercivity2})
ensure that for all
$ x \in H_\gamma $
it holds
that
$ H_\gamma \subseteq [ W_P^{1, 2}( \D, \R ) \cap L^\infty( \mu_\D; \R ) ] $
and
\begin{equation}
\begin{split}
&
2 \langle x, F(x) \rangle_H 
=
2 \eta \| x \|_H^2 + \langle x, x \, \partial x \rangle_U
=
2 \eta \| x \|_H^2 + \langle x, R ( x \, \partial x ) \rangle_H
\\
&
=
3 \eta \| x \|_H^2 
-
[ 
\langle x, R( \eta x ) \rangle_H
-
\langle x, R( x \, \partial x ) \rangle_H
]
=
3 \eta \| x \|_H^2 - \langle x, F(x) \rangle_H 
.
\end{split}
\end{equation}
Hence, we obtain that
for all
$ I\in \mathcal{P}_0( \H ) $, 
$ x \in P_I(H) $
it holds that
	\begin{equation}
	\label{eq:coercivity_Kuramoto}
	  \langle x , P_I F(x) \rangle_H =
\langle P_I x, F(x) \rangle_H
=
\langle x, F(x) \rangle_H
=
\eta   \| x \|_H^2.
 \end{equation} 
In the next step
we observe that
\eqref{eq:isONS2}--\eqref{eq:eig_decrease2b},
the fact that
$ \sup_{h\in \H} \| h
\|_{L^\infty( \mu_\D; \R)}
= \sqrt{2} $, 
and
Corollary~\ref{corollary_continuous} 
assure that 
$
F \in \mathcal{C}( H_\gamma, H )
$
and
$
B \in \mathcal{C}( H_\gamma, \HS(U,H) )
$.
This proves that
\begin{equation} 
\label{eq:meas2}
F \in \mathcal{M}( \mathcal{B}(H_\gamma), \mathcal{B}(H) )
\qquad
\text{and}
\qquad
B \in \mathcal{M}( \mathcal{B}(H_\gamma), \mathcal{B}(\HS(U,H)) )
.
\end{equation}
Moreover, \eqref{eq:F_estimateKuramoto} and Lemma~\ref{lemma:estimate_B_1} 
(with
$ v = x $,
$ w = x $
for $ x \in \cup_{h \in (0,T] } \cup_{I \in \mathcal{P}_0(\H)} D_h^I $
in the notation of
Lemma~\ref{lemma:estimate_B_1})
imply that for all 
$ h\in (0,T] $, 
$ I \in \mathcal{P}_0(\H) $, 
$ J \in \mathcal{P}_0( \mathbb{U} ) $,
$ x \in D_h^I $ it holds that
\begin{equation}
\begin{split}
\label{eq:growt_estimate_Kuramoto}
&
\max \{ 
\| P_I F(x) \|_H
,
\| P_I B(x) \hat P_J \|_{\HS( U, H )} 
\}
\leq
\max \{ 
\| F(x) \|_H
,
\| B(x) \|_{\HS( U, H )} 
\}
\\
&
\leq
\max \big\{
\eta \| x \|_H
+ 
 5
 \max\{ 1, \eta^{- \gamma} \}  
\| x \|_{H_\gamma}^2
, 
\sqrt{\vartheta}
\big\}
\leq r(x) \leq ch^{-\delta} 
. 
\end{split}
\end{equation}
Furthermore, we observe that
the fact that
$ \forall \, v \in H_\gamma \colon
\sqrt{\theta} + \varepsilon \| v \|_H^2
\leq 
r(v) $
implies 
that
for all 
$ I\in \mathcal{P}_0(\H) $, 
$ h\in (0,T] $ 
it holds
that
\begin{equation} 
\begin{split}
\label{eq:subset_consistent_Kuramoto}
D_h^I
= \{ x \in P_I( H ) \colon
r(x) \leq ch^{-\delta}
\}
&
\subseteq 
\{x\in P_I( H )\colon \sqrt{\vartheta} + \varepsilon \| x \|_H^2
\leq c h^{-\delta}\}
\\
&
\subseteq
\{v \in H \colon \sqrt{\vartheta} + \varepsilon \| v \|_H^2
\leq c h^{-\delta}\}
.
\end{split}
\end{equation}
In addition, we note that
Lemma~\ref{lemma:estimate_B_1}
ensures that $ \sup_{x\in H_\gamma} \| B(x) \|_{\HS(U,H)}^2 \leq \vartheta < \infty $.
Combining  \eqref{eq:coercivity_Kuramoto}--\eqref{eq:subset_consistent_Kuramoto}  and
Corollary~\ref{Corollary:full_discrete_scheme_convergence}  
(with
$ H = H $,
$ U = U $, 
$ \H = \H $,
$ \mathbb{U} = \mathbb{U} $,
$ T = T $,
$ \gamma = \gamma $,
$ \delta = \delta $,
$ \lambda = \lambda $,
$ A = A $,
$ \xi = \xi $,
$ F = F $,
$ B = B $, 
$ D_h^I = D_h^I $,
$ \vartheta = \vartheta $,
$ b_1 = 0 $,
$ b_2 = \eta $,
$ \varepsilon = \varepsilon $,
$ \varsigma = \delta $,
$ c = c $,
$ Y^{\theta, I, J} = Y^{\theta, I, J} $
for 
$ h \in (0,T] $,
$ \theta \in \varpi_T $,
$ I \in \mathcal{P}_0(\H) $,
$ J \in \mathcal{P}_0(\mathbb{U}) $ 
in the notation of Corollary~\ref{Corollary:full_discrete_scheme_convergence})
hence completes the proof of Corollary~\ref{corollary:Kuramoto}.
\end{proof}
\begin{remark}
Consider the setting of Corollary~\ref{corollary:Kuramoto}.
Then the stochastic processes
$ Y^{\theta, I, J} \colon $ $ [0,T] \times \Omega \to P_I(H) $,
$ \theta \in \varpi_T $,
$ I \in \mathcal{P}_0(\H) $,
$ J \in \mathcal{P}_0(\mathbb{U}) $,
are space-time-noise discrete numerical
approximation processes
for the   
stochastic Kuramoto-Sivashinsky equation 
\begin{equation}
dX_t(x) =
\big[   
- 
\tfrac{\partial^4}{\partial x^4} X_t(x) 
-
\tfrac{\partial^2}{\partial x^2} X_t(x)  
-
X_t (x) 
\cdot
\tfrac{\partial}{\partial x} 
X_t(x) \big]\, dt
+ 
b(x,X_t(x)) 
\,  d ( \sqrt{Q} W )_t(x) ,
\end{equation}
with
$ X_t(0)=X_t(1) $, 
$ X_t'(0)=X_t'(1) $,
$ X_t''(0)=X_t''(1) $, 
$ X_t^{(3)}(0) = X_t^{(3)}(1) $,
and
$ X_0(x)=\xi(x) $
for $ t \in [0,T] $, 
$ x \in (0,1) $
(cf., e.g, 
Duan \& Ervin~\cite{DuanErvin2001}
and
Section~1
in
Hutzenthaler et al.~\cite{HutzenthalerJentzenSalimova2016}).
\end{remark}
\subsection{Two-dimensional stochastic Navier-Stokes equations}
\label{sec:2DNavier}
\begin{corollary}
\label{corollary:2DNavier}
Assume the setting in Subsection~\ref{setting:Examples},
let 
$ (\varphi_k)_{k\in \Z} \subseteq 
\mathcal{C}( (0,1), \R) $, 
$ (\phi_{k,l})_{k,l\in \Z} 
\subseteq \mathcal{C}( \D, \R) $,
$ (e_{i, j, 0})_{i, j \in \Z} \subseteq U $,
$ e_{0,0,1}  \in U $,
and
assume for all 
$ n \in \N $,
$ (k, l) \in \Z^2 \backslash \{(0,0)\} $, 
$ v \in H_\gamma $, 
$ x, y \in (0,1) $
that  
$ \H =  \{ e_{0,0,1} \}
\cup \{ e_{i, j, 0} \colon i, j \in \Z \} $,
$ \eta \in (0,\infty) $,
$ d=2 $,  
$ \gamma > \nicefrac{1}{2} $, 
$ \varphi_0(x)  = 1 $, 
$ \varphi_{n}(x)  = 
 \sqrt{2} \cos( 2 n \pi x ) $,
$ \varphi_{-n} (x)  = 
  \sqrt{2}\sin( 2 n \pi x ) $, 
$ \phi_{k, l}(x,y )  = 
 \varphi_k(x)  \varphi_l(y) $, 
$ e_{0,0,0}  = 
[
\{ 
( 1, 0 ) 
\}_{(x,y)\in \D}
]_{\mu_\D, \mathcal{B}(\R^2)} $,
$ e_{0,0,1} = 
 [
 \{  
(0, 1)
 \}_{(x,y)\in \D}
 ]_{\mu_\D, \mathcal{B}(\R^2)} $,
$ e_{k, l, 0}  = 
\big[
\{
\nicefrac{1}{ \sqrt{k^2 + l^2}  } 
	(
	l \phi_{k,l}(x,y),
	k \phi_{-k, -l}(x,y) 
	)
\}_{(x,y)\in \D}
\big]_{\mu_\D, \mathcal{B}(\R^2)}
$, 
$ \lambda_{e_{k,l,0}} = - \eta - 4 \pi^2 (k^2 + l^2) $,
$ \lambda_{ e_{0,0,0}  } 
=
\lambda_{ e_{0, 0, 1}  } = - \eta $,
$ r(v) \geq
\max\! \big\{ 
\sqrt{\vartheta}
+
\varepsilon \| v \|_H^2,
\eta
\| v \|_H
+ 
6
\big[
\eta^{-2 \gamma}
+ 
\smallsum_{i, j \in \Z }
( \eta + 4 \pi^2 ( i^2 + j^2 ) )^{-2 \gamma} 
\big]^{ \nicefrac{1}{2} }
 \| v \|_{H_\gamma}^2
\big\}
 $. 
Then  
\begin{equation}
  \begin{split}
  \label{eq:expMomentBound2d}
  &
  \sup_{ \theta \in \varpi_T}
  \sup_{ I \in \mathcal{P}_0(\H)}
  \sup_{ J \in \mathcal{P}_0(\mathbb{U})}
  \sup_{t\in [0,T]} 
  \E\!\left[
       \exp\!\left(
         \tfrac{   
         \varepsilon      
           \| Y^{  \theta, I ,  J }_t \|_H^2
         }
         {
         e^{ 2 ( \eta +  \varepsilon \vartheta ) t }
         }
       \right)
     \right]
<\infty
.
\end{split}
  \end{equation}
\end{corollary}
\begin{proof}[Proof of Corollary~\ref{corollary:2DNavier}]
Observe that 
\begin{equation} 
\begin{split}
\label{eq:estimate_sum1}
&
\smallsum\limits_{ h \in \H }
| \lambda_h |^{-2 \gamma} 
= 
\eta^{-2 \gamma}
+ 
\smallsum\limits_{(k, l) \in \Z^2 }
( \eta + 4 \pi^2 ( k^2 + l^2 ) )^{-2 \gamma} 
\\
&
=
\eta^{-2 \gamma}
+
\eta^{-2 \gamma}
+
2
\smallsum\limits_{l = 1}^\infty
( \eta + 4 \pi^2 l^2 )^{-2 \gamma}
+
2
\smallsum\limits_{k = 1}^\infty
( \eta + 4 \pi^2 k^2 )^{-2 \gamma}
\\
&
\quad
+
4
\smallsum\limits_{k=1}^\infty
\smallsum\limits_{l=1}^\infty
( \eta + 4 \pi^2 ( k^2 + l^2 ) )^{-2\gamma}
\\
&
\leq
2 \eta^{-2 \gamma} 
+
4
\smallsum\limits_{k = 1}^\infty
( 4 \pi^2 k^2 )^{-2 \gamma}
+
4
\smallsum\limits_{k=1}^\infty
\smallsum\limits_{l=1}^\infty
( 4 \pi^2 ( k^2 + l^2 ) )^{-2\gamma}
\\
&
=
2 \eta^{-2 \gamma} 
+
4^{1 - 2 \gamma}
\pi^{ - 4 \gamma }
\smallsum\limits_{k = 1}^\infty
k^{-4 \gamma} 
+
4^{ 1 - 2 \gamma }
\pi^{ - 4 \gamma }
\smallsum\limits_{k=1}^\infty
\smallsum\limits_{l=1}^\infty
 ( k^2 + l^2 )^{-2\gamma}
 \\
 &
 \leq
 2 \eta^{-2 \gamma} 
 + 
 \smallsum\limits_{k = 1}^\infty
 k^{-4 \gamma} 
 + 
 \smallsum\limits_{k,l=1}^\infty 
 ( k^2 + l^2 )^{-2\gamma}
 .
\end{split}
\end{equation}
Next note that the fact that
$ \forall \, k \in \N  \colon
k^{-4\gamma} \leq
\int_{k-1}^k x^{-4\gamma} \, dx $
proves that
\begin{equation}
\begin{split}
\label{eq:estimate_sum2}
&
\smallsum\limits_{k=1}^\infty
k^{-4\gamma}
=
1
+
\smallsum\limits_{k=2}^\infty 
k^{-4\gamma}
\leq
1
+
\smallsum\limits_{k=2}^\infty
\smallint\limits_{k-1}^k x^{-4\gamma} \, dx
=
1
+
\smallint\limits_1^\infty
x^{-4 \gamma} \, dx
=
1
+
\tfrac{ 1 }{ 4 \gamma - 1 }
.
\end{split}
\end{equation}
In addition, 
we observe that
the fact that
$ \forall \, k, l \in \N \colon
( k^2 + l^2 )^{ - 2 \gamma }
=
\int_{k-1}^k
\int_{l-1}^l
( k^2 + l^2 )^{-2 \gamma}
\, dx \, dy
\leq
\int_{k-1}^k 
\int_{l-1}^l
(y^2 + x^2)^{-2 \gamma} 
\, dx \, dy
=
\int_{k-1}^k \int_{l-1}^l (x^2 + y^2 )^{ - 2 \gamma }
\, dx \, dy $
proves that
\begin{equation}
\begin{split}
\label{eq:estimate_sum3}
&
\smallsum\limits_{k, l=1}^\infty
 ( k^2 + l^2 )^{-2\gamma}
 = 
 \smallsum\limits_{k=1}^\infty 
 ( k^2 + 1 )^{-2\gamma}
 + 
\smallsum\limits_{l=2}^\infty
 ( 1 + l^2 )^{-2\gamma}
 +
 \smallsum\limits_{k=2}^\infty
\smallsum\limits_{l=2}^\infty
 ( k^2 + l^2 )^{-2\gamma}
 \\
 &
 \leq
 2
 \smallsum\limits_{k=1}^\infty k^{-4\gamma}
 +
 \smallsum\limits_{k=2}^\infty
\smallsum\limits_{l=2}^\infty
\int_{k-1}^k \int_{l-1}^l (x^2 + y^2 )^{ - 2 \gamma }
\, dx \, dy
\\
&
=
2
 \smallsum\limits_{k=1}^\infty k^{-4\gamma}
 + 
\int_1^\infty \int_1^\infty (x^2 + y^2 )^{ - 2 \gamma } \, dx \, dy
=
 2
 \smallsum\limits_{k=1}^\infty k^{-4\gamma}
 +
 \int_0^{2\pi}
 \int_1^\infty
 s^{1-4\gamma}
 \, ds \, du
\\
&
=
 2 \pi
 \smallint\nolimits_1^\infty
 s^{1-4\gamma} \, ds
 +
 2
 \smallsum\limits_{k=1}^\infty k^{-4\gamma}
 =
 \tfrac{ 2 \pi } {  4 \gamma - 2  }
 +
 2
 \smallsum\limits_{k=1}^\infty k^{-4\gamma}
 =
 \tfrac{ \pi } {  2 \gamma - 1  }
 +
 2
 \smallsum\limits_{k=1}^\infty k^{-4\gamma}
.
\end{split}
\end{equation}
Combining~\eqref{eq:estimate_sum1}--\eqref{eq:estimate_sum3}
proves that
\begin{equation}
\begin{split}
\label{eq:sum_conv3} 
\smallsum\limits_{h \in \H}
| \lambda_h |^{-2\gamma}
<   \infty .
\end{split}
\end{equation}
Moreover,
note that for all 
$ (k, l ) \in \Z^2 \backslash \{ (0,0) \} $
it holds that
\begin{equation}
\begin{split}
\label{eq:eig_decrease3}
&
\| \partial_1 e_{k, l, 0}
\|_U
| \lambda_{e_{k, l, 0}} |^{- \gamma } 
\\
&
=
\|
 [
\{
( k^2 + l^2 )^{-\nicefrac{1}{2}}
	(
	l \, \tfrac{\partial}{\partial x} \phi_{k, l}(x,y),
	k \, \tfrac{\partial}{\partial x} \phi_{-k, -l}(x,y) 
	)
\}_{(x,y)\in \D}
 ]_{\mu_\D, \mathcal{B}(\R^2)}
\|_U
| \lambda_{e_{k, l, 0}} |^{- \gamma } 
\\
&
=
\|
 [
\{
( k^2 + l^2 )^{-\nicefrac{1}{2}}
	(
	- 2 \pi kl \phi_{-k, l}(x,y),
	  2 \pi k^2 \phi_{k, -l}(x,y) 
	)
\}_{(x,y)\in \D}
 ]_{\mu_\D, \mathcal{B}(\R^2)}
\|_U
| \lambda_{e_{k, l, 0}} |^{- \gamma } 
\\
&
= 
( k^2 + l^2 )^{-\nicefrac{1}{2}}
2\pi k
  \sqrt{ k^2 + l^2 }
| \lambda_{e_{k, l, 0}} |^{- \gamma } 
=
	2 \pi  k
| 4 \pi^2 ( k^2 + l^2 ) + \eta |^{-\gamma} 
\\
&
\leq
\tfrac{ 2 \pi k }{ | 4 \pi^2 ( k^2 + l^2 ) + \eta |^{ \nicefrac{1}{2} } }
\leq
\tfrac{ 2 \pi k }{ 2 \pi ( k^2 + l^2 )^{ \nicefrac{1}{2} } }
\leq
1  
\end{split}
\end{equation}
and
\begin{equation}
\begin{split}
\label{eq:eig_decrease3b}
&
\| \partial_2 e_{k, l, 0}
\|_U
| \lambda_{ e_{k, l, 0} } |^{- \gamma } 
\\
&
=
\|
 [
\{
( k^2 + l^2 )^{-\nicefrac{1}{2}}
	(
	l \, \tfrac{\partial}{\partial y} \phi_{k, l}(x,y),
	k \, \tfrac{\partial}{\partial y} \phi_{-k, -l}(x,y) 
	)
\}_{(x,y)\in \D}
 ]_{\mu_\D, \mathcal{B}(\R^2)}
\|_U
| \lambda_{ e_{k, l, 0} } |^{- \gamma } 
\\
&
=
\|
 [
\{
( k^2 + l^2 )^{-\nicefrac{1}{2}}
	(
	 -  2 \pi l^2 \phi_{k, -l}(x,y),
	  2 \pi k l \phi_{-k, l}(x,y) 
	)
\}_{(x,y)\in \D}
 ]_{\mu_\D, \mathcal{B}(\R^2)}
\|_U
| \lambda_{e_{k, l, 0}} |^{- \gamma } 
\\
&
= 
( k^2 + l^2 )^{-\nicefrac{1}{2}}
2\pi l 
  \sqrt{ k^2 + l^2 }
| \lambda_{e_{k, l, 0}} |^{- \gamma } 
= 
	2 \pi  l
| 4 \pi^2( k^2 + l^2 ) + \eta |^{-\gamma} 
\\
&
\leq 
\tfrac{ 2 \pi l }
{ | 4 \pi^2 ( k^2 + l^2 ) + \eta |^{\nicefrac{1}{2}} }
\leq
\tfrac{ 2 \pi l }
{ 2 \pi ( k^2 + l^2 )^{\nicefrac{1}{2}} }
\leq
1 
.
\end{split}
\end{equation}
Furthermore, observe that for all 
$ ( k, l ) \in \Z^2 \backslash \{ ( 0, 0 ) \} $
it holds that
\begin{equation}
\begin{split}
\| e_{k, l, 0} \|_{L^\infty( \mu_\D; \R^2)}
&
=
\tfrac{ 1 } { \sqrt{ k^2 + l^2 } }
\| ( l \phi_{k, l }, k \phi_{-k,- l} ) \|_{L^\infty( \mu_\D; \R^2 )}
\\
&
=
\tfrac{ 1 } { \sqrt{ k^2 + l^2 } }
\| l^2  | \phi_{k, l } ( \cdot ) |^2
+
 k^2 | \phi_{-k, -l} ( \cdot ) |^2  \|_{L^\infty( \mu_\D; \R )}^{ \nicefrac{1}{2} }
\\
&
=
\tfrac{ 1 }{ \sqrt{ k^2 + l^2 } }
\sup_{ x_1, x_2 \in (0, 1) }
\sqrt{
l^2 
| \varphi_k (x_1) |^2
| \varphi_l (x_2) |^2
+
k^2
| \varphi_{-k} (x_1) |^2
| \varphi_{-l} (x_2) |^2
}
\\
&
\leq
\tfrac{ 1 }{ \sqrt{ k^2 + l^2 } }
\sqrt{ l^2 4 + k^2 4 }
=
\tfrac{ 2 \sqrt{ l^2 + k^2} } { \sqrt{ k^2 + l^2} }
=
2
.
\end{split}
\end{equation}
Hence, we obtain that
\begin{equation}
\begin{split}
\label{eq:estimate_Linfty}
&
\sup\nolimits_{h \in \H} \| h \|_{L^\infty(\mu_\D; \R^2)} 
\\
&
=
\max \big\{
\| e_{0,0,0} \|_{L^\infty(\mu_\D; \R^2)} 
,
\| e_{0,0,1} \|_{L^\infty(\mu_\D; \R^2)} 
,
\sup\nolimits_{(k, l)\in \Z^2 \backslash \{ (0,0) \} }
\| e_{k, l, 0} \|_{L^\infty(\mu_\D; \R^2)} 
\big \}
\\
&
\leq
\max \{ 1, 1, 2 \}
=
2
.
\end{split}
\end{equation}
Combining~\eqref{eq:estimate_sum1}, 
\eqref{eq:sum_conv3}, \eqref{eq:eig_decrease3}, \eqref{eq:eig_decrease3b},
\eqref{eq:estimate_Linfty},
and 
Lemma~\ref{lemma:zero_coercivity}
(with  
$ \rho = \gamma $,
$ u = v $
for
$ v \in H_\gamma $ 
in the notation of 
Lemma~\ref{lemma:zero_coercivity}) 
proves that for all
$ v \in H_\gamma $ 
it holds
that
$ H_\gamma \subseteq 
[ W^{1,2}(\D, \R^2) 
\cap L^\infty( \mu_{\D}; \R^2) ] $
and 
\begin{equation}
\begin{split}
\label{eq:F_B_estimateNavier}
&
\| F(v) \|_H
\leq 
\eta 
\| v \|_H
+  
6
\big[
\eta^{-2 \gamma}
+ 
\smallsum_{(k, l) \in \Z^2 }
( \eta + 4 \pi^2 ( k^2 + l^2 ) )^{-2 \gamma} 
\big]^{ \nicefrac{1}{2} }
\| v \|_{H_\gamma }^2
.
\end{split}
\end{equation}
Next note that
\begin{equation}
\H 
\subseteq 
\overline{ \mathcal{C}_P^\infty ( \D; \R^2 ) }^{ W^{1,2}( \D, \R^2) }
=
W_P^{1, 2}( \D, \R ) 
.
\end{equation}
This,
\eqref{eq:sum_conv3}, \eqref{eq:eig_decrease3}, \eqref{eq:eig_decrease3b},
\eqref{eq:estimate_Linfty},  
and
Lemma~\ref{lemma:zero_coercivity2}
(with   
$ \rho = \gamma $,
$ u = u $
for  
$ u =(u_1, u_2) \in H_\gamma $ 
in the notation of 
Lemma~\ref{lemma:zero_coercivity2}) 
ensure that for all
$ u = ( u_1, u_2 ) \in H_\gamma $ 
it holds that
$
H_\gamma \subseteq [ W^{1,2}_P( \D, \R^2) \cap L^\infty( \mu_\D; \R^2 ) ] $
and
\begin{equation}
\label{eq:coerc12}
2 \langle u, F(u) \rangle_H
=
2
\eta \| u \|_H^2 
+
\langle ( u_1 )^2 + ( u_2 )^2, 
\partial_1 u_1  + \partial_2 u_2 \rangle_{ L^2( \mu_\D; \R ) }
.
\end{equation}
In addition, note that for all
$ (k, l) \in \Z^2 $,
$ x, y \in (0,1) $
it holds that
\begin{equation}
\begin{split}
\big|
l \tfrac{ \partial }{ \partial x } \phi_{k,  l} (x,y)
+
k \tfrac{ \partial }{ \partial y } \phi_{-k, -l}(x,y)
\big|
= 
\big|
-
2 \pi k l \phi_{-k, l }(x,y)
+
2 \pi k l \phi_{-k, l }(x,y)
\big|
=
0.
\end{split}
\end{equation}
This assures that for all
$ h = (h_1, h_2) \in \H $
it holds that
\begin{equation} 
\label{eq:div_free_basis}
\partial_1 h_1 + \partial_2 h_2 = [ \{0\}_{x \in \D} ]_{ \mu_\D, \mathcal{B}(\R) } 
.
\end{equation}
Moreover, note that~\eqref{eq:eig_decrease3}, \eqref{eq:eig_decrease3b},
and
item~\eqref{item:conv_partial_L2} in 
Lemma~\ref{lemma:weak_exists}
(with 
$ \rho = \gamma $,
$ u = u $,
$ j = j $
for  
$ u \in H_\gamma $,
$ j \in \{ 1, 2 \} $
in the notation of
Lemma~\ref{lemma:weak_exists})
prove that for all
$ u \in H_\gamma $,
$ j \in \{1, 2 \} $
it holds
that
\begin{equation}
\begin{split}
\label{eq:limite_necessary}
&
\limsup\nolimits_{\mathcal{P}_0( \H) \ni I \to \H}
  \| 
\partial_j u 
- 
\smallsum_{h \in I} \langle h, u \rangle_H \partial_j h 
  \|_{ L^2( \mu_\D; \R^2 ) } 
=
0
.
\end{split}
\end{equation}
This implies that for all 
$ u = (u_1, u_2) \in H_\gamma $,
$ j \in \{ 1, 2 \} $
it holds that
\begin{equation}
\begin{split}
\label{eq:limite_necessary_component}
&
\limsup\nolimits_{\mathcal{P}_0( \H) \ni I \to \H}
  \| 
\partial_j u_j 
- 
\smallsum_{h = (h_1, h_2) \in I} \langle h, u \rangle_H \partial_j h_j 
  \|_{ L^2( \mu_\D; \R ) } 
=
0
.
\end{split}
\end{equation}
Next note that~\eqref{eq:div_free_basis}
ensures that  
 for all 
 $ u = (u_1, u_2) \in W^{1,2} (\D, \R^2) $,
 $ I \in \mathcal{P}_0(\H) $
 it holds that
 \begin{equation}
 \begin{split}
 \label{eq:crucial_triangle}
 &
 \| 
 \partial_1 u_1
 +
 \partial_2 u_2 
 \|_{L^2 ( \mu_D; \R) }
 \\
 &
 =
  \| 
\partial_1 u_1
+
\partial_2 u_2 
- 
\smallsum_{h = (h_1, h_2) \in I} 
\langle h, u \rangle_H
( \partial_1 h_1
+
\partial_2 h_2
)
  \|_{ L^2( \mu_\D; \R ) }
 \\
 &
 \leq
  \| 
\partial_1 u_1  
- 
\smallsum_{h = (h_1, h_2) \in I} 
\langle h, u \rangle_H
 \partial_1 h_1 
  \|_{ L^2( \mu_\D; \R ) }
  \\
  &
  \quad
  +
   \|  
\partial_2 u_2 
- 
\smallsum_{h = (h_1, h_2) \in I} 
\langle h, u \rangle_H 
\partial_2 h_2 
  \|_{ L^2( \mu_\D; \R ) }
  .
 \end{split}
 \end{equation}
 Combining~\eqref{eq:limite_necessary_component}
 with the fact that
 $ H_\gamma \subseteq W^{1,2}(\D, \R^2) $
  hence
shows that 
 for all $ u = (u_1, u_2) \in H_\gamma $ 
 it holds that
 \begin{equation}
 \begin{split}
 &
 \| 
 \partial_1 u_1
 +
 \partial_2 u_2 
 \|_{L^2 ( \mu_D; \R) }
 \\
 &
 =
 \limsup\nolimits_{ \mathcal{P}_0(\H) \ni I \to \H}
  \| 
\partial_1 u_1
+
\partial_2 u_2 
- 
\smallsum_{h = (h_1, h_2) \in I} 
\langle h, u \rangle_H
(
 \partial_1 h_1
+
 \partial_2 h_2
)
  \|_{ L^2( \mu_\D; \R ) }
   \\
   &
   \leq
   \limsup\nolimits_{ \mathcal{P}_0(\H) \ni I \to \H}
   \| 
   \partial_1 u_1 
   - 
   \smallsum_{h = (h_1, h_2) \in I} 
   \langle h, u \rangle_H 
   \partial_1 h_1 
   \|_{ L^2( \mu_\D; \R ) }
   \\
   &
   \quad
   +
   \limsup\nolimits_{ \mathcal{P}_0(\H) \ni I \to \H}
   \|  
   \partial_2 u_2 
   - 
   \smallsum_{h = (h_1, h_2) \in I} 
   \langle h, u \rangle_H 
   \partial_2 h_2 
   \|_{ L^2( \mu_\D; \R ) }
  =
  0.
 \end{split}
 \end{equation}
This assures that
 for all $ u = (u_1, u_2) \in H_\gamma $ 
 it holds that
 \begin{equation}
 \partial_1 u_1 + \partial_2 u_2 = [ \{ 0 \}_{x\in \D} ]_{\mu_{ \D }, \mathcal{B}(\R)} .
 \end{equation}
Equation~\eqref{eq:coerc12} therefore
proves that for all
$ I\in \mathcal{P}_0( \H ) $, 
$ x \in P_I(H) $
it holds that
	\begin{equation} 
\label{eq:coercivity_Navier}
\langle x , P_I F(x) \rangle_H =
\langle P_I x, F(x) \rangle_H
=
\langle x, F(x) \rangle_H
=
\eta \| x \|_H^2 
.
\end{equation} 
In the next step
we observe that
\eqref{eq:sum_conv3}, \eqref{eq:eig_decrease3}, \eqref{eq:eig_decrease3b},
\eqref{eq:estimate_Linfty},
and
Corollary~\ref{corollary_continuous} 
assure that 
$
F \in \mathcal{C}( H_\gamma, H )
$
and
$
B \in \mathcal{C}( H_\gamma, \HS(U,H) )
$.
This proves that
\begin{equation} 
\label{eq:meas3}
F \in \mathcal{M}( \mathcal{B}(H_\gamma), \mathcal{B}(H) )
\qquad
\text{and}
\qquad
B \in \mathcal{M}( \mathcal{B}(H_\gamma), \mathcal{B}(\HS(U,H)) )
.
\end{equation}
Moreover, note that~\eqref{eq:F_B_estimateNavier} and Lemma~\ref{lemma:estimate_B_1}
imply that for all 
$ h\in (0,T] $, 
$ I \in \mathcal{P}_0(\H) $, 
$ J \in \mathcal{P}_0( \mathbb{U} ) $,
$ x \in D_h^I $ 
it holds that
\begin{equation}
\begin{split}
\label{eq:growt_estimate_Navier}
&
\max\!\big\{
\| P_I F(x) \|_H
,
\| P_I B(x) \hat P_J \|_{\HS( U, H )}
\big\}
\leq
\max\!\big\{
\| F(x) \|_H
,
\| B(x) \|_{\HS( U, H )}
\big\}
\\
&
\leq
\max \! \big \{
\eta \| x \|_H
+  
6
\big[
\eta^{-2 \gamma}
+ 
\smallsum_{(k, l) \in \Z^2 }
( \eta + 4 \pi^2 ( k^2 + l^2 ) )^{-2 \gamma} 
\big]^{ \nicefrac{1}{2} }
\| x \|_{H_\gamma }^2
, 
\sqrt{\vartheta}
\big \}
\\
&
\leq r(x) \leq ch^{-\delta} 
. 
\end{split}
\end{equation}
Furthermore, we observe that the fact that 
$ \forall \, v \in H_\gamma \colon
\sqrt{\theta} + \varepsilon \| v \|_H^2
\leq 
r(v) $
implies
that 
for all 
$ I \in \mathcal{P}_0(\H) $, 
$ h \in (0,T] $ 
it holds
that
\begin{equation} 
\begin{split}
\label{eq:subset_consistent_Navier}
D_h^I
= \{ x \in P_I( H ) \colon
r(x) \leq ch^{-\delta}
\}
&
\subseteq \{x \in P_I( H )\colon \sqrt{\vartheta} + \varepsilon \| x \|_H^2
\leq c h^{-\delta}\}
\\
&
\subseteq \{v \in H \colon \sqrt{\vartheta} + \varepsilon \| v \|_H^2
\leq c h^{-\delta}\}
.
\end{split}
\end{equation}
In addition, we note that
Lemma~\ref{lemma:estimate_B_1}
ensures that $ \sup_{x\in H_\gamma} \| B(x) \|_{\HS(U,H)}^2 \leq \vartheta < \infty $.
Combining 
\eqref{eq:coercivity_Navier}--\eqref{eq:subset_consistent_Navier}  
and
Corollary~\ref{Corollary:full_discrete_scheme_convergence}  
(with 
$ H = H $,
$ U = U $,
$ \H = \H $,
$ \mathbb{U} = \mathbb{U} $,
$ T = T $,
$ \gamma = \gamma $,
$ \delta = \delta $, 
$ \lambda = \lambda $,
$ A = A $,
$ \xi = \xi $,
$ F = F $,
$ B = B $,
$ D_h^I = D_h^I $,
$ \vartheta = \vartheta $,
$ b_1 = 0 $,
$ b_2 = \eta $,
$ \varepsilon = \varepsilon $,
$ \varsigma = \delta $,
$ c = c $,
$ Y^{\theta, I, J} = Y^{\theta, I, J} $
for 
$ h \in (0,T] $,
$ \theta \in \varpi_T $,
$ I \in \mathcal{P}_0(\H) $,
$ J \in \mathcal{P}_0(\mathbb{U}) $
in the notation of Corollary~\ref{Corollary:full_discrete_scheme_convergence})
hence
completes the proof of Corollary~\ref{corollary:2DNavier}.
 \end{proof} 
\begin{remark}
Consider the setting of Corollary~\ref{corollary:2DNavier}.
Then the stochastic processes
$ Y^{\theta, I, J} \colon $ $ [0,T] \times \Omega \to P_I(H) $,
$ \theta \in \varpi_T $,
$ I \in \mathcal{P}_0(\H) $,
$ J \in \mathcal{P}_0(\mathbb{U}) $,
are space-time-noise discrete numerical
approximation processes
for the  
two-dimensional
stochastic Navier-Stokes equations
\begin{equation} 
d X_t(x) 
=
\big[
( \tfrac{ \partial^2 }{ \partial x_1^2 } + \tfrac{ \partial^2 }{ \partial x_2^2 } ) X_t(x)
+
( R ( ( \tfrac{\partial}{\partial x} X_t ) \cdot X_t ) ) (x)
\big]\, dt
+
b( x, X_t(x) )\, d (\sqrt{Q} W  )_t(x)
\end{equation}
with periodic boundary conditions,
$ ( \operatorname{div} X_t )(x) = 0 $, and 
$ X_0(x) = \xi(x) $ for
$ t \in [0,T] $,
$ x = (x_1, x_2) \in (0,1)^2 $
(cf., e.g., Section~2 in
Da Prato {\emph{\&}} Debussche~\cite{DaPratoDebussche2002},
Carelli {\emph{\&}} Prohl~\cite{CarelliProhl2012},
Carelli et al.~\cite{CarelliHausenblasProhl2012},
Brze{\'z}niak et al.~\cite{BrzezniakCarelliProhl2013},
and
Bessaih et al.~\cite{BessaihBrzezniakMillet2014}).
\end{remark}
\subsection*{Acknowledgements}
We gratefully acknowledge Zdzis\l{}aw Brze{\'z}niak
for several useful comments that helped to improve
the presentation of the results.
This project has been supported through the SNSF-Research project $ 200021\_156603 $ ``Numerical 
approximations of nonlinear stochastic ordinary and partial differential equations''.
\bibliographystyle{acm} 
\bibliography{bibfile}
\end{document}